\documentclass{article}

\usepackage[T1]{fontenc}
\usepackage{lmodern}
\usepackage{vmargin}
\setmarginsrb{28mm}{25mm}{28mm}{25mm}{0pt}{0mm}{0pt}{0mm}
\usepackage{authblk}
\usepackage{amssymb}
\usepackage{amsmath}
\usepackage{amsthm}
\usepackage{graphicx}
\usepackage{tikz}
\usepackage{bbm}
\usepackage{subcaption}
\usepackage[ruled,linesnumbered]{algorithm2e}
\usepackage{mathtools}
\mathtoolsset{showonlyrefs}
\usepackage{lipsum}
\usepackage[title,titletoc]{appendix}
\usepackage{booktabs}
\usepackage{here}
\usepackage[colorlinks=true,linkcolor=blue,citecolor=red]{hyperref}
\usepackage[numbers]{natbib}
\usepackage{enumerate}
\usepackage{enumitem}
\usepackage{color}
\usepackage[normalem]{ulem}
\usepackage{dsfont}
\usepackage{bm}
\usepackage{nicefrac}
\usepackage{todonotes}

\newcommand{\N}{{\mathbb{N}}} %
\newcommand{\R}{{\mathbb{R}}} %

\DeclareSymbolFont{bbold}{U}{bbold}{m}{n}
\DeclareSymbolFontAlphabet{\mathbbold}{bbold}
\newcommand{\ind}{{\mathbbold{1}}}

\newcommand{\calE}{{\mathcal{E}}}

\newcommand{\calH}{{\mathcal{H}}}

\newcommand{\calO}{{\mathcal{O}}}

\newcommand{\calS}{{\mathcal{S}}}

\newcommand{\calV}{{\mathcal{V}}}
\newcommand{\calW}{{\mathcal{W}}}
\newcommand{\calX}{{\mathcal{X}}}

\newcommand{\dif}{d} %

\newcommand{\abs}[1]{\left\lvert#1\right\rvert}

\newcommand{\norm}[1]{\left\lVert#1\right\rVert}

\newcommand{\inp}[2]{\left \langle #1, #2 \right \rangle}

\newcommand{\Exp}[1]{\mathbb{E}\left[#1\right]}
\newcommand{\Var}[1]{\operatorname{Var}\left(#1\right)}
\newcommand{\Prob}[1]{\mathbb{P}\left(#1\right)}

\newcommand{\err}{\ensuremath{e}}

\newcommand{\yb}{\ensuremath{\boldsymbol{y}}}

\newcommand{\zb}{\ensuremath{\boldsymbol{z}}}

\newcommand{\tol}{\ensuremath{\mathrm{tol}}} %
\newcommand{\pde}{\ensuremath{\mathrm{PDE}}}

\newcommand{\adj}{\ensuremath{\ast}}

\newtheorem{theorem}{Theorem}[section]
\newtheorem{proposition}{Proposition}[section]
\newtheorem{lemma}{Lemma}[section]
\newtheorem{corollary}{Corollary}[section]
\newtheorem{remark}{Remark}[section]
\newtheorem{assumption}{Assumption}
\newtheorem{definition}{Definition}[section]

\SetKwInput{Input}{Input}
\SetKwInput{Output}{Output}

\definecolor{mygreen}{rgb}{0,0.4,0}

\title{On the sample complexity of the active subspace method}
\author[$\ast$]{F. Nobile}
\author[$\ast$]{M. Raviola}
\author[$\dag$]{R. Tempone}
\affil[$\ast$]{{\footnotesize CSQI Chair, \'{E}cole Polytechnique F\'{e}d\'{e}rale de Lausanne, Switzerland}}
\affil[$\dag$]{{\footnotesize KAUST, Saudi Arabia}}
\date{\today}

\begin{document}

\maketitle
\begin{abstract}
    Active subspaces identify low-dimensional linear structure in high-dimensional parameter-to-output maps by estimating the dominant eigenspace of a gradient covariance operator.
    In practice this covariance is replaced by a Monte Carlo estimator built from a limited number of gradient evaluations.
    Classical analyses based on controlling the covariance error in operator norm lead to sample-complexity estimates that can be substantially more pessimistic than the sampling rules commonly used in computations.
    This paper studies the empirical active subspace method directly in the projection-error metric relevant for ridge approximation.
    We derive non-asymptotic quasi-optimality bounds governed by a {regularized} inverse Christoffel function associated with the gradient field.
    Under a bounded-gradient assumption, the resulting estimates already improve the sample-complexity estimates obtained from operator-norm covariance bounds.
    We then show that additional smoothness of the gradient map, expressed through membership in a reproducing kernel Hilbert space, yields sharper coherence estimates and motivates tractable importance sampling from kernel diagonal measures.
    Furthermore, {the} same smoothness {assumption yields} a priori decay bounds for the population
    active subspace tail energy, which can be combined with {our} finite-sample {estimate} to
    prescribe rank, {regularization} scale, and sample size, allowing to fully characterize the a priori sample complexity.
    The abstract assumptions are verified for lognormal Gaussian and affine uniform parametric elliptic PDEs using weighted {summability of} Hermite and Legendre {series expansions}.
    \medskip


\end{abstract}

\clearpage
\thispagestyle{empty}
\tableofcontents

\clearpage
\pagenumbering{arabic}
\setcounter{page}{1}

{
\section{Introduction}
\label{sec:introduction}

Parametric and random partial differential equations (PDEs) arise naturally in uncertainty quantification when coefficients, loads, geometries, or boundary data are not known exactly.
For a scalar quantity of interest, such a model induces a parameter-to-output map $\yb\mapsto f(\yb)$, where $\yb$ ranges over a parameter domain $\Gamma$ endowed with a probability measure $\mu$.
The parameter dimension $d$ may be large, or even countably infinite when a distributed input is represented through a countable expansion.
The efficient approximation of such maps can therefore be hindered by the ambient dimension of the parameter space.
This has motivated methods that exploit low-dimensional or sparse structure in parametric models, including sparse polynomial approximation \cite{cohen2015approximation,bachmayr2017sparsei,bachmayr2017sparseii}, basis adaptation for polynomial chaos expansions \cite{tsilifis2017reduced,tsilifis2019compressive}, and active subspaces \cite{constantine2015active}.

This work studies the active subspace method, which seeks a low-dimensional linear projection of the input variables that captures the dominant variation of $f$.
We work on finite or countably infinite product parameter domains equipped with probability measures for which a suitable Sobolev calculus is available.
In its standard form, the method is based on the gradient covariance operator
\begin{equation*}
  C
  =
  \int_\Gamma
  \nabla f(\yb)\otimes\nabla f(\yb)
  \dif\mu(\yb).
\end{equation*}
The dominant eigenspaces of $C$ identify directions along which $f$ has large average squared directional derivative.
For a finite-rank orthogonal projector $\Pi$, the population gradient projection error is defined by
\begin{equation*}
  \err(\Pi)^2
  = \mathbb{E}\left[\norm{(I-\Pi) \nabla f(Y)}^2_{\ell^2}\right]
  =
  \operatorname{tr}\left((I-\Pi)C\right),
\end{equation*}
where $Y$ is a random variable with law $\mu$.
A minimizer of this error over all rank-$r$ projectors is the spectral projector $\Pi_r$ associated with the $r$ largest eigenvalues of $C$ and its range is the rank-$r$ active subspace of $f$.
Then $\err(\Pi_r)$ is the rank-$r$ population active subspace error, while $\err(\Pi_r)^2$ is equivalently the active subspace tail energy.
Under suitable conditional Poincar\'e inequalities, the active subspace error also controls the error of the ideal ridge approximation of $f$, that is the best approximation by a function of the projected variables \cite{constantine2017near,zahm2020gradient}.

In practice, the gradient covariance operator $C$ and its spectral decomposition cannot be computed exactly.
One therefore replaces $C$ by a weighted empirical covariance $\hat C$ constructed from $M$ independent gradient samples and computes the corresponding rank-$r$ empirical spectral projector $\hat\Pi_r$.
We allow the samples to be drawn from a probability measure $\rho$ with $\mu\ll\rho$, using the importance weight $w=\dif\mu/\dif\rho$.

Given a tolerance $\tol>0$, this formulation leads to two questions:
\begin{enumerate}
  \item the approximation question: whether the population active subspace error can be made smaller than $\tol$ using a moderate rank $r$;
  \item the sampling question: after choosing such a rank, how many gradient samples are sufficient for $\hat\Pi_r$ to have population error comparable to $\err(\Pi_r)$.
\end{enumerate}
We first address the sampling question for a fixed rank and sampling measure.
We then use smoothness assumptions for the gradient field $\nabla f$ to control both the rank required to reach $\tol$ and the corresponding sufficient sampling condition.

\subsection{Active subspace recovery}

Existing analyses of empirical active subspaces estimate the covariance operator or its dominant eigenspaces using concentration and perturbation arguments \cite{constantine_gleich_2014_computing,holodnak2018probabilistic,lam2020multifidelity}.
To obtain an active subspace error of order $\tol$, this route requires covariance accuracy of order $\tol^2/r$.
Under standard bounded-gradient assumptions, covariance concentration then gives a sample-size dependence of order $r^2\tol^{-4}$, in addition to covariance-scale and intrinsic-dimension factors.
This guarantee is considerably more conservative than the practical rule of thumb that approximately $r\log d$ gradient samples often suffice \cite{constantine2015active}.

We instead analyze the population error of the empirical active subspace directly.
Our main sampling result is a non-asymptotic quasi-optimality estimate of the form
\begin{equation*}
  \err(\hat\Pi_r)^2
  \leq
  c\,\err(\Pi_r)^2
  +
  R_{r,\lambda},
\end{equation*}
where $\lambda>0$ is a regularization scale.
  {To distinguish covariance regularization from ridge approximation, we use the adjective \emph{regularized} for quantities defined through the shift $C+\lambda I_{\mathcal R}$.}
  {The sufficient sampling condition is governed by a regularized inverse Christoffel coherence, which measures the largest regularized leverage of a sampled gradient.}

Suppose that $r$ is chosen so that $\err(\Pi_r)\leq\tol$ and set $\lambda=\tol^2/r$, making the remainder $R_{r,\lambda}$ of order $\tol^2$.
Sampling directly from $\mu$ under a uniform gradient bound gives the sufficient condition
\begin{equation*}
  M
  \gtrsim
  r\tol^{-2}\log(1+r),
\end{equation*}
up to gradient-dependent constants and probability factors.
Minimizing this coherence over all sampling measures gives an oracle importance measure with a sufficient sample size of order $r\log r$.
Its density, however, depends on the unknown gradient covariance.
To obtain implementable importance measures with improved tolerance dependence, we assume that scalar projections of $\nabla f$ belong to a reproducing kernel Hilbert space (RKHS).
If the eigenvalues of the corresponding kernel integral operator are $\ell^\theta$-summable for some $\theta\in(0,1]$, kernel diagonal sampling gives
\begin{equation*}
  M
  \gtrsim
  r^\theta\tol^{-2\theta}\log(1+r),
\end{equation*}
up to regularity-dependent constants and probability factors.
Thus, for a fixed target rank, smoothness improves the dependence on $\tol$ from $\tol^{-2}$ to $\tol^{-2\theta}$.

Related summability assumptions on the same RKHS regularity also address the approximation question by giving algebraic decay
\begin{equation*}
  \err(\Pi_r)^2
  \lesssim
  r^{-q},
\end{equation*}
where $q>0$ depends on the RKHS spectral summability and additional summability across parameter directions, resulting in an overall sample complexity of $M \gtrsim \tol^{-2\theta(1+1/q)}$, up to logarithmic factors.

We verify these assumptions for linear quantities of interest associated with lognormal Gaussian and affine uniform elliptic PDE models \cite{bachmayr2017sparseii,bachmayr2017sparsei}.
In the lognormal Gaussian case, our analysis leads to an a priori ridge-error bound with twice the decay exponent of the ridge approximation based on the best-ranked coordinates according to the coefficients of the input field expansion.
We hence prove that allowing for output-adapted arbitrary input linear combinations in the ridge approximation can improve the decay exponent by a factor of two for a class of functions arising from parametric elliptic PDEs.

\subsection{Other related work}
The active subspace method has been studied directly in infinite-dimensional settings; it has been formulated for scalar functionals on Hilbert spaces \cite{kundu2026active}, while a Sobolev framework for derivative-informed dimension reduction over infinite-dimensional Gaussian input measures has been developed in the context of operator learning \cite{luo2025dimension}.

Empirical active subspaces can be seen as an uncentered PCA problem for the Hilbert-valued gradient, and direct reconstruction-error bounds for empirical PCA in separable Hilbert spaces are well established \cite{rudi2013sample,reiss2020nonasymptotic,milbradt2020high}.
These results assume observations arise from a fixed distribution and do not provide relative error guarantees under designable ridge-leverage sampling.
Our sampling analysis is more closely related to ridge-leverage sampling results that preserve the reconstruction error of every rank-$r$ linear subspace \cite{musco2020pcp}.
For finite matrices, it has been shown in \cite{cohen2017ridge} that ridge-leverage sampling at a scale determined by the best rank-$r$ residual preserves the reconstruction error of every rank-$r$ linear subspace, however their theorem does not cover importance-weighted samples of a general Hilbert-valued field.
Related regularized leverage embeddings for integral and continuous operators appear in \cite{bach2017equivalence,avron2019universal}, but their conclusions concern random-feature approximation or regularized signal reconstruction rather than quasi-optimality of the top-$r$ eigenspace of the unregularized empirical covariance.
Analogous results for continuous operators in the random-feature setting preserve reconstruction errors under ridge-leverage sampling, but are specialized to that construction \cite{musco2017kernel}.

The PDE component in this paper builds on weighted Hermite and Legendre coefficient estimates for lognormal Gaussian and affine uniform elliptic equations, originally developed to establish sparse polynomial approximation rates \cite{cohen2015approximation,HoangSchwab2014,bachmayr2017sparsei,bachmayr2017sparseii,cohen2023near}.

\subsection{Contributions and outline}

The main contributions are as follows.
\begin{itemize}
  \item
        We prove a finite-sample quasi-optimality theorem controlling the population error of the empirical rank-$r$ active subspace by the rank-$r$ population active subspace error for importance-weighted samples of a measurable Hilbert-valued gradient field; see Theorem~\ref{thm:slas-ridge-quasi-opt}.
        The sufficient sampling condition is governed by a {regularized} inverse Christoffel coherence, leading to the bounded-gradient and oracle sample-complexity regimes described above.

  \item
        We show that RKHS smoothness replaces the coherence growth $\lambda^{-1}$ by $\lambda^{-\theta}$, and we construct {regularized} kernel diagonal and $\theta$-kernel diagonal importance measures which lead to better sampling conditions than those in the bounded-gradient regime; see Lemma~\ref{prop:smoothness-ridge-coherence}, Proposition~\ref{prop:smoothness-ridge-diagonal-sampling}, and Corollary~\ref{prop:smoothness-diagonal-sampling}.

  \item
        We use related smoothness assumptions to prove algebraic decay of the active subspace tail energy and obtain joint tolerance-dependent prescriptions for rank, {regularization} scale, and sample size; see Proposition~\ref{prop:smoothness-active-subspace-tail} and Theorem~\ref{thm:smoothness-tolerance-complexity}.

  \item
        We translate weighted Hermite and Legendre regularity estimates for lognormal Gaussian and affine uniform elliptic PDE models into the smoothness assumptions required by the active subspace analysis; see Theorems~\ref{thm:lognormal-pde-smoothness-consequences} and~\ref{thm:uniform-pde-smoothness-consequences}.
        In the lognormal Gaussian case, this yields an a priori ridge-error bound with twice the decay exponent of the ridge approximation based on the best-ranked coordinates according to the coefficients of the input field expansion.
\end{itemize}

The paper is organized as follows.
Section~\ref{sec:slas} introduces the active subspace formulation, its empirical importance sampling version, and its connection with ridge approximation.
Section~\ref{sec:slanal} proves the finite-sample quasi-optimality theorem and compares it with covariance-estimation approaches.
Section~\ref{sec:sample-complexity-smoothness} derives smoothness-based coherence estimates, importance sampling measures, and active subspace tail bounds.
Section~\ref{sec:parametric-pdes-smoothness} verifies the smoothness assumptions for lognormal Gaussian and affine uniform elliptic PDE models.
Section~\ref{sec:numerical-experiments} complements the analysis with numerical illustrations focused on the lognormal elliptic PDE benchmark.
}

\section{The active subspace method}
\label{sec:slas}

In this section, we review the active subspace method for possibly infinite-dimensional parameter spaces and introduce most of the notation used throughout the paper.
Let $d\in\N\cup\{+\infty\}$.
If $d<+\infty$, we define the parameter space as $\Gamma=\R^d$ with its Borel $\sigma$-algebra.
If $d=+\infty$, we set $\Gamma=\R^\N$ with the product topology and its Borel $\sigma$-algebra.
In both cases, we let $\mu$ be a Borel probability measure on $\Gamma$.
    {For} $m=1,\ldots,d$, {we} denote {by $\pi_m:\Gamma\to\R$} the {$m$th} coordinate {map, so that $\pi_m(\yb)=y_m$}.
When $d=+\infty$, coordinate-indexed sequences and sums are understood over $m\in\N$.
Throughout, $\yb\in\Gamma$ denotes a generic parameter value.
    {A $\Gamma$-valued random parameter is denoted by $Y$, with its law specified by context; independent} random samples are denoted by $Y_1,\ldots,Y_M$, and their realizations by $\yb_1,\ldots,\yb_M$.
We denote by $\ell^2(d)$ the space of admissible directions for parameter perturbations.
Note that, in general, a random parameter $Y\sim\mu$ need not itself belong to $\ell^2(d)$ when $d=+\infty$.
For the standard Gaussian product measure, $\ell^2(d)$ is the Cameron--Martin space.

Since the active subspace method relies on gradients, we first define a first-order Sobolev structure via the closure of the cylindrical gradient, that is the standard gradient operator for functions which are smooth and depend on a finite number of coordinates.
Let $\mathcal C^\infty_{\mathrm{cyl}}$ be the space of functions of the form
\begin{equation}
    h(\yb)
    =
    \phi\left({\pi_{i_1}}(\yb),\ldots,{\pi_{i_k}}(\yb)\right),
    \qquad
    \phi\in C_b^\infty(\R^k),
\end{equation}
where $k<+\infty$, $i_1,\ldots,i_k$ are coordinate indices, and $C_b^\infty(\R^k)$ denotes the space of bounded infinitely differentiable functions on $\R^k$ {with bounded partial derivatives of all orders}.
For such a function, set
\begin{equation}
    \label{eq:cylindrical-gradient}
    \nabla_0 h(\yb)
    :=
    \sum_{j=1}^k
    \partial_j\phi\left({\pi_{i_1}}(\yb),\ldots,{\pi_{i_k}}(\yb)\right)e_{i_j}
    \in\ell^2(d),
\end{equation}
where $\{e_m\}$ is the canonical orthonormal basis of $\ell^2(d)$ {and} ${\partial_j\phi}$ {denotes differentiation with respect to the} ${j}${th argument of} ${\phi}$.
We assume that this assignment defines a densely defined closable operator
\begin{equation}
    \nabla_0:
    \mathcal C^\infty_{\mathrm{cyl}}
    \subset L^2_\mu(\Gamma)
    \longrightarrow
    L^2_\mu(\Gamma;\ell^2(d)).
\end{equation}
We denote its closure by $\nabla$ and define $H^1_\mu(\Gamma)$ as the domain of $\nabla$ equipped with the norm
\begin{equation}
    \label{eq:H1-closed-gradient-definition}
    \norm{h}_{H^1_\mu(\Gamma)}^2
    :=
    \norm{h}_{L^2_\mu(\Gamma)}^2
    +
    \norm{\nabla h}_{L^2_\mu(\Gamma;\ell^2(d))}^2.
\end{equation}
Thus, the Sobolev space is defined once the cylindrical gradient is known to be closable.
When $d<+\infty$, a sufficient condition for closability is that $\mu$ admit a strictly positive $C^1$ density with respect to Lebesgue measure.
In this case, $H^1_\mu(\R^d)$ coincides with the standard first-order weighted Sobolev space; see \cite[Corollary~1.2]{tolle2012uniqueness}.
For the standard Gaussian infinite-product measure and the infinite product of centered uniform measures on $[-1,1]$, closability is proved in Appendix~\ref{app:product-sobolev-calculus}.
In these two cases, the resulting {Sobolev} space {and gradient operator admit} the coordinate characterization
\begin{equation}
    \label{eq:coordinate-weak-derivative-characterization}
    H^1_\mu(\Gamma)
    =
    \left\{
    h\in L^2_\mu(\Gamma)
    :
    (\partial_mh)_{m=1}^d
    \in L^2_\mu(\Gamma;\ell^2(d))
    \right\},
    \qquad
    \nabla h=(\partial_mh)_{m=1}^d,
\end{equation}
where $\partial_mh$ is the weak partial derivative along the $m$th coordinate.
    {The latter is defined as follows.
        For the Gaussian product measure, a function $g_m\in L^2_\mu(\Gamma)$ is called the $m$th weak derivative of $h\in L^2_\mu(\Gamma)$ if}
\begin{equation}
    {\label{eq:gaussian-weak-derivative}
        \int_\Gamma g_m\Phi\dif\mu
        =
        \int_\Gamma
        h\left(\pi_m\Phi-\partial_m\Phi\right)
        \dif\mu}
\end{equation}
{for every $\Phi\in\mathcal C^\infty_{\mathrm{cyl}}$.
For the uniform product measure, $g_m$ is called the $m$th weak derivative of $h$ if}
\begin{equation}
    {\label{eq:uniform-weak-derivative}
        \int_\Gamma g_m\Phi\dif\mu
        =
        -
        \int_\Gamma h\partial_m\Phi\dif\mu}
\end{equation}
{for every cylindrical test function $\Phi$ whose dependence on the $m$th coordinate is compactly supported in $(-1,1)$.}
Throughout the paper, {$f$ denotes the target parameter-to-output function to which} we {apply the active subspace method.
        We} assume that $f\in H^1_\mu(\Gamma)$ and that its gradient is nonzero in $L^2_\mu(\Gamma;\ell^2(d))$.
    {Since the gradient field is central to the subsequent analysis, we} write
\begin{equation}
    \label{eq:gradient-field-notation}
    g:=\nabla f
    \in L^2_\mu(\Gamma;\ell^2(d))\setminus\{0\}.
\end{equation}

We next define the linear observations used to form reduced variables.
Let $c_{00}(d)\subset\ell^2(d)$ be the space of finitely supported sequences, with $c_{00}(d)=\R^d$ when $d<+\infty$.
For $a\in c_{00}(d)$, set
\begin{equation}
    \label{eq:finite-linear-observable}
    {X_a}
    :=
    \sum_{m=1}^d {a_m\pi_m},
\end{equation}
where the sum is finite.
We assume that {the measure} ${\mu}$ {satisfies, for some constant} ${C_\mu}<+\infty$,
\begin{equation}
    \label{eq:linear-observable-bound}
    \norm{{X_a}}_{L^2_\mu(\Gamma)}^2
    \leq
    {C_\mu}\norm{a}_{\ell^2(d)}^2,
    \qquad
    a\in c_{00}(d).
\end{equation}
{Consequently, for} $v\in\ell^2(d)$, {the following limit exists in $L^2_\mu(\Gamma)$, and} we write
\begin{equation}
    \label{eq:linear-observation-functional}
    {X_v}
    :=
    L^2_\mu\text{-}\lim_{n\to+\infty}
    \sum_{m=1}^{\min\{n,d\}}{v_m\pi_m}.
\end{equation}
Here and below, $\min\{n,+\infty\}=n$.
For {the standard} Gaussian {infinite-product measure on $\Gamma$}, \eqref{eq:linear-observable-bound} holds with ${C_\mu}=1${, while} for {the infinite product of} centered uniform {probability measures} on $[-1,1]$ it holds with ${C_\mu}=1/3$.
    {More generally, under the standing closability assumption}, Lemma~\ref{lem:linear-observables-H1} shows that {every measure satisfying \eqref{eq:linear-observable-bound} also satisfies}, for {every} $v\in\ell^2(d)$,
\begin{equation}
    \label{eq:gradient-linear-observable}
    {X_v}\in H^1_\mu(\Gamma),
    \qquad
    \nabla{X_v}=v.
\end{equation}
{Thus, the linear observations used as reduced variables are compatible with the Sobolev structure.}

Fix now a target dimension $r\in\N$, with $r\leq d$ when $d<+\infty$.
Let $V:\R^r\to\ell^2(d)$ be an isometry {and set} ${v_j:=V(e_j)}$ {for} ${j=1},\ldots,{r}$, so that $V^\adj V=I_r$.
We {write}
\begin{equation}
    \label{eq:active-observation-map}
    {X_V}
    :=
    \left(
    {X_{v_1}},\ldots,{X_{v_r}}
    \right)
    \in L^2_\mu(\Gamma;\R^r),
\end{equation}
{and} call {its} components the reduced variables associated with $V${.
        We fix a measurable representative of $X_V$ and write $X_V(\yb)$ for its value at $\yb$.
        When $d<+\infty$, $V$ is a $d\times r$ matrix and $X_V(\yb)$ is the usual matrix product $V^\top\yb$.
        When $d=+\infty$, $X_V$ is interpreted componentwise through the $L^2_\mu$ limits in \eqref{eq:linear-observation-functional}.
        This does not require $\yb$ to belong to $\ell^2(d)$, and statements involving $X_V(\yb)$ are understood $\mu$-almost everywhere}.
The corresponding orthogonal projector on $\ell^2(d)$ is
\begin{equation}
    \Pi_V:=VV^\adj.
\end{equation}
We denote by
\begin{equation}
    \mathcal G_V
    :=
    \overline{\sigma({X_V})}^{\,\mu}
\end{equation}
the $\mu$-completion of the $\sigma$-algebra generated by ${X_V}$, and by
$\mu_V:=({X_V})_\#\mu$ its law on $\R^r$.
If $Q\in\R^{r\times r}$ is orthogonal, then
\begin{equation*}
    {X_{VQ}}=Q^\adj {X_V}
    \qquad
    \mu\text{-almost surely}.
\end{equation*}
Consequently, $\mathcal G_V$ and $\mu_V$ up to the corresponding orthogonal change of coordinates depend only on the range of $V$ or, equivalently, on $\Pi_V$.

Let us now describe the ridge-approximation problem which underlies the active subspace method.
Given an isometry $V$, the idea is to construct a surrogate of $f$ by {composing} ${X_V}$ {with} a regressor depending only on $r$ inputs, namely a ridge {function} of the form
\begin{equation}
    \label{eq:approx-model}
    \varphi({X_V}),
    \qquad
    \varphi\in L^2_{\mu_V}(\R^r).
\end{equation}
The corresponding ideal ridge-approximation problem then is
\begin{equation}
    \label{eq:general-AS-obj}
    \inf_{\substack{V:\R^r\to\ell^2(d)\\ V^\adj V=I_r}}
    \inf_{\varphi\in L^2_{\mu_V}(\R^r)}
    \norm{f-\varphi({X_V})}_{L^2_\mu(\Gamma)}^2.
\end{equation}
The active subspace method avoids solving this joint nonlinear problem and selects the subspace from first-order information about $f$.

The intuition stems from the fact that the most important directions in the input space are those along which the output varies the most.
For a unit vector $v\in\ell^2(d)$, output variability can be measured by the mean-square directional sensitivity of $f$ along $v$, that is the quantity
\begin{equation}
    \label{eq:directional-sensitivity}
    \int_\Gamma
    \abs{\inp{g(\yb)}{v}_{\ell^2(d)}}^2
    \dif\mu(\yb).
\end{equation}
These sensitivities can be represented by the operator
\begin{equation}
    \label{eq:cov}
    C
    :=
    \int_\Gamma
    g(\yb)\otimes g(\yb)
    \dif\mu(\yb),
\end{equation}
where $\left(g(\yb)\otimes g(\yb)\right)v := \inp{g(\yb)}{v}_{\ell^2(d)}g(\yb)$.
Note that $C$ is self-adjoint, nonnegative, and trace class, with
\begin{equation}
    \operatorname{tr}(C)
    =
    \int_\Gamma
    \norm{g(\yb)}_{\ell^2(d)}^2
    \dif\mu(\yb).
\end{equation}
Following the active subspace literature, we refer to \eqref{eq:cov} as the gradient covariance operator, even though it is the uncentered second moment of $g$ \cite{constantine2015active}.

For an orthogonal projector $\Pi$ on $\ell^2(d)$, define the population gradient projection error $\err(\Pi)$ by
\begin{equation}
    \label{eq:active-subspace-error}
    \err(\Pi)^2
    :=
    \int_\Gamma
    \norm{(I-\Pi)g(\yb)}_{\ell^2(d)}^2
    \dif\mu(\yb)
    =
    \operatorname{tr}\left((I-\Pi)C\right).
\end{equation}
The squared error $\err(\Pi)^2$ is the discarded gradient energy and measures the mean-square directional sensitivity of $f$ along directions orthogonal to $\operatorname{Ran}(\Pi)$.
Hence, one defines a rank-$r$ active subspace projector as any solution of
\begin{equation}
    \label{eq:population-active-subspace-projector}
    \Pi_r
    \in
    \operatorname*{argmin}_{\substack{
        \Pi=\Pi^\adj=\Pi^2,\\
        \operatorname{rank}(\Pi)=r
    }}
    \err(\Pi)^2,
\end{equation}
and the corresponding active subspace as $\operatorname{Ran}(\Pi_r)$.
Let $V_r:\R^r\to\ell^2(d)$ be an isometry satisfying $\Pi_{V_r}=\Pi_r$.
Following the active subspace literature, we call the components of ${X_{V_r}}$ the active variables \cite{constantine2015active}.
When $d<+\infty$, let $W_r:\R^{d-r}\to\R^d$ be an isometry with columns $w_{r+1},\ldots,w_d$ such that
\begin{equation*}
    \operatorname{Ran}(W_r)
    =
    \operatorname{Ran}(\Pi_r)^\perp.
\end{equation*}
The components of ${X_{W_r}}$ are called the inactive variables.
When $d=+\infty$, an orthonormal basis $(w_j)_{j>r}$ of $\operatorname{Ran}(\Pi_r)^\perp$ analogously determines a countable family $({X_{w_j}})_{j>r}$ of inactive variables.
In general, this family need not define an $\ell^2$-valued random variable.

Since $C$ is positive and trace class, it admits a spectral decomposition
\begin{equation}
    \label{eq:cov-spectral-decomposition}
    C
    =
    \sum_{j=1}^{d}
    \sigma_j^2u_j\otimes u_j,
    \qquad
    \sigma_1^2\geq\sigma_2^2\geq\cdots\geq0.
\end{equation}
When $d<+\infty$, we extend the eigenvalue sequence by setting $\sigma_j=0$ for $j>d$.
Ky Fan's variational principle for positive compact operators implies that a minimizer in \eqref{eq:population-active-subspace-projector} is
\begin{equation}
    \label{eq:population-active-subspace-spectral-projector}
    \Pi_r
    =
    \sum_{j=1}^r u_j\otimes u_j,
\end{equation}
and that
\begin{equation}
    \label{eq:population-active-subspace-tail}
    \err(\Pi_r)^2
    =
    \sum_{j>r}\sigma_j^2;
\end{equation}
see \cite[Sections 1.6 and 1.8]{simon2005trace}.
Thus, $\err(\Pi_r)$ is the rank-$r$ population active subspace error, while the active subspace tail energy $\err(\Pi_r)^2$ is the residual sensitivity outside the selected subspace.

For fixed $V$, the best ridge approximation {of $f$} is the conditional expectation
$\mathbb E[f\mid\mathcal G_V]$.
By the Doob--Dynkin lemma, there exists
$\varphi_V^\star\in L^2_{\mu_V}(\R^r)$, unique up to
$\mu_V$-almost everywhere equality, such that
\begin{equation}
    \label{eq:best-g}
    \varphi_V^\star({X_V})
    =
    \mathbb E\left[f\,\middle\vert\,\mathcal G_V\right]
    \qquad
    \mu\text{-almost surely}.
\end{equation}
Thus, $\varphi_V^\star({X_V})$ is the orthogonal projection of $f$ onto $L^2_\mu(\Gamma,\mathcal G_V)$; see \cite[Chapter 9]{williams_1991}.

Let us now introduce a functional inequality that allows us to convert the squared population gradient projection error \eqref{eq:active-subspace-error} into a ridge-reconstruction guarantee.
\begin{definition}[Subspace conditional Poincar\'e inequality]
    \label{asmp:poincare}
    Let ${V:\R^r\to}\ell^2(d)$ be {an} isometry.
    We say that $\mu$ satisfies the subspace conditional Poincar\'e inequality for $\Pi{_V}$ with constant $C_P(\Pi{_V})<+\infty$ if
    \begin{equation}
        \label{eq:conditional-poincare}
        \norm{
            h-\mathbb E\left[h\,\middle\vert\,\mathcal G_V\right]
        }_{L^2_\mu(\Gamma)}^2
        \leq
        C_P(\Pi{_V})
        \norm{
            (I-\Pi{_V})\nabla h
        }_{L^2_\mu(\Gamma;\ell^2(d))}^2
    \end{equation}
    for every $h\in H^1_\mu(\Gamma)$.
\end{definition}
\noindent
{The} definition {depends only on} $\Pi{_V}$ because any two {isometries with the same range} differ by an orthogonal change of coordinates and generate the same completed $\sigma$-algebra.
The above inequality allows us to bound the mean-square error of the best ridge approximation for a given isometry in terms of the squared population gradient projection error.

\begin{theorem}[Ridge reconstruction bound]
    \label{thm:err-bound}
    Let $V:\R^r\to\ell^2(d)$ be an isometry and suppose that $\mu$ satisfies the subspace conditional Poincar\'e inequality for $\Pi_V$.
    Then
    \begin{equation}
        \label{eq:err-bound}
        \norm{
            f-\varphi_V^\star({X_V})
        }_{L^2_\mu(\Gamma)}^2
        \leq
        C_P(\Pi_V)\err(\Pi_V)^2.
    \end{equation}
    Moreover, for every $\varphi\in L^2_{\mu_V}(\R^r)$,
    \begin{align}
        \label{eq:err-approx-f}
        \norm{
            f-\varphi({X_V})
        }_{L^2_\mu(\Gamma)}^2
         & =
        \norm{
            f-\varphi_V^\star({X_V})
        }_{L^2_\mu(\Gamma)}^2
        +
        \norm{
            \varphi_V^\star-\varphi
        }_{L^2_{\mu_V}(\R^r)}^2
        \nonumber \\
         & \leq
        C_P(\Pi_V)\err(\Pi_V)^2
        +
        \norm{
            \varphi_V^\star-\varphi
        }_{L^2_{\mu_V}(\R^r)}^2.
    \end{align}
    In particular, if $V_r:\R^r\to\ell^2(d)$ satisfies $\Pi_{V_r}=\Pi_r$ and the subspace conditional Poincar\'e inequality holds for $\Pi_r$, then
    \begin{equation}
        \label{eq:ridge-bound-spectral-tail}
        \norm{
            f-\varphi_{V_r}^\star({X_{V_r}})
        }_{L^2_\mu(\Gamma)}^2
        \leq
        C_P(\Pi_r)
        \sum_{j>r}\sigma_j^2.
    \end{equation}
\end{theorem}

\begin{proof}
    Applying \eqref{eq:conditional-poincare} to $h=f$ and using \eqref{eq:best-g} gives
    \begin{equation}
        \norm{
            f-\varphi_V^\star({X_V})
        }_{L^2_\mu(\Gamma)}^2
        \leq
        C_P(\Pi_V)
        \norm{
            (I-\Pi_V)\nabla f
        }_{L^2_\mu(\Gamma;\ell^2(d))}^2,
    \end{equation}
    which is \eqref{eq:err-bound}.
    Since
    $f-\varphi_V^\star({X_V})$ is orthogonal in $L^2_\mu(\Gamma)$
    to every $\mathcal G_V$-measurable square-integrable random variable,
    it is orthogonal to $\varphi_V^\star({X_V})-\varphi({X_V})$.
    The Pythagorean theorem and the change-of-variables identity under $\mu_V=({X_V})_\#\mu$ therefore give \eqref{eq:err-approx-f}.
    Finally, \eqref{eq:ridge-bound-spectral-tail} follows from \eqref{eq:err-bound} and \eqref{eq:population-active-subspace-tail}.
\end{proof}

The following proposition shows that the subspace conditional Poincar\'e inequality holds for Gaussian product measures with a constant independent of the dimension and the projector.
\begin{proposition}[Gaussian product measures]
    \label{prop:gaussian-conditional-poincare}
    Let $d\in\N\cup\{+\infty\}$ and let $\mu=\gamma^{\otimes d}$ be the standard Gaussian product measure on $\Gamma$.
    Then, every finite-rank orthogonal projector $\Pi$ on $\ell^2(d)$ satisfies the subspace conditional Poincar\'e inequality with
    \begin{equation}
        C_P(\Pi)=1.
    \end{equation}
\end{proposition}
\begin{proof}
    See Appendix~\ref{app:gaussian-conditional-poincare}.
\end{proof}

\begin{remark}[Uniform product measures]
    \label{rem:uniform-product-poincare}
    Let $\lambda=\frac12\mathcal L^1\vert_{[-1,1]}$ and $\mu=\lambda^{\otimes d}$.
    The closed-gradient Sobolev structure and the identity \eqref{eq:gradient-linear-observable} remain valid{.
            If $V$ selects coordinate directions}, {the one-dimensional Poincar\'e inequality and tensorization give the} subspace conditional Poincar\'e inequality {with $C_P(\Pi_V)=4/\pi^2$.
            For a general $V$, however, conditioning on $\mathcal G_V$ couples the coordinates: the conditional measures are uniform measures on affine sections of the cube and are no longer product measures.
            Therefore,} the one-dimensional {inequality cannot be applied by tensorization.
    When $d<+\infty$, the Poincar\'e} inequality {for convex domains yields the dimension-dependent bound $C_P(\Pi_V)\leq 4d/\pi^2$; see \cite{bebe2003}}.
    {When $d=+\infty$, a suitable uniform Poincar\'e bound for these conditional measures must instead be established separately, which} is {beyond} the {scope of this work}.
\end{remark}

The analysis in the remainder of the paper uses only the square-integrable gradient field in \eqref{eq:gradient-field-notation} and the objective \eqref{eq:active-subspace-error}.
The subspace conditional Poincar\'e inequality is needed only to convert squared population gradient projection error into a ridge-reconstruction guarantee.
We therefore focus on the covariance operator \eqref{eq:cov} and its empirical counterpart introduced next, which is the main object of study in the next section.

In practice, \eqref{eq:cov} -- which {we} refer to as population covariance -- is replaced by a sample average -- the empirical covariance -- since the true integral is typically unknown.
Let $\rho$ be a probability measure on $\Gamma$ such that $\mu\ll\rho$, and set
\begin{equation}
    \label{eq:is-weight}
    w
    :=
    \frac{\dif\mu}{\dif\rho}.
\end{equation}
Choose a Borel version of $g$ and set it equal to zero on $\{w=0\}$.
This convention does not affect weighted quantities: if two Borel versions agree $\mu$-almost everywhere, then they differ only on a set on which $w=0$ $\rho$-almost everywhere.
Then,
\begin{equation}
    C
    =
    \int_\Gamma
    w(\yb)g(\yb)\otimes g(\yb)
    \dif\rho(\yb).
\end{equation}
Given independent random samples
$Y_1,\ldots,Y_M\stackrel{\mathrm{iid}}{\sim}\rho$, define the empirical covariance operator by the importance sampling estimator
\begin{equation}
    \label{eq:emp-cov-is}
    \hat C
    :=
    \frac1M
    \sum_{i=1}^M
    w(Y_i)g(Y_i)\otimes g(Y_i).
\end{equation}
The estimator is unbiased as a trace-class-valued random operator since $\mathbb E_\rho[\hat C]=C$.
When $\rho=\mu$, one has $w=1$ almost everywhere and \eqref{eq:emp-cov-is} is the standard Monte Carlo covariance estimator.
We define the empirical active subspace projector $\hat\Pi_r$ as any leading rank-$r$ spectral projector of $\hat C$.
Equivalently, it minimizes $\operatorname{tr}((I-\Pi)\hat C)$ over rank-$r$ orthogonal projectors.
The construction is summarized in Algorithm~\ref{algo:ASM}.

\begin{algorithm}[h]
    \caption{Active subspace method}
    \label{algo:ASM}
    \Input{Gradient $g=\nabla f$, sampling measure $\rho$, importance weight $w=\dif\mu/\dif\rho$, target rank $r$, and sample size $M$.}

    Draw independent sample values
    $\{\yb_i\}_{i=1}^M$ from $\rho$.

    Evaluate $g(\yb_i)$ for $i=1,\ldots,M$.

    Form
    \begin{equation}
        \hat C
        =
        \frac1M
        \sum_{i=1}^M
        w(\yb_i)g(\yb_i)\otimes g(\yb_i).
    \end{equation}

    Compute a leading rank-$r$ spectral projector $\hat\Pi_r$ of $\hat C$.

    \Output{The empirical active subspace projector $\hat\Pi_r$.}
\end{algorithm}

If $\hat\Pi_r=\Pi_{\hat V_r}$, then the components of ${X_{\hat V_r}}$
are the empirical active variables, or reduced coordinates, for visualization, sensitivity analysis,
regression, optimization, and surrogate construction
\cite{constantine2015active}.
The next section studies how many gradient samples are sufficient for $\hat\Pi_r$ to be quasi-optimal for the population objective, namely for
\begin{equation}
    \label{eq:active-subspace-quasi-optimality-goal}
    \err(\hat\Pi_r)^2
    \lesssim
    \err(\Pi_r)^2
\end{equation}
to hold with high probability.

\section{Quasi-optimality of the empirical active subspace}
\label{sec:slanal}

We keep the notation introduced in Section~\ref{sec:slas}.
In particular, $g=\nabla f$, $C$ denotes the gradient covariance
\eqref{eq:cov}, $\hat C$ denotes the empirical covariance
\eqref{eq:emp-cov-is}, and $\err(\Pi)$ denotes the population gradient
projection error \eqref{eq:active-subspace-error}.
The reconstruction estimate in Theorem~\ref{thm:err-bound} shows that, under the
subspace conditional Poincar\'e inequality, this error controls the ideal
ridge-reconstruction error.
In this section we focus on the statistical part of the problem: the replacement
of the population active subspace projector $\Pi_r$ by the empirical projector
$\hat\Pi_r$ computed from sampled gradients.

Throughout this section, the random samples $\{Y_i\}_{i=1}^M$ are drawn independently
from the sampling measure $\rho$, and $w=\dif\mu/\dif\rho$ denotes the
importance weight.
All probability statements are with respect to this sampling procedure.
For a measurable vector field $h:\Gamma\to\ell^2(d)$, we define the empirical
weighted seminorm
\begin{equation}
    \label{eq:empirical-weighted-seminorm}
    \norm{h}_M^2
    :=
    \frac{1}{M}
    \sum_{i=1}^M
    w(Y_i)\norm{h(Y_i)}_{\ell^2(d)}^2.
\end{equation}
The empirical counterpart of the population gradient projection error $\err(\Pi)$ is then
\begin{equation}
    \label{eq:empirical-active-subspace-error}
    \hat{\err}(\Pi)^2
    :=
    \norm{(I-\Pi)g}_M^2
    =
    \operatorname{tr}\left((I-\Pi)\hat C\right).
\end{equation}
Thus, $\Pi_r$ and $\hat\Pi_r$ are characterized by the two variational problems
\begin{equation}
    \label{eq:population-empirical-projectors}
    \Pi_r
    \in
    \operatorname*{argmin}_{
    \substack{
        \Pi=\Pi^\adj=\Pi^2,\\
        \operatorname{rank}(\Pi)=r
    }
    }
    \err(\Pi)^2,
    \qquad
    \hat\Pi_r
    \in
    \operatorname*{argmin}_{
    \substack{
        \Pi=\Pi^\adj=\Pi^2,\\
        \operatorname{rank}(\Pi)=r
    }
    }
    \hat{\err}(\Pi)^2.
\end{equation}
The goal is to obtain conditions
on the number of samples $M$
under which the empirical
minimizer is quasi-optimal for the population criterion on a high-probability event, namely
\begin{equation}
    \label{eq:quasi-optimality-target}
    \err(\hat\Pi_r)^2
    \leq
    c\,\err(\Pi_r)^2
\end{equation}
with a moderate constant $c>1$.
The results below give such conditions in terms of a regularized stability
quantity associated with the sampled gradients.

Using the spectral decomposition \eqref{eq:cov-spectral-decomposition},
the corresponding rank-$r$ projector
\eqref{eq:population-active-subspace-spectral-projector},
and the tail identity \eqref{eq:population-active-subspace-tail},
we shall compare the empirical minimizer $\hat\Pi_r$ with $\Pi_r$.
We denote by
\begin{equation}
    \label{eq:covariance-range}
    \mathcal R
    :=
    \overline{\operatorname{Ran}(C)}
    =
    (\ker C)^\perp
    =
    \overline{\operatorname{span}}\{u_j:\sigma_j>0\}
\end{equation}
the covariance support space.
Then $g(\yb)\in\mathcal R$ for $\mu$-almost every $\yb$.
Indeed, if $v\in\ker C=\mathcal R^\perp$, then
\begin{equation*}
    0
    =
    \inp{Cv}{v}_{\ell^2(d)}
    =
    \int_\Gamma
    \abs{\inp{g(\yb)}{v}_{\ell^2(d)}}^2
    \dif\mu(\yb).
\end{equation*}
Since $\mathcal R^\perp$ is separable, applying this identity on a countable
dense subset of $\mathcal R^\perp$ gives
$g(\yb)\perp\mathcal R^\perp$,
and hence $g(\yb)\in\mathcal R$,
for $\mu$-almost every $\yb$.

The argument below follows the same general mechanism as randomized low-rank
approximation via ridge-leverage{-score} sampling {\cite{cohen2017ridge}}:
one first proves a uniform concentration estimate for a regularized covariance
and then converts {it} into a uniform statement for rank-$r$ projection errors.
{The regularized inverse Christoffel function is the continuous analogue of a regularized leverage score, commonly called a ridge leverage score in randomized linear algebra \cite{pauwels2018relating}.}
Let us further remark that the use of intrinsic-dimension matrix concentration in active subspace
estimation {has been used before to} estimate $C$ uniformly in operator norm; {see, e.g., \cite{lam2020multifidelity}}.

\subsection{\texorpdfstring{The {regularized} inverse Christoffel function}{The regularized inverse Christoffel function}.}

Fix $\lambda>0$ and define the regularized covariance operator on $\mathcal R$ by
\begin{equation}
    \label{eq:regularized-covariance}
    C_\lambda := C+\lambda I_{\mathcal R}.
\end{equation}
On the eigenbasis {from \eqref{eq:cov-spectral-decomposition} we have}
$C_\lambda^{-1/2}u_j=(\sigma_j^2+\lambda)^{-1/2}u_j$ for $\sigma_j>0$.
When convenient, we extend $C_\lambda^{-1/2}$ by zero on $\ker C$.
The weighted {regularized} inverse Christoffel function associated with the gradient $g$
and sampling measure $\rho$ is
\begin{equation}
    \label{eq:ridge-christoffel}
    k_{\rho,\lambda}(\yb)
    :=
    w(\yb)\,
    \|C_\lambda^{-1/2}g(\yb)\|_{\ell^2(d)}^2.
\end{equation}
Using the Karhunen--Loeve expansion {of $g$}
\begin{equation}
    \label{eq:gradient-kl}
    g(\yb)
    =
    \sum_{j\geq 1}\sigma_j\,\eta_j(\yb)\,u_j,
    \qquad
    \int_\Gamma \eta_i(\yb)\eta_j(\yb)\,\dif\mu(\yb)
    =
    \delta_{ij},
\end{equation}
we may also write
\begin{equation}
    \label{eq:ridge-christoffel-expansion}
    k_{\rho,\lambda}(\yb)
    =
    w(\yb)
    \sum_{j\geq 1}
    \frac{\sigma_j^2}{\sigma_j^2+\lambda}
    \eta_j(\yb)^2.
\end{equation}
For later use, if $\xi=(\xi_j)_{j\geq 1}$ is a nonnegative sequence and
$t>0$, we write
\begin{equation}
    \label{eq:generic-effective-dimension}
    d_{\mathrm{eff}}(t,\xi)
    :=
    \sum_{j\geq 1}
    \frac{\xi_j}{\xi_j+t}.
\end{equation}
We define the corresponding coherence constant $K_{\rho,\lambda}(g)$ of $g$
and the covariance effective dimension $d_{\mathrm{eff}}(\lambda)$ by
\begin{equation}
    \label{eq:ridge-coherence-effective-dimension}
    K_{\rho,\lambda}(g)
    :=
    \rho\text{-}\operatorname*{ess\,sup}_{\yb\in\Gamma}
    k_{\rho,\lambda}(\yb),
    \qquad
    d_{\mathrm{eff}}(\lambda)
    :=
    \int_\Gamma \sum_{j\geq 1}
    \frac{\sigma_j^2}{\sigma_j^2+\lambda}
    \eta_j(\yb)^2 \,\dif\mu(\yb).
\end{equation}
Notice that
\begin{equation}
    \label{eq:effective-dimension-expansion}
    d_{\mathrm{eff}}(\lambda)
    =
    d_{\mathrm{eff}}\left(\lambda,(\sigma_j^2)_{j\geq 1}\right)
    = \operatorname{tr}\left(C(C+\lambda I_{\mathcal R})^{-1}\right)
    =
    \sum_{j\geq 1}
    \frac{\sigma_j^2}{\sigma_j^2+\lambda},
\end{equation}
hence the number $d_{\mathrm{eff}}(\lambda)$ measures the dimension of the covariance
above the {regularization} scale $\lambda$.
When $\err(\Pi_r)^2>0$, we define the rank-$r$ canonical {regularization} scale by
\begin{equation}
    \label{eq:canonical-ridge-scale}
    \lambda_r
    :=
    \frac{\err(\Pi_r)^2}{r}
    =
    \frac{1}{r}\sum_{j>r}\sigma_j^2,
\end{equation}
Thus, $\lambda_r$ is the active subspace tail energy normalized by the target rank.
For this scale,
\begin{equation}
    \label{eq:effective-dimension-bound}
    d_{\mathrm{eff}}(\lambda_r)
    \leq
    \sum_{j\leq r}1
    +
    \frac{1}{\lambda_r}\sum_{j>r}\sigma_j^2
    =
    2r.
\end{equation}
We shall use the convention
\begin{equation}
    \label{eq:K-lambda-plus}
    K_{\rho,\lambda}^+(g):=\max\{K_{\rho,\lambda}(g),1\}.
\end{equation}

{Regularized leverage-score embeddings closely related to the event below have
previously been established for integral and continuous operators
\cite{bach2017equivalence,avron2019universal}.}

\begin{lemma}[The regularized covariance event]
    \label{lem:ridge-good-event}
    Assume $K_{\rho,\lambda}(g)<\infty$ for some $\lambda>0$.
    There exist universal constants $c_0,c_1>0$ such that, for every
    $\delta\in(0,1)$ and every $\eta\in(0,1)$, the condition
    \begin{equation}
        \label{eq:good-event-sample-condition}
        M
        \geq
        \frac{K_{\rho,\lambda}^+(g)}{c_1 \delta^2}
        \log\left(\frac{c_0(1+d_{\mathrm{eff}}(\lambda))}{\eta}\right)
    \end{equation}
    implies
    \begin{equation}
        \label{eq:good-event-probability}
        \Prob{\calE_\delta}
        \geq
        1-\eta,
    \end{equation}
    where
    \begin{equation}
        \label{eq:good-event}
        \calE_\delta
        :=
        \left\{
        \left\|
        C_\lambda^{-1/2}
        (\hat C-C)
        C_\lambda^{-1/2}
        \right\|_{\mathcal L(\mathcal R)}
        \leq
        \delta
        \right\}.
    \end{equation}
\end{lemma}

\begin{proof}
    Set
    \begin{equation}
        \label{eq:whitened-gradient}
        z(\yb):=C_\lambda^{-1/2}g(\yb)\in\mathcal R,
    \end{equation}
    and
    \begin{equation}
        \label{eq:whitened-covariance}
        B_\lambda
        :=
        \int_\Gamma
        w(\yb)z(\yb)\otimes z(\yb)
        \dif\rho(\yb)
        =
        C_\lambda^{-1/2}CC_\lambda^{-1/2}
        =
        C(C+\lambda I_{\mathcal R})^{-1}.
    \end{equation}
    Then,
    \begin{equation}
        \label{eq:whitened-covariance-trace}
        \operatorname{tr}(B_\lambda)
        =
        d_{\mathrm{eff}}(\lambda),
        \qquad
        \|B_\lambda\|_2\leq 1,
    \end{equation}
    and notice that the event $\calE_\delta$ is precisely the event $\|Z\|_2 \leq \delta$, where
    \begin{equation}
        Z :=
        \frac{1}{M}\sum_{i=1}^M
        w(Y_i)z(Y_i)\otimes z(Y_i)
        -
        B_\lambda.
    \end{equation}
    Define the centered random operators
    \begin{equation}
        \label{eq:bernstein-summands}
        S_i
        :=
        w(Y_i)z(Y_i)\otimes z(Y_i)-B_\lambda,
    \end{equation}
    which are independent, self-adjoint, and zero mean.
    Since
    \begin{equation}
        \label{eq:ridge-christoffel-bound-summands}
        \|w(Y_i)z(Y_i)\otimes z(Y_i)\|_2
        =
        w(Y_i)\|z(Y_i)\|_{\ell^2(d)}^2
        =
        k_{\rho,\lambda}(Y_i)
        \leq
        K_{\rho,\lambda}(g),
    \end{equation}
    {and $-B_\lambda \preceq S_i \preceq w(Y_i)z(Y_i)\otimes z(Y_i)$,}
    we have the uniform bound
    \begin{equation}
        \|S_i\|_2
        \leq
        {\max\{}
        \|w(Y_i)z(Y_i)\otimes z(Y_i)\|_2,
        \|B_\lambda\|_2
        \}
        \leq
        K_{\rho,\lambda}^+(g).
    \end{equation}
    Furthermore,
    \begin{equation}
        \label{eq:variance-bound-ridge}
        \Exp{S_i^2}
        =
        \Exp{\left(w(Y_i)z(Y_i)\otimes z(Y_i)\right)^2} - B_\lambda^2
        \preceq
        \Exp{\left(w(Y_i)z(Y_i)\otimes z(Y_i)\right)^2}
        \preceq
        K_{\rho,\lambda}(g)\,B_\lambda,
    \end{equation}
    therefore
    \begin{equation}
        \Var{Z}
        := \Exp{(Z - \Exp{Z})^2}
        = \frac{1}{M^2} \sum_{i=1}^M \Exp{S_i^2}
        \preceq
        \frac{ K_{\rho,\lambda}(g)}{M} B_\lambda.
    \end{equation}
    We now use the intrinsic-dimension form of the Bernstein inequality for self-adjoint Hilbert--Schmidt operators in the formulation given in
    \cite[Section~3.2]{minsker2017bernstein}.
    We apply the theorem with the inflated variance bound
    \begin{equation}
        \Var{Z}
        \preceq
        \frac{ K_{\rho,\lambda}^+(g)}{M} (B_\lambda + (1-\|B_\lambda\|_2) P_1),
    \end{equation}
    where $P_1$ is {a rank-one orthogonal} projector {onto an eigenvector} of $B_\lambda$ associated with its largest eigenvalue, so that the effective rank defined in \cite{minsker2017bernstein} can be bounded by $1+d_{\mathrm{eff}}(\lambda)$.
    This gives
    \begin{equation}
        \label{eq:intrinsic-bernstein-tail}
        \Prob{\calE_\delta^c}
        \leq
        14 \left(1+d_{\mathrm{eff}}(\lambda)\right)
        \exp\left(
        -
        \frac{M\delta^2}{2{(1 + \delta/3)} K_{\rho,\lambda}^+(g)}
        \right),
    \end{equation}
    provided $M \delta^2 \geq K_{\rho,\lambda}^+(g) + \delta K_{\rho,\lambda}^+(g) /3$.
    It is easy to check that the sampling condition \eqref{eq:good-event-sample-condition} implies this requirement for appropriately chosen universal constants $c_0,c_1>0$.
    The sampling condition \eqref{eq:good-event-sample-condition} follows from imposing the
    right-hand side of the above to be at most $\eta$,
    which proves \eqref{eq:good-event-probability}.
\end{proof}

\subsection{From regularized concentration to projection-error control}

We next describe how the event $\calE_\delta$ controls empirical projection errors.
Fix $\lambda>0$ and split the spectrum into the head and tail index sets
\begin{equation}
    \label{eq:head-tail-index-sets}
    \mathcal I_H
    :=
    \{j:\sigma_j^2>\lambda\},
    \qquad
    \mathcal I_T
    :=
    \{j:\sigma_j^2\leq\lambda\}.
\end{equation}
Let
\begin{equation}
    \label{eq:head-tail-projectors}
    P_H
    :=
    \sum_{j\in\mathcal I_H}u_j\otimes u_j,
    \qquad
    P_T
    :=
    I_{\mathcal R}-P_H.
\end{equation}
Since $C$ is trace class and $\lambda>0$, the head rank
\begin{equation}
    \label{eq:head-rank}
    s_\lambda:=\operatorname{rank}(P_H)
\end{equation}
is finite.
Then, let us define the head and tail components of the gradient
\begin{equation}
    \label{eq:head-tail-gradients}
    g_H:=P_Hg,
    \qquad
    g_T:=P_Tg,
\end{equation}
the head and tail covariance operators
\begin{equation}
    \label{eq:head-tail-covariances}
    C_H:=P_HCP_H,
    \qquad
    C_T:=P_TCP_T,
\end{equation}
and their empirical analogues
\begin{equation}
    \label{eq:empirical-head-tail-covariances}
    \hat C_H:=P_H\hat C P_H,
    \qquad
    \hat C_T:=P_T\hat C P_T,
    \qquad
    \hat C_{HT}:=P_H\hat C P_T.
\end{equation}
The continuous cross covariance vanishes because $P_H$ and $P_T$ are spectral
projectors of $C$, namely $P_HCP_T=0$.

\begin{lemma}[Properties of the set $\calE_\delta$]
    \label{lem:ridge-event-consequences}
    Fix $\lambda>0$ and $\delta\in(0,1)$.
    On the event $\calE_\delta$, the following estimates hold.
    \begin{enumerate}[label=(\roman*)]
        \item Equivalence of the empirical and continuous norms on the head component:
              for every bounded linear operator $A$ on $\ell^2(d)$,
              \begin{equation}
                  \label{eq:head-norm-equivalence}
                  \left|
                  \|Ag_H\|_M^2
                  -
                  \|Ag_H\|_{L^2_\mu(\Gamma)}^2
                  \right|
                  \leq
                  2\delta\,
                  \|Ag_H\|_{L^2_\mu(\Gamma)}^2.
              \end{equation}
        \item Bound on the spectral tail
              \begin{equation}
                  \label{eq:tail-spectral-control}
                  \|\hat C_T-C_T\|_2
                  \leq
                  2\delta\lambda.
              \end{equation}
        \item Bound on the head-tail cross covariance
              \begin{equation}
                  \label{eq:head-tail-cross-control}
                  \|C_H^{\dagger/2}\hat C_{HT}\|_F
                  \leq
                  2\delta\sqrt{s_\lambda\lambda},
              \end{equation}
              where $C_H^\dagger$ denotes the Moore--Penrose inverse on $\operatorname{Ran}(P_H)$ and $C_H^{\dagger/2}$ denotes its square root.
    \end{enumerate}
    Consequently, for every rank-$r$ orthogonal projector $\Pi$, denoting $c_T:=\operatorname{tr}(C_T)-\operatorname{tr}(\hat C_T)$,
    one has
    \begin{equation}
        \label{eq:shifted-error-arbitrary-lambda}
        \left|
        \hat{\err}(\Pi)^2
        +
        c_T
        -
        \err(\Pi)^2
        \right|
        \leq
        4\delta\,\err_H(\Pi)^2
        +
        2\delta\,r\lambda
        +
        2\delta\,s_\lambda\lambda,
    \end{equation}
    where
    \begin{equation}
        \label{eq:head-error-definition}
        \err_H(\Pi)^2
        :=
        \operatorname{tr}\left((I-\Pi)C_H\right).
    \end{equation}
\end{lemma}

\begin{proof}
    Let $E:=\hat C-C$.
    On $\calE_\delta$ we have the operator inequality
    \begin{equation}
        \label{eq:ridge-loewner-event}
        -\delta C_\lambda
        \preceq
        E
        \preceq
        \delta C_\lambda
        \qquad
        \text{on }\mathcal R.
    \end{equation}

    We first prove the head estimate (i).
    Since $\sigma_j^2>\lambda$ on $\mathcal I_H$, we have
    \begin{equation}
        \label{eq:head-ridge-controlled-by-covariance}
        P_H C_\lambda P_H
        =
        C_H+\lambda P_H
        \preceq
        2C_H.
    \end{equation}
    Hence, by \eqref{eq:ridge-loewner-event},
    \begin{equation}
        \label{eq:head-operator-control}
        -2\delta C_H
        \preceq
        P_HEP_H
        \preceq
        2\delta C_H.
    \end{equation}
    Testing this inequality against $A^\adj A$ gives
    \begin{equation}
        \left|
        \operatorname{tr}\left(A(P_HEP_H)A^\adj\right)
        \right|
        \leq
        2\delta\operatorname{tr}(AC_HA^\adj),
    \end{equation}
    which is exactly \eqref{eq:head-norm-equivalence}.

    We next prove the tail estimate (ii).
    Let $x\in\operatorname{Ran}(P_T)$ with $\|x\|_{\ell^2(d)}=1$.
    Since $\sigma_j^2\leq\lambda$ on $\mathcal I_T$,
    \begin{equation}
        x^\adj Cx\leq \lambda.
    \end{equation}
    Therefore, using \eqref{eq:ridge-loewner-event},
    \begin{equation}
        \left|
        x^\adj Ex
        \right|
        \leq
        \delta\,x^\adj(C+\lambda I_{\mathcal R})x
        \leq
        2\delta\lambda.
    \end{equation}
    Taking the supremum over unit vectors in $\operatorname{Ran}(P_T)$ gives
    \eqref{eq:tail-spectral-control}.

    For the cross term, observe that $P_HCP_T=0$, so
    \begin{equation}
        \label{eq:cross-block-is-error}
        \hat C_{HT}=P_HEP_T.
    \end{equation}
    The event $\calE_\delta$ implies
    \begin{equation}
        \label{eq:whitened-cross-block}
        \left\|
        (C_H+\lambda P_H)^{\dagger/2}
        \hat C_{HT}
        (C_T+\lambda P_T)^{\dagger/2}
        \right\|_2
        \leq
        \delta.
    \end{equation}
    Indeed, $P_H$ and $P_T$ commute with
    $C_\lambda$ and $C_\lambda^{-1/2}$ since they are spectral projectors of $C$, hence the left-hand side in
    \eqref{eq:whitened-cross-block} is
    $\|P_HC_\lambda^{-1/2}EC_\lambda^{-1/2}P_T\|_2$, which is bounded by
    $\delta$ on $\calE_\delta$.
    Since $C_H\succeq \lambda P_H$ and $C_T\preceq \lambda P_T$,
    \begin{equation}
        \label{eq:cross-factor-bounds}
        \left\|
        C_H^{\dagger/2}(C_H+\lambda P_H)^{1/2}
        \right\|_2
        \leq
        \sqrt{2},
        \qquad
        \left\|
        (C_T+\lambda P_T)^{1/2}
        \right\|_2
        \leq
        \sqrt{2\lambda}.
    \end{equation}
    Therefore,
    \begin{equation}
        \label{eq:cross-operator-bound}
        \begin{split}
            \|C_H^{\dagger/2}\hat C_{HT}\|_2
             & =
            \left\|
            C_H^{\dagger/2}(C_H+\lambda P_H)^{1/2}
            (C_H+\lambda P_H)^{\dagger/2}
            \hat C_{HT}
            (C_T+\lambda P_T)^{\dagger/2}
            (C_T+\lambda P_T)^{1/2}
            \right\|_2 \\
             & \leq
            2\delta\sqrt{\lambda}.
        \end{split}
    \end{equation}
    Finally, note that the range of $C_H^{\dagger/2}\hat C_{HT}$ lies in
    $\operatorname{Ran}(P_H)$, hence its rank is at most $s_\lambda$.
    This yields the bound
    \begin{equation}
        \|C_H^{\dagger/2}\hat C_{HT}\|_F
        \leq
        \sqrt{s_\lambda}\,
        \|C_H^{\dagger/2}\hat C_{HT}\|_2
        \leq
        2\delta\sqrt{s_\lambda\lambda},
    \end{equation}
    which proves \eqref{eq:head-tail-cross-control}.

    It remains to derive \eqref{eq:shifted-error-arbitrary-lambda}.
    Fix a rank-$r$ orthogonal projector $\Pi$ and write $Q:=I-\Pi$.
    Since $C=C_H+C_T$ and the continuous cross term vanishes,
    \begin{equation}
        \label{eq:continuous-error-head-tail}
        \err(\Pi)^2
        =
        \operatorname{tr}(QC_H)
        +
        \operatorname{tr}(QC_T).
    \end{equation}
    On the other hand,
    \begin{equation}
        \label{eq:empirical-error-head-tail}
        \hat{\err}(\Pi)^2
        =
        \operatorname{tr}(Q\hat C_H)
        +
        \operatorname{tr}(Q\hat C_T)
        +
        2 \operatorname{tr}(Q\hat C_{HT}),
    \end{equation}
    where we used the fact that $\operatorname{tr}(Q\hat C_{HT})= \operatorname{tr}(Q\hat C_{TH})$ by the cyclic property of the trace and the self-adjointness of $Q$.
    The head estimate \eqref{eq:head-norm-equivalence}, applied with $A=Q$,
    gives
    \begin{equation}
        \label{eq:head-error-bound}
        \left|
        \operatorname{tr}(Q\hat C_H)-\operatorname{tr}(QC_H)
        \right|
        \leq
        2\delta\,\err_H(\Pi)^2.
    \end{equation}
    For the tail, the scalar shift $c_T$ removes the fluctuation of the total tail
    trace:
    \begin{equation}
        \label{eq:tail-shift-identity}
        \operatorname{tr}(Q\hat C_T)
        +
        c_T
        -
        \operatorname{tr}(QC_T)
        =
        -\operatorname{tr}\left(\Pi(\hat C_T-C_T)\right).
    \end{equation}
    Therefore, by \eqref{eq:tail-spectral-control},
    \begin{equation}
        \label{eq:tail-error-bound}
        \left|
        \operatorname{tr}(Q\hat C_T)
        +
        c_T
        -
        \operatorname{tr}(QC_T)
        \right|
        \leq
        r\,\|\hat C_T-C_T\|_2
        \leq
        2\delta r\lambda.
    \end{equation}
    Finally, for the cross term we write
    \begin{equation}
        \operatorname{tr}(Q\hat C_{HT})
        =
        \operatorname{tr}\left(
        Q C_H^{1/2}
        \left(C_H^{\dagger/2}\hat C_{HT}\right)
        \right).
    \end{equation}
    By Cauchy--Schwarz in the Frobenius inner product and
    \eqref{eq:head-tail-cross-control},
    \begin{equation}
        \label{eq:cross-error-bound-first}
        \left|
        \operatorname{tr}(Q\hat C_{HT})
        \right|
        \leq
        \|QC_H^{1/2}\|_F
        \|C_H^{\dagger/2}\hat C_{HT}\|_F
        \leq
        2\delta
        \sqrt{\err_H(\Pi)^2s_\lambda\lambda}.
    \end{equation}
    Hence
    \begin{equation}
        \label{eq:cross-error-bound}
        \left|
        2\operatorname{tr}(Q\hat C_{HT})
        \right|
        \leq
        4\delta
        \sqrt{\err_H(\Pi)^2s_\lambda\lambda}
        \leq
        2\delta\,\err_H(\Pi)^2
        +
        2\delta\,s_\lambda\lambda.
    \end{equation}
    Combining \eqref{eq:head-error-bound}, \eqref{eq:tail-error-bound}, and
    \eqref{eq:cross-error-bound} proves \eqref{eq:shifted-error-arbitrary-lambda}.
\end{proof}

The scalar shift $c_T$ in Lemma~\ref{lem:ridge-event-consequences} is independent
of $\Pi$, and therefore it does not affect the empirical minimizer.
This allows the shifted comparison to be converted into an error bound for the
empirical active subspace.
The terms involving $r\lambda$ and $s_\lambda\lambda$ quantify the residual
effect of working at {regularization} scale $\lambda${.
The head--tail mechanism parallels ridge-leverage analyses for finite matrices
and continuous random features \cite{cohen2017ridge,musco2017kernel}.
Projector-independent shifts of this kind are standard in low-rank sketching \cite{musco2020pcp}}.

\begin{theorem}[Quasi-optimality of the empirical active subspace]
    \label{thm:slas-ridge-quasi-opt}
    Let $r\geq 1$, let $\lambda>0$, and assume
    $K_{\rho,\lambda}(g)<\infty$.
    Let $\hat\Pi_r$ be the rank-$r$ empirical active subspace projector computed
    from \eqref{eq:emp-cov-is}.
    There exist universal constants $c_0,c_1>0$ such that, for every $\varepsilon\in(0,1)$ and $\eta\in(0,1)$, if
    \begin{equation}
        \label{eq:samp-ineq-eta}
        M
        \geq
        \frac{K_{\rho,\lambda}^+(g)}{c_0 \varepsilon^2}
        \log\left(
        \frac{c_1(1+d_{\mathrm{eff}}(\lambda))}{\eta}
        \right),
    \end{equation}
    then, with probability at least $1-\eta$,
    \begin{equation}
        \label{eq:quasi-opt-general-lambda}
        \err(\hat\Pi_r)^2
        \leq
        \frac{4+\varepsilon}{4-\varepsilon} \err(\Pi_r)^2
        +
        \frac{\varepsilon}{4 - \varepsilon} (r+s_\lambda)\lambda.
    \end{equation}
    In particular, if $\err(\Pi_r)^2>0$ and
    $\lambda=\lambda_r$, with $\lambda_r$ as in
    \eqref{eq:canonical-ridge-scale}, then the same sampling condition
    implies
    \begin{equation}
        \label{eq:quasi-opt-high-probability}
        \err(\hat\Pi_r)^2
        \leq
        \frac{1+\varepsilon}{1-\varepsilon}
        \err(\Pi_r)^2.
    \end{equation}
    If instead $\err(\Pi_r)=0$, let $k:=\operatorname{rank}(C)$, so that
    $1\leq k\leq r$.
    If $0<\lambda\leq\sigma_k^2$, then the same sampling condition implies,
    with probability at least $1-\eta$,
    \begin{equation*}
        \err(\hat\Pi_r)
        =
        0.
    \end{equation*}
\end{theorem}

\begin{proof}
    By Lemma~\ref{lem:ridge-good-event}, with $\delta=\varepsilon/16$, the sampling
    condition \eqref{eq:samp-ineq-eta} implies that
    $\calE_\delta$ holds with probability at least $1-\eta$, after modifying the
    universal constant $c_0$.

    On $\calE_\delta$, \eqref{eq:shifted-error-arbitrary-lambda} and
    $\err_H(\Pi)^2\leq\err(\Pi)^2$ imply that, for every rank-$r$ orthogonal
    projector $\Pi$,
    \begin{equation}
        \label{eq:shifted-error-epsilon-lambda}
        \left|
        \hat{\err}(\Pi)^2
        +
        c_T
        -
        \err(\Pi)^2
        \right|
        \leq
        \frac{\varepsilon}{4}\err(\Pi)^2
        +
        \frac{\varepsilon}{8}(r+s_\lambda)\lambda.
    \end{equation}
    Since $\hat\Pi_r$ minimizes $\hat{\err}(\Pi)^2$ over rank-$r$ projectors,
    and since $c_T$ is independent of $\Pi$, we have
    \begin{equation}
        \label{eq:empirical-optimality-with-shift}
        \hat{\err}(\hat\Pi_r)^2+c_T
        \leq
        \hat{\err}(\Pi_r)^2+c_T.
    \end{equation}
    Applying \eqref{eq:shifted-error-epsilon-lambda} first to $\hat\Pi_r$ and
    then to $\Pi_r$ gives
    \begin{equation}
        \begin{split}
            \left(1-\frac{\varepsilon}{4}\right)\err(\hat\Pi_r)^2
            -
            \frac{\varepsilon}{8}(r+s_\lambda)\lambda
             & \leq
            \hat{\err}(\hat\Pi_r)^2+c_T \\
             & \leq
            \hat{\err}(\Pi_r)^2+c_T     \\
             & \leq
            \left(1+\frac{\varepsilon}{4}\right)\err(\Pi_r)^2
            +
            \frac{\varepsilon}{8}(r+s_\lambda)\lambda.
        \end{split}
    \end{equation}
    Rearranging proves \eqref{eq:quasi-opt-general-lambda}.

    Suppose first that $\err(\Pi_r)>0$ and
    $\lambda=\lambda_r$, with $\lambda_r$ as in
    \eqref{eq:canonical-ridge-scale}.
    In this case, at most $r$ head eigenvalues occur
    among the first $r$ indices.
    For indices $j>r$, each head eigenvalue contributes more than
    $\lambda_r$ to $\err(\Pi_r)^2$, hence there can be at most
    $\err(\Pi_r)^2/\lambda_r=r$ such indices.
    Thus $s_{\lambda_r}\leq 2r$, and consequently
    \begin{equation}
        (r+s_{\lambda_r})\lambda_r
        \leq
        3\err(\Pi_r)^2.
    \end{equation}
    Substituting this into \eqref{eq:quasi-opt-general-lambda} gives
    \begin{equation}
        \err(\hat\Pi_r)^2
        \leq
        \frac{1+\varepsilon}{1-\frac{\varepsilon}{4}}\err(\Pi_r)^2
        \leq
        \frac{1+\varepsilon}{1-\varepsilon}\err(\Pi_r)^2,
    \end{equation}
    which proves \eqref{eq:quasi-opt-high-probability}.

    Finally, suppose that $\err(\Pi_r)=0$ and let
    $k=\operatorname{rank}(C)$.
    The tail identity \eqref{eq:population-active-subspace-tail} gives
    $1\leq k\leq r$.
    Since $g(\yb)\in\mathcal R$ for $\mu$-almost every $\yb$ and
    $w=\dif\mu/\dif\rho$, each weighted covariance summand in
    \eqref{eq:emp-cov-is} has range in $\mathcal R$ almost surely.
    Hence,
    \begin{equation*}
        \operatorname{Ran}(\hat C)
        \subseteq
        \mathcal R.
    \end{equation*}
    If $0<\lambda\leq\sigma_k^2$, then, on $\calE_\delta$,
    \begin{equation*}
        \begin{split}
            \hat C
             & \succeq
            C-\delta\left(C+\lambda I_{\mathcal R}\right) \\
             & \succeq
            (1-2\delta)\sigma_k^2 I_{\mathcal R}.
        \end{split}
    \end{equation*}
    Thus, $\hat C$ is positive definite on $\mathcal R$.
    Together with the range inclusion above, this yields
    $\operatorname{Ran}(\hat C)=\mathcal R$.
    Since $k\leq r$, the range of every leading rank-$r$ spectral projector of $\hat C$
    contains $\mathcal R$, and therefore
    \begin{equation*}
        \err(\hat\Pi_r)^2
        =
        \operatorname{tr}\left((I-\hat\Pi_r)C\right)
        =
        0.
    \end{equation*}
\end{proof}

The function $k_{\rho,\lambda}$ in \eqref{eq:ridge-christoffel} is the continuous
analogue of a ridge leverage score.
In the eigenbasis of the gradient covariance, it assigns the point $\yb$ the
weight
\begin{equation}
    k_{\rho,\lambda}(\yb)
    =
    w(\yb)
    \sum_{j\geq 1}
    \frac{\sigma_j^2}{\sigma_j^2+\lambda}\eta_j(\yb)^2.
\end{equation}
Thus, directions with $\sigma_j^2\gg\lambda$ are treated almost as in the
unregularized inverse Christoffel function, while directions with
$\sigma_j^2\ll\lambda$ are downweighted proportionally to
$\sigma_j^2/\lambda$.
Moreover, the definition \eqref{eq:ridge-coherence-effective-dimension} implies the lower bound on the coherence $K_{\rho,\lambda}(g) \geq d_{\mathrm{eff}}(\lambda)$.
The {regularized} Christoffel measure is obtained by normalizing the unweighted {regularized} inverse Christoffel function $k_{\mu,\lambda}$:
\begin{equation}
    \label{eq:optimal-sampling-measure}
    \dif \rho_\lambda^\star(\yb)
    =
    \frac{\|C_\lambda^{-1/2}g(\yb)\|_{\ell^2(d)}^2}{d_{\mathrm{eff}}(\lambda)}
    =
    \frac{1}{d_{\mathrm{eff}}(\lambda)} \sum_{j\geq 1} \frac{\sigma_j^2}{\sigma_j^2+\lambda}\eta_j(\yb)^2
    \dif \mu(\yb).
\end{equation}
Then $K_{\rho_\lambda^\star,\lambda}(g)=d_{\mathrm{eff}}(\lambda)$.
Thus, for a fixed $\lambda$, the {regularized} Christoffel measure $\rho_\lambda^\star$ minimizes the
coherence factor in the sampling condition \eqref{eq:samp-ineq-eta}.
When $\err(\Pi_r)>0$, choosing the canonical {regularization} scale
\eqref{eq:canonical-ridge-scale} and using
\eqref{eq:effective-dimension-bound} gives
$d_{\mathrm{eff}}(\lambda_r)\leq 2r$, and the sampling condition becomes
$M=\calO(r\log(r))$ up to constants and probability factors.
However, the {regularized} Christoffel measure depends on the unknown covariance eigenfunctions, and therefore it is not directly useful for sampling.
Moreover, if the eigenfunctions were known, then sampling would involve evaluations of $g$, which poses further problems in settings where $g$ is computationally expensive to query.

Let us further discuss the role of $\lambda$, which is not only technical.
Indeed, it prevents the sampling condition from being controlled by the smallest
retained covariance eigenvalue.
To see this, suppose for simplicity that one only assumes the weighted
gradient bound
\begin{equation}
    \label{eq:weighted-gradient-bound}
    B_\rho(\nabla f)^2
    :=
    \rho\text{-}\operatorname*{ess\,sup}_{\yb\in\Gamma}
    w(\yb)\|\nabla f(\yb)\|_{\ell^2(d)}^2
    <
    \infty.
\end{equation}
Then the {regularized} inverse Christoffel function satisfies the elementary estimate
\begin{equation}
    \label{eq:ridge-coherence-bounded-gradient}
    k_{\rho,\lambda}(\yb)
    =
    w(\yb)\|C_\lambda^{-1/2}\nabla f(\yb)\|_{\ell^2(d)}^2
    \leq
    \frac{w(\yb)\|\nabla f(\yb)\|_{\ell^2(d)}^2}{\lambda},
\end{equation}
and hence
\begin{equation}
    \label{eq:Klambda-bounded-gradient}
    K_{\rho,\lambda}(g)
    \leq
    \frac{B_\rho(g)^2}{\lambda}.
\end{equation}
In contrast, an unregularized rank-$r$ analysis based only on the top active
directions typically contains denominators of size $\sigma_r^2$.
Under the same boundedness assumption, this leads to coherence bounds of the
form $B_\rho(\nabla f)^2/\sigma_r^2$.
In the positive-tail regime, this is unfavorable when the spectrum is flat or
slowly decaying near the truncation index, since $\sigma_r^2$ can be much
smaller than the canonical {regularization} scale
$\lambda_r=\err(\Pi_r)^2/r$.
The canonical {regularization} scale \eqref{eq:canonical-ridge-scale} replaces the sensitivity to
$\sigma_r^2$ by sensitivity to the active subspace tail energy normalized by the target rank.
The resulting estimate is therefore stable under small covariance eigenvalues
below the accuracy scale.

This should be distinguished from an unregularized Monte Carlo covariance estimate.
A direct bound on $\|\hat C-C\|_2$ gives an additive covariance error that
decays at the usual rate $M^{-1/2}$ under boundedness assumptions.
Such a bound does not distinguish between energetic and negligible spectral
directions.
To use it for active subspaces, one must make the additive covariance error
small compared with the target projection error, or impose additional
spectral separation assumptions.
The regularized estimate in Lemma~\ref{lem:ridge-good-event} is instead relative
to $C+\lambda I_{\mathcal R}$: directions above the {regularization} scale are preserved, while directions below the {regularization} scale are controlled only up to
the additive scale $\lambda$.
This is precisely the scale needed in Lemma~\ref{lem:ridge-event-consequences} to
control the tail and the empirical head--tail interaction.

Finally, let us further comment on how our result compares to known estimates in the AS literature.
{In particular, we} consider the intrinsic-dimension
analysis of the single-fidelity active subspace estimator in
\cite{lam2020multifidelity}.
In that work, the AS error analysis is based on the additive bound
\begin{equation} \label{eq:add-err-bound-marzouk}
    \err(\hat\Pi_r)^2
    \leq
    \err(\Pi_r)^2
    +
    2r\|\hat C-C\|_2,
\end{equation}
where $\hat C$ is the standard (unweighted and unregularized) covariance estimator.
The above follows from the empirical optimality of $\hat \Pi_r$ and the identity
\begin{equation}
    \hat{\err}(\Pi)^2
    +
    \operatorname{tr}(C)-\operatorname{tr}(\hat C)
    -
    \err(\Pi)^2
    =
    -\operatorname{tr}(\Pi (\hat C-C)),
\end{equation}
which implies
\begin{equation}
    \left|
    \hat{\err}(\Pi)^2
    +
    \operatorname{tr}(C)-\operatorname{tr}(\hat C)
    -
    \err(\Pi)^2
    \right|
    \leq
    r\|\hat C-C\|_2.
\end{equation}
The difference with our analysis is hence that the bound \eqref{eq:add-err-bound-marzouk} requires the control of the operator norm of the error on the whole covariance operator
before it is converted into a squared active subspace error.
Indeed, under a boundedness assumption of the form
\begin{equation}
    \mu\text{-}\operatorname*{ess\,sup}_{\yb\in\Gamma}
    \frac{\|\nabla f(\yb)\|_{\ell^2(d)}^2}{\operatorname{tr}(C)}
    \leq
    \beta^2
\end{equation}
their single-fidelity result gives
\begin{equation}
    \|\hat C-C\|_2
    \lesssim
    \varepsilon\|C\|_2
\end{equation}
with a sample complexity
\begin{equation}
    M \gtrsim\frac{\beta^2 \delta_C \log(\delta_C)}{\varepsilon^{2}},
\end{equation}
where $\delta_C:=\operatorname{trace}(C)/\|C\|_2$ denotes the intrinsic dimension of the covariance.
Consequently, to guarantee a projection-error tolerance $\tol$, let
\begin{equation}
    r_\tol
    :=
    \min\left\{
    r\geq 1:\err(\Pi_r)^2\leq \tol^2
    \right\}.
\end{equation}
The additive covariance route must then enforce
\begin{equation}
    \|\hat C-C\|_2\lesssim \frac{\tol^2}{r_\tol},
\end{equation}
which, for the relative covariance estimate above, corresponds to
\begin{equation}
    \varepsilon_{\mathrm{cov}}
    \lesssim
    \frac{\tol^2}{r_\tol\|C\|_2}.
\end{equation}
Substituting this tolerance into the intrinsic-dimension covariance bound gives
the sample complexity
\begin{equation}
    M
    \gtrsim
    \beta^2 \delta_C \log(\delta_C)
    \left(
    \frac{r_\tol\|C\|_2}{\tol^2}
    \right)^2.
\end{equation}
Instead, our regularized analysis can be applied directly at the
tolerance-dependent {regularization} scale
\begin{equation}
    \lambda_\tol:=\frac{\tol^2}{r_\tol}.
\end{equation}
Since $\err(\Pi_{r_\tol})^2\leq \tol^2$, the elementary spectral estimates
\begin{equation}
    d_{\mathrm{eff}}(\lambda_\tol)\leq 2r_\tol,
    \qquad
    s_{\lambda_\tol}\leq 2r_\tol,
\end{equation}
which are used again in the proof of
Theorem~\ref{thm:smoothness-tolerance-complexity}, show that the additive
{regularization-scale} term in \eqref{eq:quasi-opt-general-lambda} remains of order
$\tol^2$.
Thus, this choice of $\lambda_\tol$ makes
Theorem~\ref{thm:slas-ridge-quasi-opt} yield a squared population error
$\err(\hat\Pi_{r_\tol})^2\lesssim \tol^2$ for the empirical active subspace, up to the fixed
quasi-optimality constant.
Under the bounded-gradient estimate \eqref{eq:Klambda-bounded-gradient}, the
corresponding sampling condition scales as
\begin{equation}
    \label{eq:samp-complexity-bounded-gradient}
    M
    \gtrsim
    \frac{B_\rho(\nabla f)^2}{\tol^2} r_\tol \log r_\tol
\end{equation}
up to constants and probability factors.

However, under the sole boundedness assumption our analysis still gives a Monte Carlo-type
dependence on the tolerance.
The advantage of the formulation through the {regularized} inverse Christoffel function is that
it can be sharpened when additional structure of $\nabla f$ is available.
In the next section we therefore impose stronger assumptions on the gradient,
with the goal of improving the bound on $K_{\rho,\lambda}(\nabla f)$ beyond the crude estimate
$K_{\rho,\lambda}(\nabla f)\le B_\rho(\nabla f)^2/\lambda$.

\section{Sample complexity from smoothness}
\label{sec:sample-complexity-smoothness}

We keep the notation of the previous sections.
In particular, $g=\nabla f$ denotes the gradient field,
$C$ denotes its covariance operator,
and $C_\lambda$ denotes the regularized covariance operator in
\eqref{eq:regularized-covariance}.
The quasi-optimality estimate in Theorem~\ref{thm:slas-ridge-quasi-opt} separates the
sampling condition, controlled by the {regularized} coherence constant
$K_{\rho,\lambda}(g)$ in \eqref{eq:ridge-coherence-effective-dimension}, from
the {regularization-scale} error term.
The purpose of this section is to bound this coherence quantity, and the
corresponding active subspace tail energy $\err(\Pi_r)^2$ from smoothness
assumptions on the gradient field.

The elementary estimate \eqref{eq:Klambda-bounded-gradient} gives the
worst-case behavior $K_{\rho,\lambda}(g)\lesssim \lambda^{-1}$.
This estimate is robust with respect to small covariance eigenvalues,
but it does not exploit regularity of the map $\yb\mapsto g(\yb)$.
Here we use a reproducing kernel Hilbert space
and separate two related effects: {i) boundedness} of the {RKHS covariance} and spectral summability of the RKHS
embedding control the {regularized} coherence; {ii) the} same spectral summability, combined with Schatten summability of the
    {RKHS covariance} across the coordinate index, controls the decay of the
active subspace tail energy.
The combination gives an a priori sample-complexity estimate as a function of
the tolerance.

Let $(\calH,\inp{\cdot}{\cdot}_{\calH})$ be a real RKHS on $\Gamma$ with
measurable reproducing kernel $\kappa$.
Assume that {its kernel diagonal is positive $\mu$-almost everywhere and
        integrable:}
\begin{equation*}
    {\kappa(\yb,\yb)>0
        \quad
        \text{for $\mu$-a.e. } \yb\in\Gamma,
        \qquad
        \int_\Gamma \kappa(\yb,\yb)\,\dif\mu(\yb)<\infty.}
\end{equation*}
{The diagonal integrability condition makes} the canonical embedding
$\imath:\calH\to L^2_\mu(\Gamma)$ {Hilbert--Schmidt, and hence} compact.
    {Assume in addition that $\imath$ is} injective, and identify $\calH$ with its
image under $\imath$.
The associated integral operator $T:=\imath\imath^*$ on
$L^2_\mu(\Gamma)$ is given by
\begin{equation}
    (T h)(\yb)
    :=
    \int_\Gamma
    \kappa(\yb,\zb)h(\zb)\,\dif\mu(\zb).
\end{equation}
{Let} $\{(\gamma_n,\phi_n)\}_{n\geq 1}$ {denote its nonzero eigenpairs}, where
$\gamma_n>0${, $\gamma_n\to0$ as $n\to\infty$,} and
$\{\phi_n\}_{n\geq 1}$ is an orthonormal basis of {the closure of
        $\imath(\calH)$ in} $L^2_\mu(\Gamma)$.
In these spectral coordinates,
\begin{equation}
    \label{eq:smoothness-rkhs-spectral-characterization}
    \calH
    =
    \left\{
    p=\sum_{n\geq 1}p_n\phi_n\in L^2_\mu(\Gamma)
    :
    \norm{p}_{\calH}^2
    :=
    \sum_{n\geq 1}\gamma_n^{-1}\abs{p_n}^2
    <
    \infty
    \right\},
    \qquad
    p_n
    :=
    \inp{p}{\phi_n}_{L^2_\mu(\Gamma)}.
\end{equation}
The functions $\{\gamma_n^{1/2}\phi_n\}_{n\geq1}$ form an orthonormal basis
of $\calH$, so the singular values of $\imath$ are
$\{\gamma_n^{1/2}\}_{n\geq1}$.
With the corresponding pointwise representatives of the functions $\phi_n$,
\begin{equation}
    \label{eq:smoothness-mercer-kernel}
    \kappa(\yb,\zb)
    =
    \sum_{n\geq 1}
    \gamma_n\phi_n(\yb)\phi_n(\zb)
\end{equation}
pointwise on $\Gamma\times\Gamma$.
For $\theta\in(0,1]$, define the $\theta$-kernel diagonal and its trace by
\begin{equation}
    \label{eq:smoothness-theta-diagonal}
    \kappa^\theta(\yb)
    :=
    \sum_{n\geq 1}
    \gamma_n^\theta\phi_n(\yb)^2,
    \qquad
    Z_\theta
    :=
    \sum_{n\geq 1}
    \gamma_n^\theta.
\end{equation}
We call $\theta$ an admissible exponent if
\begin{equation}
    \label{eq:smoothness-admissible-scale}
    0<\kappa^\theta(\yb)<\infty
    \quad
    \text{for $\mu$-a.e. } \yb\in\Gamma,
    \qquad
    Z_\theta<\infty.
\end{equation}
{In particular,} ${\theta=1}$ {is always admissible}.
{Indeed}, {\eqref{eq:smoothness-mercer-kernel} and Tonelli's theorem give}
\begin{equation*}
    {\kappa^1(\yb)
        =
        \kappa(\yb,\yb),
        \qquad
        Z_1}
    =
    \int_\Gamma {\kappa(\yb,\yb)\,\dif\mu}(\yb)
    {<
        \infty.}
\end{equation*}
{Notice that, up to restricting $\Gamma$ to a set of full $\mu$-measure, the
conditions in \eqref{eq:smoothness-admissible-scale} are equivalent to requiring
that}
\begin{equation*}
    {\calH^\theta
        :=
        \left\{
        p=\sum_{n\geq1}p_n\phi_n\in L^2_\mu(\Gamma)
        :
        \norm{p}_{\calH^\theta}^2
        :=
        \sum_{n\geq1}\gamma_n^{-\theta}\abs{p_n}^2
        <
        \infty
        \right\}}
\end{equation*}
{is an RKHS whose reproducing kernel has strictly positive, integrable diagonal
$\kappa^\theta$.
Indeed, Tonelli's theorem gives}
\begin{equation*}
    {\int_\Gamma \kappa^\theta}(\yb)\,\dif\mu(\yb)
        =
        {Z_\theta}.
\end{equation*}

\subsection{\texorpdfstring{{Regularized}}{Regularized} coherence from a bounded
    \texorpdfstring{{RKHS covariance}}{RKHS covariance}}

{For} ${v\in \ell^2(d)}$, {define the} scalar projection
\begin{equation}
    \label{eq:smoothness-scalar-projection}
    p_v(\yb)
    :=
    \inp{v}{g(\yb)}_{\ell^2(d)}.
\end{equation}
{We make the following smoothness assumption on $g$.}

\begin{assumption}{(Bounded RKHS covariance)
        \label{ass:smoothness-covariance-bounded}
        Every scalar projection $p_v$} belongs to $\calH$, and {there exists a bounded
            positive operator $H\in\mathcal L(\ell^2(d))$ such that}
    \begin{equation}
        {\label{eq:smoothness-covariance-operator}}
        \inp{v}{{Hu}}_{\ell^2(d)}
        {=
            \inp{p_v}{p_u}_{\calH},
            \qquad
            u,v\in\ell^2(d)}.
    \end{equation}
    {We call $H$ the RKHS covariance of $g$ associated with $\calH$.}
\end{assumption}

{\noindent
The above assumption is equivalent to requiring} that the scalar-projection operator
\begin{equation}
    G:\ell^2(d)\to\calH,
    \qquad
    Gv
    :=
    p_v,
\end{equation}
is bounded and satisfies
\begin{equation}
    H
    =
    G^*G,
    \qquad
    C
    =
    (\imath G)^*(\imath G).
\end{equation}
{Moreover, we have}
\begin{equation*}
    {\sup_{\norm{v}_{\ell^2(d)}=1} \norm{p_v}_{\mathcal H}^2 \leq \norm{H}_{\mathcal L(\ell^2(d))}.}
\end{equation*}
Whenever point values are used below, we take the compatible representative
$g(\yb):=G^*\kappa(\cdot,\yb)$.
It agrees with the original gradient field $\mu$-a.e. and ensures that
$(Gv)(\yb)=\inp{v}{g(\yb)}_{\ell^2(d)}$ simultaneously for every $v$.
The standing assumption $g\neq0$ and the identities above imply that
$H\neq0$.

For $\alpha>0$, define the {regularized} RKHS norm
\begin{equation}
    \label{eq:smoothness-ridge-rkhs-norm}
    \norm{p}_{\alpha}^2
    :=
    \norm{p}_{L^2_\mu(\Gamma)}^2
    +
    \alpha\norm{p}_{\calH}^2
    =
    \sum_{n\geq 1}
    \frac{\gamma_n+\alpha}{\gamma_n}
    \abs{\inp{p}{\phi_n}_{L^2_\mu(\Gamma)}}^2.
\end{equation}
The reproducing kernel of this {regularized} RKHS, which we call the {regularized} kernel, is
\begin{equation}
    \label{eq:smoothness-ridge-diagonal-kernel}
    \kappa_\alpha(\yb,\zb)
    :=
    \sum_{n\geq 1}
    \frac{\gamma_n}{\gamma_n+\alpha}
    \phi_n(\yb)\phi_n(\zb).
\end{equation}
For its {regularized} kernel diagonal, we use the shorthand
\begin{equation}
    \label{eq:smoothness-diagonal-shorthand}
    \kappa_\alpha(\yb)
    :=
    \kappa_\alpha(\yb,\yb).
\end{equation}
{The pointwise kernel expansion and the standing nondegeneracy assumption give,
for every $\alpha>0$,}
\begin{equation*}
    {0
        <
        \kappa_\alpha(\yb)
        \leq
        \alpha^{-1}\kappa(\yb,\yb)
        <
        \infty
        \quad
        \text{for $\mu$-a.e. } \yb\in\Gamma.}
\end{equation*}

\begin{lemma}[{Regularized} kernel diagonal bound for the {regularized} coherence]
    \label{prop:smoothness-ridge-coherence}
    Suppose that Assumption~\ref{ass:smoothness-covariance-bounded} holds.
    Let $\lambda>0$ and set
    \begin{equation}
        \label{eq:smoothness-ridge-alpha-lambda}
        \alpha_\lambda
        :=
        \frac{\lambda}{\norm{H}_{\mathcal L(\ell^2(d))}}.
    \end{equation}
    Then, for every sampling measure $\rho$ such that $\mu\ll\rho$,
    \begin{equation}
        \label{eq:smoothness-pointwise-ridge-bound}
        k_{\rho,\lambda}(\yb)
        \leq
        w(\yb)\kappa_{\alpha_\lambda}(\yb)
        \quad
        \text{for $\rho$-a.e. } \yb\in\Gamma.
    \end{equation}
    Consequently,
    \begin{equation}
        \label{eq:smoothness-generic-coherence-bound}
        K_{\rho,\lambda}(g)
        \leq
        \rho\text{-}\operatorname*{ess\,sup}_{\yb\in\Gamma}
        w(\yb)\kappa_{\alpha_\lambda}(\yb).
    \end{equation}
\end{lemma}

\begin{proof}
    Since $g(\yb)\in\mathcal R$ for $\mu$-a.e. $\yb\in\Gamma$, the Hilbert-space
    Rayleigh quotient for $C_\lambda$ gives
    \begin{equation}
        \label{eq:smoothness-ridge-rayleigh}
        \norm{C_\lambda^{-1/2}g(\yb)}_{\ell^2(d)}^2
        =
        \sup_{v\in\mathcal R\setminus\{0\}}
        \frac{
            \abs{\inp{v}{g(\yb)}_{\ell^2(d)}}^2
        }{
            \inp{v}{C_\lambda v}_{\ell^2(d)}
        }.
    \end{equation}
    For fixed $v\in\mathcal R$, the {definition of the regularized norm} and {the
            factorizations $H=G^*G$ and $C=(\imath G)^*(\imath G)$} imply
    \begin{equation}
        \label{eq:smoothness-ridge-norm-identity}
        \norm{p_v}_{\alpha_\lambda}^2
        =
        \inp{v}{Cv}_{\ell^2(d)}
        +
        \alpha_\lambda\inp{v}{Hv}_{\ell^2(d)}.
    \end{equation}
    By the definition of $\alpha_\lambda$,
    \begin{equation}
        \label{eq:smoothness-ridge-norm-domination}
        \norm{p_v}_{\alpha_\lambda}^2
        \leq
        \inp{v}{Cv}_{\ell^2(d)}
        +
        \lambda\norm{v}_{\ell^2(d)}^2
        =
        \inp{v}{C_\lambda v}_{\ell^2(d)}.
    \end{equation}
    The reproducing property in the {regularized} RKHS gives
    \begin{equation}
        \label{eq:smoothness-ridge-point-evaluation}
        \abs{\inp{v}{g(\yb)}_{\ell^2(d)}}^2
        =
        \abs{p_v(\yb)}^2
        \leq
        \kappa_{\alpha_\lambda}(\yb)
        \norm{p_v}_{\alpha_\lambda}^2.
    \end{equation}
    Combining \eqref{eq:smoothness-ridge-rayleigh},
    \eqref{eq:smoothness-ridge-norm-domination}, and
    \eqref{eq:smoothness-ridge-point-evaluation}, and multiplying by
    $w(\yb)$, proves \eqref{eq:smoothness-pointwise-ridge-bound}.
    Taking the essential supremum with respect to $\rho$ gives
    \eqref{eq:smoothness-generic-coherence-bound}.
\end{proof}

Recall that $d_{\mathrm{eff}}(t,\xi)$ denotes the effective dimension of the
supplied sequence $\xi$, whereas the one-argument notation
$d_{\mathrm{eff}}(\lambda)$ denotes the covariance effective dimension by
convention.
For $\alpha>0$, the corresponding kernel effective dimension is
\begin{equation}
    \label{eq:smoothness-ridge-effective-dimension}
    d_{\mathrm{eff}}(\alpha,\gamma)
    :=
    \int_\Gamma
    \kappa_\alpha(\yb)\,\dif\mu(\yb)
    =
    \sum_{n\geq 1}
    \frac{\gamma_n}{\gamma_n+\alpha}.
\end{equation}
{The above satisfies the bound}
\begin{equation*}
    {d_{\mathrm{eff}}(\alpha,\gamma)
    \leq
    \alpha^{-1}Z_1
    <
    \infty.}
\end{equation*}
{We can hence define a sampling measure whose density is proportional to the regularized kernel
diagonal}
\begin{equation}
    {\label{eq:smoothness-ridge-diagonal-proposal}
        \dif\rho_\alpha(\yb)
        :=
        \frac{\kappa_\alpha(\yb)}
        {d_{\mathrm{eff}}(\alpha,\gamma)}
        \,\dif\mu(\yb).}
\end{equation}
{The standing assumptions imply that $\rho_\alpha\sim\mu$.
More generally, let $\rho\sim\mu$ be a sampling measure with reciprocal
density $w=\dif\mu/\dif\rho$.
We say that $\rho$ is quasi-optimal} for the {regularized} kernel diagonal {at scale
        $\alpha$ with constant $c_2\geq1$ if}
\begin{equation}
    {\label{eq:smoothness-quasi-optimal-proposal}
    \rho\text{-}\operatorname*{ess\,sup}_{\yb\in\Gamma}
    w(\yb)\kappa_\alpha(\yb)
    \leq
    c_2d_{\mathrm{eff}}(\alpha,\gamma).}
\end{equation}

\begin{proposition}{(Coherence bound for regularized-kernel quasi-optimal sampling)}
    \label{prop:smoothness-ridge-diagonal-sampling}
    Suppose that Assumption~\ref{ass:smoothness-covariance-bounded} holds.
    Let $\lambda>0$ and let
    $\alpha_\lambda=\lambda/\norm{H}_{\mathcal L(\ell^2(d))}$.
        {Let} ${\rho}$ {be quasi-optimal} for the {regularized} kernel diagonal {at scale}
    $\alpha_\lambda$ {with constant $c_2$}.
    Then
    \begin{equation}
        \label{eq:smoothness-ridge-diagonal-coherence}
        K_{\rho,\lambda}(g)
        \leq
        {c_2d_{\mathrm{eff}}}(\alpha_\lambda,\gamma).
    \end{equation}
\end{proposition}

\begin{proof}
    {Combine Lemma~\ref{prop:smoothness-ridge-coherence} with
        \eqref{eq:smoothness-quasi-optimal-proposal}.}
\end{proof}

{Notice that for every $\gamma_n>0$, $\alpha>0$, and $0<\theta\leq1$}, we have
\begin{equation}
    {\label{eq:smoothness-ridge-theta-scalar-comparison}
        \frac{\gamma_n}{\gamma_n+\alpha}}
    {\leq
        \alpha^{-\theta}\gamma_n^\theta.}
\end{equation}
{Consequently,
\begin{equation}
    \label{eq:smoothness-ridge-theta-dimension-comparison}
    \kappa_{\alpha}(\yb)
    \leq
    \alpha^{-\theta}\kappa^\theta(\yb),
    \quad
    d_{\mathrm{eff}}(\alpha,\gamma)
    \leq
    \alpha^{-\theta}Z_\theta.
\end{equation}
}
{Thus,} ${\kappa^\theta}$ {provides an upper bound for} the {regularized} kernel diagonal
    {whose profile} is {independent of the regularization scale}.
For an admissible exponent $\theta$, define the $\theta$-kernel diagonal measure
\begin{equation}
    \label{eq:smoothness-diagonal-proposal}
    \dif\rho^\theta(\yb)
    =
    \frac{\kappa^\theta(\yb)}{Z_\theta}\,\dif\mu(\yb).
\end{equation}
Admissibility of $\theta$ implies that $\rho^\theta\sim\mu$.

\begin{corollary}{($\theta$-kernel diagonal {sampling})}
    \label{prop:smoothness-diagonal-sampling}
    Suppose that Assumption~\ref{ass:smoothness-covariance-bounded} holds.
    Let $\lambda>0$, let $\theta\in(0,1]$ be an admissible exponent, and set
    $\alpha_\lambda=\lambda/\norm{H}_{\mathcal L(\ell^2(d))}$.
    Then
    \begin{equation}
        \label{eq:smoothness-ridge-to-theta-effective-dimension}
        d_{\mathrm{eff}}(\alpha_\lambda,\gamma)
        \leq
        Z_\theta
        \left(
        \frac{\norm{H}_{\mathcal L(\ell^2(d))}}{\lambda}
        \right)^\theta.
    \end{equation}
    {Moreover}, $\rho{^\theta}$ {is quasi-optimal for the regularized kernel diagonal at
            scale} ${\alpha}_{\lambda}$ {with constant}
    \begin{equation} \label{eq:smoothness-lambda-damped-coherence-constant}
        {c_2
            =
            \frac{\alpha_\lambda^{-\theta}Z_\theta}
            {d_{\mathrm{eff}}(\alpha_\lambda,\gamma)},}
    \end{equation}
    and
    \begin{equation}
        {\label{eq:smoothness-lambda-damped-coherence}}
        K_{\rho{^\theta},\lambda}(g)
        \leq
        Z_\theta
        \left(
        \frac{\norm{H}_{\mathcal L(\ell^2(d))}}{\lambda}
        \right)^\theta.
    \end{equation}
\end{corollary}

\begin{proof}
    The kernel effective-dimension bound
    \eqref{eq:smoothness-ridge-to-theta-effective-dimension} follows from
    \eqref{eq:smoothness-ridge-theta-dimension-comparison} with
    $\alpha=\alpha_\lambda$.
        {If $w^\theta=\dif\mu/\dif\rho^\theta$, then}
    \begin{equation*}
        {w^\theta(\yb)\kappa_{\alpha_\lambda}(\yb)
            \leq
            \alpha_\lambda^{-\theta}Z_\theta
            \quad
            \text{for $\rho^\theta$-a.e. }\yb\in\Gamma,}
    \end{equation*}
    {and normalizing by $d_{\mathrm{eff}}(\alpha_\lambda,\gamma)$ gives \eqref{eq:smoothness-lambda-damped-coherence-constant}.}
    Proposition~\ref{prop:smoothness-ridge-diagonal-sampling} { then gives the coherence estimate}.
\end{proof}

{

\begin{remark}[Choosing between the kernel diagonal measures]
    Sampling from the exact regularized kernel diagonal measure $\rho_{\alpha_\lambda}$ gives the sharper coherence bound based on the kernel effective dimension at scale $\alpha_\lambda$, which satisfies
    \begin{equation*}
        d_{\mathrm{eff}}(\alpha_\lambda,\gamma) \leq \min_{\theta \, \mathrm{admissible}} Z_\theta \left( \frac{\norm{H}_{\mathcal L(\ell^2(d))}}{\lambda} \right)^\theta.
    \end{equation*}
    The price is that $\alpha_\lambda$ must be fixed before sampling and, when $\lambda$ is chosen from a target rank $r$, this requires advance
    estimates of both $r$ and $\norm{H}_{\mathcal L(\ell^2(d))}$.
    Conservative upper estimates are sufficient since they produce a smaller value
    of $\alpha_\lambda$, which preserves the regularized-norm domination, at the cost
    of a larger kernel effective dimension and potentially more samples.
    By contrast, $\rho^\theta$ depends only on the RKHS { embedding singular values} and the
    selected exponent $\theta$, so it can be constructed without estimating
    $r$ or $\norm{H}$ and reused across regularization levels.
    This convenience requires committing to a specific $\theta$ before sampling and a larger coherence upper bound
    $Z_\theta(\norm{H}_{\mathcal L(\ell^2(d))}/\lambda)^\theta$.
    In terms of asymptotic complexity as $\lambda\to0$, one has interest in taking $\theta$ as small as possible while still being admissible. However, this could lead to a large constant $Z_\theta$ in the coherence bound and a poor pre-asymptotic sample complexity.
\end{remark}}

{To turn the preceding coherence bounds into an a priori sampling prescription
for a tolerance $\tol>0$, we need an a priori sufficient rank $r_\tol$
satisfying $\err(\Pi_{r_\tol})^2\leq\tol^2$.
As discussed at the end of Section~\ref{sec:slanal}, controlling the regularization-scale
contribution then requires $\lambda\lesssim\tol^2/r_\tol$, with
$\lambda_\tol=\tol^2/r_\tol$ as the natural choice.
The next subsection supplies the missing rank estimate through an a priori
decay bound for $\err(\Pi_r)^2$.}

\subsection{Active subspace tail energy decay from smoothness}

The preceding subsection used boundedness of $H$ to compare the {regularized} coherence
with the {regularized} kernel diagonal, and spectral summability of the RKHS embedding
to control the resulting bound.
We now combine the same spectral summability with a stronger summability
assumption on the eigenvalues of $H$.
This yields an a priori decay estimate for $\err(\Pi_r)^2$ and therefore
determines how large the target rank must be to meet a prescribed tolerance.

For a Hilbert space $X$ and $0<\beta<\infty$, let $\calS_\beta(X)$ denote the
Schatten-$\beta$ class on $X$, namely the set of compact operators on $X$ whose singular values are $\ell^\beta$-summable.
For $\beta=\infty$, we use the convention
\begin{equation}
    \label{eq:smoothness-schatten-infty}
    \calS_\infty(X)
    :=
    \mathcal L(X),
    \qquad
    \norm{\cdot}_{\calS_\infty(X)}
    :=
    \norm{\cdot}_{\mathcal L(X)}.
\end{equation}
For $0<\beta<1$, $\norm{\cdot}_{\calS_\beta}$ is understood as the usual
Schatten quasi-norm.

\begin{assumption}[Schatten summability of the {RKHS covariance}]
    \label{ass:smoothness-covariance-summability}
    {Assumption~\ref{ass:smoothness-covariance-bounded} holds, and there} exist an
    admissible exponent $\theta\in(0,1]$ and
    $\beta\in(0,\infty]$ such that
    \begin{equation}
        \label{eq:smoothness-tail-assumption}
        H\in\calS_\beta(\ell^2(d)).
    \end{equation}
\end{assumption}

\noindent
{
    The exponent $\theta$ quantifies the parametric regularity encoded by the RKHS
    embedding: the smaller $\theta$ can be chosen while remaining admissible, the
    stronger the parametric regularity.
    The exponent $\beta$ quantifies the additional Schatten summability of
    $H=G^*G$ across the coordinate index.
    Assumption~\ref{ass:smoothness-covariance-summability} therefore combines
    parametric regularity with summability across the coordinate directions, and
    both enter the tail estimate below.}
We use the convention $1/\infty=0$ and define
\begin{equation}
    \label{eq:smoothness-tail-exponent}
    q_{\theta,\beta}
    :=
    \frac{1}{\theta}
    +
    \frac{1}{\beta}
    -
    1.
\end{equation}

\begin{proposition}[Smoothness and active subspace tail energy decay]
    \label{prop:smoothness-active-subspace-tail}
    Suppose that Assumption~\ref{ass:smoothness-covariance-summability} holds {for
            exponents $\theta$ and $\beta$}.
    If $q_{\theta,\beta}>0$, then, for every $r\geq 1$,
    \begin{equation}
        \label{eq:smoothness-active-subspace-tail}
        \err(\Pi_r)^2
        \leq
        Z_\theta^{1/\theta}
        \norm{H}_{\calS_\beta(\ell^2(d))}
        r^{-q_{\theta,\beta}}.
    \end{equation}
\end{proposition}

\begin{proof}
    Since the singular values of $\imath$ are $\gamma_n^{1/2}$,
    we have
    \begin{equation}
        \label{eq:smoothness-embedding-schatten}
        \norm{\imath}_{\calS_{2\theta}(\calH,L^2_\mu(\Gamma))}^2
        =
        \left(
        \sum_{n\geq 1}
        \gamma_n^\theta
        \right)^{1/\theta}
        =
        Z_\theta^{1/\theta}.
    \end{equation}
    The factorization $H=G^*G$ shows that the singular values of $H$ are the
    squared singular values of $G$, and hence
    \begin{equation}
        \label{eq:smoothness-G-H-schatten}
        \norm{G}_{\calS_{2\beta}(\ell^2(d),\calH)}^2
        =
        \norm{H}_{\calS_\beta(\ell^2(d))}.
    \end{equation}
    Define $p\in(0,1)$ by
    \begin{equation}
        \label{eq:smoothness-holder-p}
        \frac{1}{p}
        :=
        \frac{1}{\theta}
        +
        \frac{1}{\beta}
    \end{equation}
    and note that the condition $q_{\theta,\beta}>0$ is equivalent to $p<1$.
    Then, by the Schatten--H\"older inequality
    \cite[Theorem~6.3 and Notes and Remarks, p.~151]{Diestel_Jarchow_Tonge_1995},
    \begin{equation}
        \label{eq:smoothness-holder-bound}
        \norm{\imath G}_{\calS_{2p}(\ell^2(d),L^2_\mu(\Gamma))}
        \leq
        \norm{\imath}_{\calS_{2\theta}(\calH,L^2_\mu(\Gamma))}
        \norm{G}_{\calS_{2\beta}(\ell^2(d),\calH)}.
    \end{equation}
    Combining \eqref{eq:smoothness-embedding-schatten},
    \eqref{eq:smoothness-G-H-schatten}, and
    \eqref{eq:smoothness-holder-bound}, we obtain
    \begin{equation}
        \label{eq:smoothness-G0-schatten-bound}
        \norm{\imath G}_{\calS_{2p}(\ell^2(d),L^2_\mu(\Gamma))}^2
        \leq
        Z_\theta^{1/\theta}
        \norm{H}_{\calS_\beta(\ell^2(d))}.
    \end{equation}

    Finally, $C=(\imath G)^*(\imath G)$, so the singular values of
    $\imath G$ are the numbers $\{\sigma_j\}_{j\geq 1}$ in
    \eqref{eq:cov-spectral-decomposition}.
    Since $p<1$, Stechkin's inequality applied to the nonincreasing sequence
    $\{\sigma_j^2\}_{j\geq 1}$ gives
    \begin{equation}
        \label{eq:smoothness-stechkin-tail}
        \err(\Pi_r)^2
        =
        \sum_{j>r}
        \sigma_j^2
        \leq
        \left(
        \sum_{j\geq 1}
        \sigma_j^{2p}
        \right)^{1/p}
        r^{1-1/p}
        =
        \norm{\imath G}_{\calS_{2p}(\ell^2(d),L^2_\mu(\Gamma))}^2
        r^{1-1/p}.
    \end{equation}
    Since
    \begin{equation}
        1-\frac{1}{p}
        =
        -q_{\theta,\beta},
    \end{equation}
    the estimate \eqref{eq:smoothness-active-subspace-tail} follows from
    \eqref{eq:smoothness-G0-schatten-bound} and
    \eqref{eq:smoothness-stechkin-tail}.
\end{proof}

\subsection{Tolerance-based sample complexity}

We finally combine the {regularized} coherence estimate with the active subspace tail energy decay estimate.
When $\err(\Pi_r)>0$, the canonical {regularization} scale
\eqref{eq:canonical-ridge-scale} is useful for interpreting the
quasi-optimality theorem, but it is not directly an a priori choice because
$\err(\Pi_r)^2$ is unknown.

\begin{theorem}[Tolerance-level sample complexity from smoothness]
    \label{thm:smoothness-tolerance-complexity}
    Suppose that Assumption~\ref{ass:smoothness-covariance-summability} holds {for
            exponents $\theta$ and $\beta$ satisfying} $q_{\theta,\beta}>0$.
    Define, for $\zeta\in(0,\infty]$,
    \begin{equation}
        \label{eq:smoothness-tail-rate-constant}
        A_{\theta,\zeta}
        :=
        Z_\theta
        \norm{H}_{\calS_\zeta(\ell^2(d))}^\theta.
    \end{equation}
    For a prescribed tolerance $\tol>0$, choose
    \begin{equation}
        \label{eq:smoothness-rank-from-tolerance}
        r_\tol
        :=
        \left\lceil
        \left(
        \frac{A_{\theta,\beta}^{1/\theta}}{\tol^2}
        \right)^{1/q_{\theta,\beta}}
        \right\rceil.
    \end{equation}
    Set
    \begin{equation}
        \label{eq:smoothness-tolerance-ridge-scale}
        \lambda_\tol
        :=
        \frac{\tol^2}{r_\tol}.
    \end{equation}
    Set
    \begin{equation}
        \label{eq:smoothness-tolerance-ridge-diagonal-scale}
        \alpha_\tol
        :=
        \frac{\lambda_\tol}{\norm{H}_{\mathcal L(\ell^2(d))}}.
    \end{equation}
    {Let $\rho$ be either the exact regularized kernel diagonal measure
    $\rho_{\alpha_\tol}$ or the $\theta$-kernel diagonal measure $\rho^\theta$.}
    There exist universal constants $c_0,c_1>0$ such that, for every
    $\varepsilon,\eta\in(0,1)$, if
    \begin{equation}
        {\label{eq:smoothness-tolerance-sample-condition}
            M
            \geq
            \frac{
                \max\left\{
                A_{\theta,\infty}
                \tol^{-2\theta}
                \left(
                1
                +
                A_{\theta,\beta}^{1/q_{\theta,\beta}}
                \tol^{-2\theta/q_{\theta,\beta}}
                \right),
                1
                \right\}
            }{
                c_0\varepsilon^2
            }
            \log\left(
            \frac{
                c_1\left(
                3
                +
                2
                \left(
                \frac{A_{\theta,\beta}^{1/\theta}}{\tol^2}
                \right)^{1/q_{\theta,\beta}}
                \right)
            }{
                \eta
            }
            \right)}.
    \end{equation}
    then the empirical active subspace projector $\hat\Pi_{r_\tol}$ computed from
    \eqref{eq:emp-cov-is} satisfies, with probability at least $1-\eta$,
    \begin{equation}
        \label{eq:smoothness-tolerance-error}
        \err(\hat\Pi_{r_\tol})^2
        \leq
        \frac{1+\varepsilon}{1-\varepsilon}
        \tol^2{.}
    \end{equation}
\end{theorem}

\begin{proof}
    By Proposition~\ref{prop:smoothness-active-subspace-tail} and
    \eqref{eq:smoothness-tail-rate-constant},
    \begin{equation}
        \label{eq:smoothness-tail-rate}
        \err(\Pi_r)^2
        \leq
        A_{\theta,\beta}^{1/\theta}r^{-q_{\theta,\beta}},
        \qquad
        r\geq 1.
    \end{equation}
    Applying \eqref{eq:smoothness-tail-rate} with
    $r=r_\tol$ and using \eqref{eq:smoothness-rank-from-tolerance} gives
    \begin{equation}
        \label{eq:smoothness-tolerance-rank-condition}
        \err(\Pi_{r_\tol})^2
        \leq
        A_{\theta,\beta}^{1/\theta}r_\tol^{-q_{\theta,\beta}}
        \leq
        \tol^2.
    \end{equation}
    For $\lambda_\tol=\tol^2/r_\tol$, this implies the covariance
    effective-dimension bound
    \begin{equation}
        \label{eq:smoothness-tolerance-effective-dimension}
        d_{\mathrm{eff}}(\lambda_\tol)
        \leq
        r_\tol
        +
        \frac{1}{\lambda_\tol}
        \sum_{j>r_\tol}
        \sigma_j^2
        =
        r_\tol
        +
        \frac{\err(\Pi_{r_\tol})^2}{\tol^2}r_\tol
        \leq
        2r_\tol.
    \end{equation}
    Moreover, the head rank \eqref{eq:head-rank} satisfies
    \begin{equation}
        \label{eq:smoothness-tolerance-head-rank}
        s_{\lambda_\tol}
        \leq
        2r_\tol.
    \end{equation}
    Indeed, at most $r_\tol$ head eigenvalues can occur among the first
    $r_\tol$ indices, and each head eigenvalue with index larger than
    $r_\tol$ contributes more than $\lambda_\tol$ to the active subspace
    tail energy $\err(\Pi_{r_\tol})^2$.
        {For $\rho=\rho_{\alpha_\tol}$,
            Proposition~\ref{prop:smoothness-ridge-diagonal-sampling} with $c_2=1$
            and \eqref{eq:smoothness-ridge-theta-dimension-comparison} give the bound below.
            For $\rho=\rho^\theta$, the same bound follows from
            Corollary~\ref{prop:smoothness-diagonal-sampling}.}
    \begin{equation}
        {\label{eq:smoothness-tolerance-special-K-bound}
            K_{\rho,\lambda_\tol}(g)
            \leq
            A_{\theta,\infty}}
        \left(
        \frac{r_\tol}{\tol^2}
        \right)^\theta.
    \end{equation}
    The ceiling bound
    \begin{equation}
        r_\tol
        \leq
        1
        +
        \left(
        \frac{A_{\theta,\beta}^{1/\theta}}{\tol^2}
        \right)^{1/q_{\theta,\beta}}
    \end{equation}
    and the subadditivity of $t\mapsto t^\theta$ on $[0,\infty)$ give
    \begin{equation}
        \label{eq:smoothness-tolerance-rank-upper-bound}
        \left(
        \frac{r_\tol}{\tol^2}
        \right)^\theta
        \leq
        \tol^{-2\theta}
        +
        A_{\theta,\beta}^{1/q_{\theta,\beta}}
        \tol^{-2\theta\left(1+1/q_{\theta,\beta}\right)}.
    \end{equation}
    Similarly,
    \begin{equation}
        \label{eq:smoothness-tolerance-log-bound}
        1+d_{\mathrm{eff}}(\lambda_\tol)
        \leq
        1+2r_\tol
        \leq
        3
        +
        2
        \left(
        \frac{A_{\theta,\beta}^{1/\theta}}{\tol^2}
        \right)^{1/q_{\theta,\beta}}.
    \end{equation}
    Therefore \eqref{eq:smoothness-tolerance-sample-condition},
    after a harmless change of the universal
    constants $c_0,c_1$, implies the sampling condition
    \eqref{eq:samp-ineq-eta} in Theorem~\ref{thm:slas-ridge-quasi-opt} with
    $r=r_\tol$ and $\lambda=\lambda_\tol$.
    Applying \eqref{eq:quasi-opt-general-lambda} gives
    \begin{equation}
        \err(\hat\Pi_{r_\tol})^2
        \leq
        \frac{
            \left(1+\frac{\varepsilon}{4}\right)\err(\Pi_{r_\tol})^2
            +
            \frac{\varepsilon}{4}(r_\tol+s_{\lambda_\tol})\lambda_\tol
        }{
            1-\frac{\varepsilon}{4}
        }.
    \end{equation}
    Since $\err(\Pi_{r_\tol})^2\leq \tol^2$,
    $r_\tol\lambda_\tol=\tol^2$, and
    $s_{\lambda_\tol}\lambda_\tol\leq 2r_\tol\lambda_\tol=2\tol^2$,
    \begin{equation}
        \err(\hat\Pi_{r_\tol})^2
        \leq
        \frac{1+\varepsilon}{1-\frac{\varepsilon}{4}}\tol^2
        \leq
        \frac{1+\varepsilon}{1-\varepsilon}\tol^2.
    \end{equation}
    This proves \eqref{eq:smoothness-tolerance-error}.
\end{proof}

{
For fixed model-dependent quantities and fixed $\varepsilon,\eta$, the leading
tolerance dependence in \eqref{eq:smoothness-tolerance-sample-condition} as
$\tol\to0$ is
\begin{equation*}
    M
    \gtrsim
    \tol^{-2\theta(1+1/q_{\theta,\beta})}
    \log\left(
    \tol^{-1}
    \right).
\end{equation*}
The factor $\tol^{-2\theta}$ comes from regularized coherence, whereas the
additional power $\tol^{-2\theta/q_{\theta,\beta}}$ reflects the growth
$r_\tol\asymp\tol^{-2/q_{\theta,\beta}}$ of the sufficient rank.
The logarithmic factor also depends on this rank growth.}

The RKHS assumptions above are abstract at this stage.
The next section verifies them for concrete model classes, where the
$\theta$-kernel diagonal and the {RKHS covariance} can be bounded from
parametric regularity estimates.

\section{Parametric elliptic PDEs}
\label{sec:parametric-pdes-smoothness}

Partial differential equations (PDEs) are commonly used to model complex systems in a variety of physical contexts.
When solving a given PDE, one typically fixes certain parameters: the shape of the physical domain, the diffusion or velocity field, the source term, the flux or reaction law, etc.
We use the terminology parametric PDEs when some of these parameters are allowed to vary over a certain range of interest.
When the latter is endowed with a probability measure, we use the term random PDEs.
These mathematical objects have been extensively studied in recent years; see, e.g., \cite{cohen2015approximation,hesthaven2016certified,adcock2022sparse,bachmayr2017sparsei,bachmayr2017sparseii} and references therein.

When treating parametric PDEs, one is interested in finding the solution for all parameters in the range of interest.
Such problems can be described by considering the formulation
\begin{equation}
    \pde(u,a) = 0,
\end{equation}
where $a$ denotes the parameters, $u$ is the unknown of the problem, and $\pde: \calV \times \calX \to \calW$ is a linear or nonlinear partial differential operator, with $(\calV, \calX, \calW)$ a triplet of Banach spaces.
We allow the parameter $a$ to range over a subset of $\calX$ over which the solution map $a \mapsto u(a)$ is well defined.

In this paper, we focus on a simple guiding example given by the linear elliptic equation obtained by considering
\begin{equation}
    \pde(u,a) = f_D + \mathrm{div}(a \nabla u),
\end{equation}
set on a bounded Lipschitz domain $D \subset \R^s$.
Here, the real-valued function $f_D$ is a fixed forcing term, the parameter is given by the real-valued diffusion coefficients $a \in L^\infty(D)$, and a possible choice for the triplet of spaces is $(\calV, \calX, \calW) = (H^1_0(D), L^\infty(D), H^{-1}(D))$.

In deterministic modelling \cite{cohen2015approximation}, the parameters are deterministic design or control variables, which may be tuned by the user so that the solution $u$, or a scalar quantity of interest, has prescribed properties.
Two common models for the parameter $a$ are the lognormal and affine uniform models.
Given a basis $\{ \psi_j \}_{j\geq 1}$ of $\calX$, the lognormal model takes
\begin{equation}
    a(x,\yb)
    =
    \exp\left(\sum_{m\geq 1} y_m\psi_m(x)\right),
    \qquad
    Y\sim\mu,
\end{equation}
where $\mu$ is the standard Gaussian product measure on $\Gamma=\R^d$, while the affine uniform model is given by
\begin{equation}
    a(x,\yb)
    =
    \bar a(x)+\sum_{m\geq 1} y_m\psi_m(x),
    \qquad
    Y\sim\mu,
\end{equation}
where now $\mu$ is the product uniform measure {supported} on $[-1,1]^d{\subset\Gamma}$.
Denoting the corresponding solution map by $\yb \mapsto u(\yb)$, the considered elliptic PDE reads in weak form as
\begin{equation}
    \label{eq:parametric-elliptic-weak-form}
    \int_D
    a(x,\yb)\nabla u(x,\yb)\cdot\nabla v(x)
    \,\dif x
    =
    \langle f_D,v\rangle_{\calV',\calV},
    \qquad
    v\in \calV.
\end{equation}
In this paper, we restrict attention to a bounded linear quantity-of-interest functional $Q\in\calV'$ and define
\begin{equation}
    \label{eq:pde-quantity-of-interest}
    f(\yb) := Q(u(\yb)).
\end{equation}
The number of variables is countably infinite, that is $d = +\infty$, or very large if the above expansions have been truncated with high accuracy.
These models are motivated by considering a Karhunen-Lo\`eve expansion of the diffusion coefficient (or of its logarithm).

The naturally high dimensionality of the parameter spaces makes random PDEs with function-valued parameters fertile ground for the application of the AS method.
Furthermore, the solution map of parametric PDEs is often very smooth, which is a key ingredient for the success of the AS method as presented in Section~\ref{sec:sample-complexity-smoothness}.
Indeed, the regularity of the solution map is a well-studied topic in the literature on parametric PDEs, and it is known that under suitable assumptions on the basis $\{ \psi_j \}_{j\geq 1}$ the solution map is infinitely differentiable with respect to the parameters.
The goal of this section is hence to introduce RKHSs which capture the regularity of the solution map of this class of parametric PDEs, and to verify the gradient regularity assumptions needed for the results from Section~\ref{sec:sample-complexity-smoothness} to apply.

We use the following multi-index notation throughout the section.
Let $\mathcal F$ be the set of finitely supported multi-indices
$\nu=(\nu_m)_{m\geq 1}$ with entries in $\mathbb N_0$.
When $d<\infty$, this is simply the usual set of multi-indices in $\mathbb N_0^d$.
All sums over $m\geq 1$ are understood with this convention; in the finite-dimensional case
they terminate at $m=d$.
For $\alpha,\nu\in\mathcal F$, we write $\alpha\leq\nu$ componentwise and define
\begin{equation}
    \binom{\nu}{\alpha}
    :=
    \prod_{m\geq 1}\binom{\nu_m}{\alpha_m}.
\end{equation}
For any strictly positive sequence $\boldsymbol c=(c_m)_{m\geq1}$ and $\gamma\in\mathcal F$, we also write
\begin{equation*}
    \boldsymbol c^{\gamma}
    :=
    \prod_{m\geq1}c_m^{\gamma_m}.
\end{equation*}

\subsection{The lognormal model}
\label{subsec:lognormal-model}

We first consider the Gaussian lognormal setting from
\cite{bachmayr2017sparseii}.
Let $\mu$ be the standard Gaussian product measure on $\Gamma=\R^d$ and let
$\{H_n\}_{n\geq 0}$ denote the $L^2_{\mu_1}(\R)$-orthonormal Hermite
polynomials, where $\mu_1$ is the standard Gaussian probability measure on
$\R$.
For $\nu\in\mathcal F$, set
\begin{equation}
    H_\nu(\yb)
    :=
    \prod_{m\geq 1}H_{\nu_m}(y_m).
\end{equation}
Then $\{H_\nu\}_{\nu\in\mathcal F}$ is the tensorized normalized Hermite basis
of $L^2_\mu(\R^d)$.
The diffusion coefficient is
\begin{equation}
    \label{eq:lognormal-coefficient}
    a(x,\yb)
    =
    \exp(b(x,\yb)),
    \qquad
    b(x,\yb)
    :=
    \sum_{m\geq 1}y_m\psi_m(x).
\end{equation}
For every $\yb$ such that $b(\cdot,\yb)\in L^\infty(D)$, the coefficient
$a(\cdot,\yb)$ is strictly positive and bounded above and below on $D$.
Thus, the weak problem \eqref{eq:parametric-elliptic-weak-form} is well posed
for such values of $\yb$, and the solution satisfies
\begin{equation}
    \label{eq:lognormal-apriori}
    \norm{u(\yb)}_{\calV}
    \leq
    \norm{f_D}_{\calV'}
    \exp\left(
    \norm{b(\cdot,\yb)}_{L^\infty(D)}
    \right).
\end{equation}

Following Assumption A from \cite{bachmayr2017sparseii}, we assume that there
exists a strictly positive sequence
$\boldsymbol\varrho=(\varrho_m)_{m\geq 1}$ such that
$\sum_{m\geq 1}\varrho_m\abs{\psi_m}$ converges in $L^\infty(D)$ and
\begin{equation}
    \label{eq:lognormal-assumption-a}
    \sum_{m\geq 1}\exp(-\varrho_m^2)<\infty.
\end{equation}
In the countably infinite case, this ensures that the coefficient and the
solution map are well defined for $\mu$-almost every $\yb$ and have finite
moments of all orders.

For $\nu\in\mathcal F$ and an integer $\mathfrak r\geq 1$, define the Hermite
weight
\begin{equation}
    \label{eq:hermite-weight-lognormal}
    b_{\nu,\mathfrak r}(\boldsymbol\varrho)
    :=
    \sum_{\substack{\alpha\leq\nu\\ \norm{\alpha}_{\ell^\infty}\leq \mathfrak r}}
    \binom{\nu}{\alpha}
    \boldsymbol\varrho^{2\alpha}
    =
    \prod_{m\geq 1}
    \left(
    \sum_{\ell=0}^{\mathfrak r}
    \binom{\nu_m}{\ell}
    \varrho_m^{2\ell}
    \right).
\end{equation}
{Here and below, we use the convention
$\binom{n}{\ell}=0$ whenever $\ell>n$.}
Define further the associated weighted Hermite space by
\begin{equation}
    \label{eq:lognormal-hermite-rkhs}
    \mathcal A^{(\mathfrak r)}_{\boldsymbol\varrho}
    :=
    \left\{
    p=\sum_{\nu\in\mathcal F}p_\nu H_\nu
    \in L^2_\mu(\R^d)
    :
    \norm{p}_{\mathcal A^{(\mathfrak r)}_{\boldsymbol\varrho}}^2
    :=
    \sum_{\nu\in\mathcal F}
    b_{\nu,\mathfrak r}(\boldsymbol\varrho)
    \abs{p_\nu}^2
    <
    \infty
    \right\}.
\end{equation}
The embedding of this weighted space into $L^2_\mu(\R^d)$ is diagonal in the
Hermite basis.
In the notation of Section~\ref{sec:sample-complexity-smoothness}, its
eigenfunctions are $\phi_\nu=H_\nu$ and its eigenvalues are
$\gamma_\nu=b_{\nu,\mathfrak r}(\boldsymbol\varrho)^{-1}$.
For $\theta\in(0,1]$, the $\theta$-kernel diagonal and normalization constant are
\begin{equation}
    \label{eq:lognormal-hermite-diagonal}
    \kappa^\theta(\yb)
    :=
    \sum_{\nu\in\mathcal F}
    b_{\nu,\mathfrak r}(\boldsymbol\varrho)^{-\theta}
    H_\nu(\yb)^2
\end{equation}
and
\begin{equation}
    \label{eq:lognormal-hermite-normalization}
    Z_\theta
    =
    \int_{\R^d}\kappa^\theta(\yb)\,\dif\mu(\yb)
    =
    \sum_{\nu\in\mathcal F}
    b_{\nu,\mathfrak r}(\boldsymbol\varrho)^{-\theta}.
\end{equation}

\begin{lemma}[Admissible exponents for the weighted Hermite space]
    \label{lem:lognormal-hermite-admissible-scales}
    Let $\mathfrak r\geq 2$ {and suppose that
            $(\varrho_m^{-1})_{m\geq 1}\in\ell^2$.
            Then $\mathcal A^{(\mathfrak r)}_{\boldsymbol\varrho}$ is an RKHS.
            Moreover, every} $\theta\in(0,1]$ {satisfying}
    \begin{equation}
        \label{eq:lognormal-theta-admissibility}
        \theta>\frac{1}{\mathfrak r},
        \qquad
        (\varrho_m^{-1})_{m\geq 1}\in\ell^{2\theta}
    \end{equation}
    {is an admissible exponent for $\mathcal A^{(\mathfrak r)}_{\boldsymbol\varrho}$.}
\end{lemma}

\begin{proof}
    {In the notation of
        \cite[Lemma 5.1]{bachmayr2017sparseii}, first take $q=2$ and $p=1$.
        Since $\mathfrak r\geq2$, we have $p>2/(\mathfrak r+1)$, and the cited
        lemma gives}
    \begin{equation*}
        {
            \sum_{\nu\in\mathcal F}
            b_{\nu,\mathfrak r}(\boldsymbol\varrho)^{-1}
            <
            \infty.}
    \end{equation*}
    Thus the kernel diagonal of
    $\mathcal A^{(\mathfrak r)}_{\boldsymbol\varrho}$ is integrable.
    Cauchy--Schwarz therefore shows that
    $\mathcal A^{(\mathfrak r)}_{\boldsymbol\varrho}$ is an RKHS.
        {Now let $\theta\in(0,1]$ satisfy
            \eqref{eq:lognormal-theta-admissibility}.
            Taking $q=2\theta$ and $p=2\theta/(1+\theta)$, the condition
            $\theta>1/\mathfrak r$ is equivalent to $p>2/(\mathfrak r+1)$.
            The cited lemma therefore gives}
    \begin{equation}\label{eq:lognormal-hermite-normalization-finite}
        Z_\theta
        =
        \sum_{\nu\in\mathcal F}
        b_{\nu,\mathfrak r}(\boldsymbol\varrho)^{-\theta}
        <
        \infty.
    \end{equation}
    Finally, \eqref{eq:lognormal-hermite-normalization} and
    \eqref{eq:lognormal-hermite-normalization-finite} show that
    $\kappa^\theta(\yb)<\infty$ for $\mu$-almost every $\yb$.
    Since its defining sum contains the strictly positive constant-mode term
    corresponding to $H_0=1$, we also have $\kappa^\theta(\yb)>0$ for every
    $\yb$.
    Hence $\theta$ is admissible.
\end{proof}

The next lemma controls both the loss of one Hermite order under differentiation
and the weighted $\ell^2$-sum of all coordinate derivatives.

\begin{lemma}[Hermite differentiation and weighted derivative summability]
    \label{lem:lognormal-one-order-loss}
    Let $\mathfrak r\geq 1$.
    If
    $p=\sum_{\nu\in\mathcal F}p_\nu H_\nu
        \in\mathcal A^{(\mathfrak r+1)}_{\boldsymbol\varrho}$,
    then, for every $m\geq 1$,
    $\partial_m p\in\mathcal A^{(\mathfrak r)}_{\boldsymbol\varrho}$ and
    \begin{equation}
        \label{eq:lognormal-derivative-component-bound}
        \norm{\partial_m p}_{\mathcal A^{(\mathfrak r)}_{\boldsymbol\varrho}}^2
        \leq
        (\mathfrak r+1)\varrho_m^{-2}
        \norm{p}_{\mathcal A^{(\mathfrak r+1)}_{\boldsymbol\varrho}}^2.
    \end{equation}
    If, in addition,
    $\underline{\varrho}:=\inf_m\varrho_m>0$, $\tau>1$, and
    $p\in\mathcal A^{(\mathfrak r+1)}_{\tau\boldsymbol\varrho}$, then
    \begin{equation}
        \label{eq:lognormal-derivative-weighted-sum-bound}
        \sum_{m\geq1}
        \varrho_m^2
        \norm{\partial_m p}_{\mathcal A^{(\mathfrak r)}_{\boldsymbol\varrho}}^2
        \leq
        C_{\log}
        \norm{p}_{\mathcal A^{(\mathfrak r+1)}_{\tau\boldsymbol\varrho}}^2,
    \end{equation}
    where
    \begin{equation}
        \label{eq:lognormal-derivative-weighted-sum-constant}
        {C_{\log}
            :=
            (\mathfrak r+1)
            \sup_{s\geq1}
            s \chi^{s-1},
            \qquad\text{and}\qquad
            \chi
            :=
            \frac{
                1+\underline{\varrho}^{\,2}
            }{
                1+\tau^2\underline{\varrho}^{\,2}
            }.}
    \end{equation}
    In particular, $0<\chi<1$, hence $C_{\log}<\infty$.
\end{lemma}

\begin{proof}
    For $n\geq 0$ and $\varrho>0$, set
    \begin{equation}
        b_{n,\mathfrak r}(\varrho)
        :=
        \sum_{\ell=0}^{\mathfrak r}
        \binom{n}{\ell}\varrho^{2\ell}.
    \end{equation}
    Using
    $(n+1)\binom{n}{\ell}=(\ell+1)\binom{n+1}{\ell+1}$, we obtain
    \begin{equation}
        \begin{aligned}
            (n+1)b_{n,\mathfrak r}(\varrho)
             & =
            \sum_{\ell=0}^{\mathfrak r}
            (n+1)\binom{n}{\ell}\varrho^{2\ell}
            \\
             & =
            \sum_{k=1}^{\mathfrak r+1}
            k\binom{n+1}{k}\varrho^{2(k-1)}
            \\
             & \leq
            (\mathfrak r+1)\varrho^{-2}
            \sum_{k=1}^{\mathfrak r+1}
            \binom{n+1}{k}\varrho^{2k}
            \\
             & \leq
            (\mathfrak r+1)\varrho^{-2}
            b_{n+1,\mathfrak r+1}(\varrho).
        \end{aligned}
    \end{equation}
    Applying this estimate to the $m$-th factor in
    \eqref{eq:hermite-weight-lognormal}, and using
    $b_{\nu_j,\mathfrak r}(\varrho_j)
        \leq b_{\nu_j,\mathfrak r+1}(\varrho_j)$ for $j\neq m$, gives
    \begin{equation}
        \label{eq:lognormal-weight-shift}
        (\nu_m+1)b_{\nu,\mathfrak r}(\boldsymbol\varrho)
        \leq
        (\mathfrak r+1)\varrho_m^{-2}
        b_{\nu+e_m,\mathfrak r+1}(\boldsymbol\varrho).
    \end{equation}

    The Hermite derivative identity gives
    \begin{equation}
        \partial_m p
        =
        \sum_{\nu\in\mathcal F}
        \sqrt{\nu_m+1}\,p_{\nu+e_m}H_\nu
    \end{equation}
    in $L^2_\mu$.
    Therefore
    \begin{equation}
        \begin{aligned}
            \norm{\partial_m p}_{\mathcal A^{(\mathfrak r)}_{\boldsymbol\varrho}}^2
             & =
            \sum_{\nu\in\mathcal F}
            b_{\nu,\mathfrak r}(\boldsymbol\varrho)
            (\nu_m+1)
            \abs{p_{\nu+e_m}}^2
            \\
             & \leq
            (\mathfrak r+1)\varrho_m^{-2}
            \sum_{\nu\in\mathcal F}
            b_{\nu+e_m,\mathfrak r+1}(\boldsymbol\varrho)
            \abs{p_{\nu+e_m}}^2
            \\
             & \leq
            (\mathfrak r+1)\varrho_m^{-2}
            \norm{p}_{\mathcal A^{(\mathfrak r+1)}_{\boldsymbol\varrho}}^2.
        \end{aligned}
    \end{equation}
    This proves \eqref{eq:lognormal-derivative-component-bound}.

    It remains to prove
    \eqref{eq:lognormal-derivative-weighted-sum-bound}.
    For $n\geq1$ and $\varrho\geq\underline{\varrho}$, set
    \begin{equation}
        \widetilde b_{n,\mathfrak r}(\varrho)
        :=
        b_{n,\mathfrak r}(\varrho)-1
        =
        \sum_{\ell=1}^{\mathfrak r}
        \binom{n}{\ell}\varrho^{2\ell}.
    \end{equation}
    Since
    $\widetilde b_{n,\mathfrak r}(\varrho)\geq\varrho^2
        \geq\underline{\varrho}^{\,2}$, we have
    \begin{equation}
        \label{eq:lognormal-weight-ratio-contraction}
        \frac{
            b_{n,\mathfrak r}(\varrho)
        }{
            b_{n,\mathfrak r+1}(\tau\varrho)
        }
        \leq
        \frac{
            1+\widetilde b_{n,\mathfrak r}(\varrho)
        }{
            1+\tau^2\widetilde b_{n,\mathfrak r}(\varrho)
        }
        \leq
        \chi.
    \end{equation}
    For $\nu\in\mathcal F$, let $s(\nu)$ denote the number of its nonzero
    components.
    For fixed $m$ in the support of $\nu$, the shifted-weight estimate
    \eqref{eq:lognormal-weight-shift} controls the $m$th factor, while each of
    the other $s(\nu)-1$ nonzero factors contributes at most $\chi$.
    Summing over the support of $\nu$ and using
    $b_{n,\mathfrak r+1}(\varrho)
        \leq b_{n,\mathfrak r+1}(\tau\varrho)$ gives
    \begin{equation}
        \begin{aligned}
             & \sum_{\substack{m\geq1 \\ \nu_m\geq1}}
            \frac{
                \varrho_m^2\nu_m
                b_{\nu-e_m,\mathfrak r}(\boldsymbol\varrho)
            }{
                b_{\nu,\mathfrak r+1}(\tau\boldsymbol\varrho)
            }
            \\
             & \qquad\leq
            (\mathfrak r+1)
            s(\nu)\chi^{s(\nu)-1}
            \leq
            C_{\log}.
        \end{aligned}
    \end{equation}
    Therefore, using the Hermite derivative identity and exchanging the
    nonnegative sums,
    \begin{equation}
        \begin{aligned}
            \sum_{m\geq1}
            \varrho_m^2
            \norm{\partial_m p}_{
            \mathcal A^{(\mathfrak r)}_{\boldsymbol\varrho}
            }^2
             & =
            \sum_{\nu\in\mathcal F}
            \abs{p_\nu}^2
            \sum_{\substack{m\geq1 \\ \nu_m\geq1}}
            \varrho_m^2\nu_m
            b_{\nu-e_m,\mathfrak r}(\boldsymbol\varrho)
            \\
             & \leq
            C_{\log}
            \sum_{\nu\in\mathcal F}
            b_{\nu,\mathfrak r+1}(\tau\boldsymbol\varrho)
            \abs{p_\nu}^2.
        \end{aligned}
    \end{equation}
    This is \eqref{eq:lognormal-derivative-weighted-sum-bound}.
\end{proof}

We now combine the Hermite coefficient estimate for the lognormal elliptic
model with the abstract RKHS consequences from
Section~\ref{sec:sample-complexity-smoothness}.

\begin{theorem}[Smoothness consequences for the lognormal elliptic model]
    \label{thm:lognormal-pde-smoothness-consequences}
    Let $\mathfrak r\geq 2$ and let $\theta\in(0,1]$ satisfy
    \begin{equation}
        \label{eq:lognormal-smoothness-exponent-assumption}
        \theta>\frac{1}{\mathfrak r},
        \qquad
        (\varrho_m^{-1})_{m\geq 1}\in\ell^{2\theta}.
    \end{equation}
    Assume that the series
    $\sum_{m\geq1}\varrho_m\abs{\psi_m}$ converges in $L^\infty(D)$,
    that \eqref{eq:lognormal-assumption-a} holds, and that the
        {sequence $(\varrho_m)_{m\geq1}$} satisfies
    the smallness condition
    \begin{equation}
        \label{eq:lognormal-weighted-smallness}
        \sup_{x\in D}
        \sum_{m\geq 1}
        \varrho_m\abs{\psi_m(x)}
        <
        \frac{\log 2}{\sqrt{\mathfrak r+1}}.
    \end{equation}
    Let $u(\yb)$ solve \eqref{eq:parametric-elliptic-weak-form} with the
    lognormal coefficient \eqref{eq:lognormal-coefficient}, let $Q\in\calV'$,
    and set $f=Q\circ u$ as in \eqref{eq:pde-quantity-of-interest}.
    Set $\mathcal H:=\mathcal A^{(\mathfrak r)}_{\boldsymbol\varrho}$ and
    $\beta:=\theta/(1+\theta)$.
    Lemma~\ref{lem:lognormal-hermite-admissible-scales} ensures that
    $\mathcal H$ is an RKHS and that $\theta$ is an admissible exponent.
    Then the following statements hold.

    \begin{enumerate}[label=(\roman*)]
        \item The target and its coordinate derivatives satisfy
              \begin{equation}
                  \label{eq:lognormal-target-and-gradient-regularity}
                  f\in
                  \mathcal A^{(\mathfrak r+1)}_{\boldsymbol\varrho},
                  \qquad
                  \partial_m f
                  \in
                  \mathcal A^{(\mathfrak r)}_{\boldsymbol\varrho},
                  \quad m\geq 1.
              \end{equation}
              The coordinate derivatives are jointly summable in $\mathcal H$:
              \begin{equation*}
                  B_f
                  :=
                  \sum_{m\geq1}
                  \varrho_m^2
                  \norm{\partial_m f}_{\mathcal H}^2
                  <
                  \infty.
              \end{equation*}
              In particular, $f\in H^1_\mu(\R^d)$ and
              $\nabla f=(\partial_m f)_{m\geq1}$ belongs to
              $L^2_\mu(\R^d;\ell^2(d))$.

        \item Define the scalar-projection operator
              \begin{equation}
                  \label{eq:lognormal-scalar-projection-operator}
                  G:\ell^2(d)\to\mathcal H,
                  \qquad
                  Gv
                  :=
                  \sum_{m\geq1}v_m\partial_m f.
              \end{equation}
              Then $G$ is well defined and bounded, and the smoothness
              covariance satisfies
              \begin{equation}
                  \label{eq:lognormal-schatten-smoothness}
                  H
                  =
                  G^*G
                  \in
                  \calS_\beta(\ell^2(d)),
                  \quad
                  \norm{H}_{\calS_\beta(\ell^2(d))}
                  \leq
                  B_f
                  \left(
                  \sum_{m\geq1}
                  \varrho_m^{-2\theta}
                  \right)^{1/\theta}.
              \end{equation}

        \item The {regularized} kernel diagonal measure
              $\rho_{\alpha_\lambda}$ from
              \eqref{eq:smoothness-ridge-diagonal-proposal} is well defined for
              every $\lambda>0$ and satisfies
              \begin{equation}
                  \label{eq:lognormal-ridge-diagonal-coherence-bound}
                  K_{\rho_{\alpha_\lambda},\lambda}(\nabla f)
                  \leq
                  d_{\mathrm{eff}}(\alpha_\lambda,\gamma)
                  \leq
                  Z_\theta
                  \norm{H}_{\mathcal L(\ell^2(d))}^{\theta}
                  \lambda^{-\theta}.
              \end{equation}
              The $\theta$-kernel diagonal measure $\rho^\theta$ from
              \eqref{eq:smoothness-diagonal-proposal} satisfies, for every
              $\lambda>0$,
              \begin{equation}
                  \label{eq:lognormal-ridge-coherence-bound}
                  K_{\rho^\theta,\lambda}(\nabla f)
                  \leq
                  Z_\theta
                  \norm{H}_{\mathcal L(\ell^2(d))}^{\theta}
                  \lambda^{-\theta}.
              \end{equation}

        \item The active subspace tail energy satisfies, for every $r\geq 1$, the
              a priori bound
              \begin{equation}
                  \label{eq:lognormal-active-subspace-tail-bound}
                  \err(\Pi_r)^2
                  \leq
                  Z_\theta^{1/\theta}
                  \norm{H}_{\calS_\beta(\ell^2(d))}
                  r^{-2/\theta}.
              \end{equation}

    \end{enumerate}

\end{theorem}

\begin{proof}
    The summability condition \eqref{eq:lognormal-assumption-a} implies
    $\underline{\varrho}:=\inf_{m\geq1}\varrho_m>0$.
    By the strict inequality in \eqref{eq:lognormal-weighted-smallness}, we
    may choose $\tau>1$ such that
    \begin{equation}
        \tau
        \sup_{x\in D}
        \sum_{m\geq1}
        \varrho_m\abs{\psi_m(x)}
        <
        \frac{\log 2}{\sqrt{\mathfrak r+1}}.
    \end{equation}
    The scaled weights $\tau\boldsymbol\varrho$ still satisfy Assumption A.
    Hence \cite[Theorems 2.1, 3.1, and 4.2]{bachmayr2017sparseii}, applied at
    order $\mathfrak r+1$, give
    \begin{equation}
        \label{eq:lognormal-solution-hermite-regularity}
        u
        =
        \sum_{\nu\in\mathcal F}u_\nu H_\nu,
        \qquad
        \sum_{\nu\in\mathcal F}
        b_{\nu,\mathfrak r+1}(\tau\boldsymbol\varrho)
        \norm{u_\nu}_{\calV}^2
        <
        \infty.
    \end{equation}
    Since $Q\in\calV'$, the same estimate holds for the Hermite coefficients of
    $f=Q\circ u$, and therefore
    $f\in\mathcal A^{(\mathfrak r+1)}_{\tau\boldsymbol\varrho}
        \subset\mathcal A^{(\mathfrak r+1)}_{\boldsymbol\varrho}$.
    Applying Lemma~\ref{lem:lognormal-one-order-loss} with $p=f$ gives
    \eqref{eq:lognormal-target-and-gradient-regularity} and
    \begin{equation}
        \label{eq:lognormal-gradient-weighted-summability}
        B_f
        =
        \sum_{m\geq1}
        \varrho_m^2
        \norm{\partial_m f}_{\mathcal H}^2
        \leq
        C_{\log}
        \norm{f}_{
        \mathcal A^{(\mathfrak r+1)}_{\tau\boldsymbol\varrho}
        }^2
        <
        \infty.
    \end{equation}
    Since $b_{\nu,\mathfrak r}(\boldsymbol\varrho)\geq1$ and
    $\varrho_m\geq\underline{\varrho}$,
    {it follows that}
    $\sum_m\norm{\partial_m f}_{L^2_\mu(\R^d)}^2
        \leq\underline{\varrho}^{-2}B_f<\infty$.
    The coordinate characterization
    \eqref{eq:coordinate-weak-derivative-characterization} therefore shows that
    $f\in H^1_\mu(\R^d)$ and identifies its gradient with
    $(\partial_m f)_{m\geq1}$.
    We factor $G$ into a weighted derivative operator and a diagonal rescaling.
    Define
    \begin{equation*}
        G_{\boldsymbol\varrho}e_m
        :=
        \varrho_m\partial_m f,
        \qquad
        R^{-1}e_m
        :=
        \varrho_m^{-1}e_m.
    \end{equation*}
    By \eqref{eq:lognormal-gradient-weighted-summability} and
    $(\varrho_m^{-1})_{m\geq1}\in\ell^{2\theta}$,
    \begin{equation*}
        \begin{aligned}
            G_{\boldsymbol\varrho}
             & \in
            \calS_2(\ell^2(d),\mathcal H),
             &
            \norm{G_{\boldsymbol\varrho}}_{\calS_2}^2
             & =
            B_f,
            \\
            R^{-1}
             & \in
            \calS_{2\theta}(\ell^2(d)),
             &
            \norm{R^{-1}}_{\calS_{2\theta}}^2
             & =
            \left(
            \sum_{m\geq1}\varrho_m^{-2\theta}
            \right)^{1/\theta}.
        \end{aligned}
    \end{equation*}
    Since $G_{\boldsymbol\varrho}R^{-1}e_m=\partial_m f$, the series in
    \eqref{eq:lognormal-scalar-projection-operator} converges in $\mathcal H$ and
    \begin{equation*}
        G=G_{\boldsymbol\varrho}R^{-1}.
    \end{equation*}
    By the definition of the {RKHS covariance}, $H=G^*G$.
    Finally, $1/(2\beta)=1/2+1/(2\theta)$, and the Schatten--H\"older inequality
    yields
    \begin{equation}
        \begin{aligned}
            \norm{H}_{\calS_\beta(\ell^2(d))}
             & =
            \norm{G}_{\calS_{2\beta}(\ell^2(d),\mathcal H)}^2
            \\
             & \leq
            \norm{G_{\boldsymbol\varrho}}_{
                \calS_2(\ell^2(d),\mathcal H)}^2
            \norm{R^{-1}}_{\calS_{2\theta}(\ell^2(d))}^2
            \\
             & =
            B_f
            \left(
            \sum_{m\geq1}\varrho_m^{-2\theta}
            \right)^{1/\theta}.
        \end{aligned}
    \end{equation}
    This proves \eqref{eq:lognormal-schatten-smoothness}.

    The constant Hermite mode ensures that
    $\kappa_{\alpha_\lambda}(\yb)>0$ for every $\yb$, while
    \eqref{eq:smoothness-ridge-theta-dimension-comparison} and
    \eqref{eq:smoothness-ridge-theta-dimension-comparison} show that
    the {regularized} kernel diagonal measure is well defined and equivalent to $\mu$.
    Proposition~\ref{prop:smoothness-ridge-diagonal-sampling} gives the first
    inequality in \eqref{eq:lognormal-ridge-diagonal-coherence-bound}, and
    \eqref{eq:smoothness-ridge-to-theta-effective-dimension} gives the second.
    Admissibility gives finiteness of $\kappa^\theta$ $\mu$-a.e., while the
    positive constant Hermite mode makes it strictly positive.
    Hence $\rho^\theta\sim\mu$, and
    \eqref{eq:lognormal-ridge-coherence-bound} follows {directly} from
    Corollary~\ref{prop:smoothness-diagonal-sampling}.
    Finally, Proposition~\ref{prop:smoothness-active-subspace-tail} gives
    \eqref{eq:lognormal-active-subspace-tail-bound}, since
    $q_{\theta,\beta}=2/\theta$.
\end{proof}

The weight-summability condition
$(\varrho_m^{-1})_{m\geq 1}\in\ell^{2\theta}$ supplies two regularity
properties.
It gives the kernel spectral summability $Z_\theta<\infty$, which enters both
the coherence and tail estimates.
Together with the weighted derivative estimate
\eqref{eq:lognormal-gradient-weighted-summability}, it also gives the
coordinate-index summability $H\in\calS_\beta$, which enters only the tail
estimate.
Here $\beta=\theta/(1+\theta)$ and hence
$q_{\theta,\beta}=2/\theta$.
The resulting coherence grows as $\lambda^{-\theta}$, while the active
subspace tail energy decays as $r^{-2/\theta}$.
Thus the anisotropy of the parametric model governs both the sample size and the
active subspace rank.

We next make the resulting tolerance-dependent complexity bound explicit.
Since $\beta=\theta/(1+\theta)$, the constants in
Theorem~\ref{thm:smoothness-tolerance-complexity} are
\begin{equation}
    A_{\theta,\infty}
    =
    Z_\theta
    \norm{H}_{\mathcal L(\ell^2(d))}^{\theta},
    \qquad
    A_{\theta,\beta}
    =
    Z_\theta
    \norm{H}_{\calS_\beta(\ell^2(d))}^{\theta},
\end{equation}
where $A_{\theta,\infty}$ controls the sampling condition and
$A_{\theta,\beta}$ controls the rank required for a prescribed tolerance.
The {explicit} sufficient sample bound
    {\eqref{eq:smoothness-tolerance-sample-condition}} applies {to both the exact
        regularized kernel diagonal measure and} the $\theta$-kernel diagonal measure
$\rho^\theta$.
Since $q_{\theta,\beta}=2/\theta$, the choices
\begin{equation}
    \label{eq:lognormal-informal-tolerance-rates}
    r_{\tol}
    =
    \left\lceil
    A_{\theta,\beta}^{1/2}
    \tol^{-\theta}
    \right\rceil,
    \qquad
    \lambda_{\tol}
    =
    \frac{\tol^2}{r_{\tol}}
\end{equation}
agree with \eqref{eq:smoothness-rank-from-tolerance} and
\eqref{eq:smoothness-tolerance-ridge-scale}.
If $M$ satisfies the condition
\eqref{eq:smoothness-tolerance-sample-condition}, then the empirical projector
satisfies \eqref{eq:smoothness-tolerance-error} with probability at least
$1-\eta$.
After inserting $q_{\theta,\beta}=2/\theta$ in that condition, the leading
algebraic power of $\tol^{-1}$ outside the logarithm is
$2\theta+\theta^2$.

It is also interesting to compare the active subspace {approximation} estimate above with the natural axis-aligned
benchmark, that is the ridge approximation depending only on the best $r$ coordinates according to the weights $\varrho_m$.
    {For the active subspace, since} $\mu$ is the standard Gaussian product measure,
Proposition~\ref{prop:gaussian-conditional-poincare} gives the subspace conditional
Poincar\'e inequality with constant one.
Consequently, if $V_r:\R^r\to\ell^2(d)$ is an isometry satisfying
$\Pi_{V_r}=\Pi_r$, then
\begin{equation}
    \label{eq:lognormal-as-ridge-error}
    \norm{
        f-\mathbb E[f\mid \mathcal G_{V_r}]
    }_{L^2_\mu(\R^d)}
    \leq
    \err(\Pi_r)
    \leq
    \left(
    Z_\theta^{1/\theta}
    \norm{H}_{\calS_\beta(\ell^2(d))}
    \right)^{1/2}
    r^{-1/\theta}.
\end{equation}
For comparison, relabel the coordinates so that
$(\varrho_m^{-1})_{m\geq1}$ is nonincreasing.
The best ridge approximation depending only on the first $r$ coordinates is
the conditional expectation $\mathbb E[f\mid {\pi_1},\ldots,{\pi_r}]$.
Since $f\in\mathcal A^{(\mathfrak r+1)}_{\boldsymbol\varrho}$ and
$b_{\nu,\mathfrak r+1}(\boldsymbol\varrho)
    \geq 1+\varrho_{r+1}^2$
whenever $\nu_m>0$ for some $m>r$, Parseval's identity gives
\begin{equation}
    \label{eq:lognormal-coordinate-ridge-error}
    \begin{aligned}
        \norm{
            f-\mathbb E[f\mid {\pi_1},\ldots,{\pi_r}]
        }_{L^2_\mu(\R^d)}^2
         & =
        \sum_{\substack{\nu\in\mathcal F \\
                \nu_m>0\ \mathrm{for\ some}\ m>r}}
        \abs{f_\nu}^2
        \\
         & \leq
        \frac{
        \norm{f}_{\mathcal A^{(\mathfrak r+1)}_{\boldsymbol\varrho}}^2
        }{
        1+\varrho_{r+1}^2
        }
        \\
         & \leq
        \norm{f}_{\mathcal A^{(\mathfrak r+1)}_{\boldsymbol\varrho}}^2
        \left(
        \sum_{m\geq1}
        \varrho_m^{-2\theta}
        \right)^{1/\theta}
        (r+1)^{-1/\theta},
    \end{aligned}
\end{equation}
where the last inequality uses
$(1+\varrho_{r+1}^2)^{-1}\leq\varrho_{r+1}^{-2}$, the monotonicity of
$(\varrho_m^{-1})_{m\geq1}$, and
\begin{equation}
    (r+1)\varrho_{r+1}^{-2\theta}
    \leq
    \sum_{m\geq1}\varrho_m^{-2\theta}.
\end{equation}
Thus, at the level of these a priori projection-error bounds, the active subspace
approximation doubles the decay exponent relative to coordinate truncation,
improving $r^{-1/(2\theta)}$ to $r^{-1/\theta}$.
The admissibility restriction $\theta>1/\mathfrak r$ shows the benefit of using
larger Hermite regularity orders: increasing $\mathfrak r$ permits smaller values of
$\theta$, hence faster decay exponents in the displayed bounds.
    {Although \eqref{eq:lognormal-weighted-smallness} becomes more
        stringent with $\mathfrak r$ for fixed weights, for any fixed $\mathfrak r$ it
        can be enforced by replacing $\boldsymbol\varrho$ with
        $c\boldsymbol\varrho$ for sufficiently small $c>0$.
        This rescaling preserves reciprocal-weight summability, and hence
        \eqref{eq:lognormal-assumption-a}, at the price of worse constants in the
        resulting estimates.}

Finally, let us remark that, for the lognormal model, the coefficient $a(x,\yb)$ is positive for
$\mu$-almost every $\yb$, but it does not admit parameter-uniform ellipticity
constants over all $\yb\in\R^d$.
Consequently, the elementary bounded-gradient estimate
\eqref{eq:Klambda-bounded-gradient} cannot be applied under base Gaussian
sampling on the basis of a parameter-uniform gradient bound.
This argument therefore does not yield a finite base-measure coherence bound
for the lognormal model.
By contrast, the two kernel diagonal measures above turn weighted Hermite
regularity into finite {regularized coherence} bounds.
The {regularized} kernel diagonal measure gives the sharper kernel effective-dimension bound
\eqref{eq:lognormal-ridge-diagonal-coherence-bound}, but its definition depends
on the {regularization} scale $\lambda$ and the smoothness scale
$\norm{H}_{\mathcal L(\ell^2(d))}$.
The $\theta$-kernel diagonal measure
\eqref{eq:smoothness-diagonal-proposal} is independent of the {regularization} scale and
gives the explicit upper bound \eqref{eq:lognormal-ridge-coherence-bound}.

\subsection{The uniform model}
\label{subsec:uniform-model}

We next consider the affine uniform model from \cite{bachmayr2017sparsei}.
{Throughout this subsection $\mu$ is the product uniform probability measure supported on $[-1,1]^d\subset\Gamma$, hence we restrict $\yb\in[-1,1]^d$.}
Here
\begin{equation}
    \label{eq:uniform-coefficient}
    a(x,\yb)
    =
    \bar a(x)+\sum_{m\geq 1}y_m\psi_m(x).
\end{equation}
As before, $d$ may be finite or countably infinite.
We assume that $\bar a\in L^\infty(D)$ is strictly positive and write
\begin{equation}
    \bar a_{\min}
    :=
    \operatorname*{ess\,inf}_{x\in D}\bar a(x)>0.
\end{equation}
We assume that there exists a sequence
$\boldsymbol\varrho=(\varrho_m)_{m\geq1}$ with $\varrho_m>1$ satisfying the
weighted uniform ellipticity condition
\begin{equation}
    \label{eq:uniform-weighted-ellipticity}
    \delta_{\boldsymbol\varrho}
    :=
    \left\|
    \frac{
        \sum_{m\geq 1}\varrho_m\abs{\psi_m}
    }{
        \bar a
    }
    \right\|_{L^\infty(D)}
    <1.
\end{equation}
Since $\varrho_m>1$, this implies the unweighted uniform ellipticity condition
\begin{equation}
    \label{eq:uniform-ellipticity-bounds}
    (1-\delta_{\boldsymbol\varrho})\bar a(x)
    \leq
    a(x,\yb)
    \leq
    (1+\delta_{\boldsymbol\varrho})\bar a(x).
\end{equation}
Thus, the weak problem \eqref{eq:parametric-elliptic-weak-form} is uniformly
coercive over the full parameter domain.

Let $\{L_n\}_{n\geq 0}$ denote the $L^2_{\mu_1}([-1,1])$-orthonormal Legendre
polynomials, where $\mu_1$ is the uniform probability measure on $[-1,1]$.
For $\nu\in\mathcal F$, set
\begin{equation}
    L_\nu(\yb)
    :=
    \prod_{m\geq 1}L_{\nu_m}(y_m),
    \qquad
    \omega_\nu
    :=
    \prod_{m\geq 1}\sqrt{2\nu_m+1}.
\end{equation}
Then, $\{L_\nu\}_{\nu\in\mathcal F}$ is the tensorized normalized Legendre basis
of $L^2_\mu(\Gamma)$, which satisfies the uniform bound
\begin{equation}
    \norm{L_\nu}_{L^\infty(\Gamma)}\leq\omega_\nu
\end{equation}
since
\begin{equation}
    \sup_{y\in[-1,1]}\abs{L_n(y)}
    \leq
    \sqrt{2n+1}.
\end{equation}

The weighted ellipticity condition also yields anisotropic decay of the
Legendre coefficients.
More precisely, under \eqref{eq:uniform-weighted-ellipticity},
\cite[Theorem 3.1]{bachmayr2017sparsei} gives the Legendre expansion
\begin{equation}
    u(\yb)
    =
    \sum_{\nu\in\mathcal F}u_\nu L_\nu(\yb)
\end{equation}
with
\begin{equation}
    \label{eq:uniform-solution-legendre-regularity}
    \sum_{\nu\in\mathcal F}
    \left(
    \frac{\boldsymbol\varrho^\nu}{\omega_\nu}
    \norm{u_\nu}_{\calV}
    \right)^2
    \leq
    C_{\mathrm{UE}},
\end{equation}
where
\begin{equation}
    \label{eq:uniform-pde-constant}
    C_{\mathrm{UE}}
    :=
    \frac{
        (2-\delta_{\boldsymbol\varrho})(1+\delta_{\boldsymbol\varrho})
        \norm{\bar a}_{L^\infty(D)}^2
        \norm{f_D}_{\calV'}^2
    }{
        2(1-\delta_{\boldsymbol\varrho})^4\bar a_{\min}^4
    }.
\end{equation}

Motivated by \eqref{eq:uniform-solution-legendre-regularity}, for any sequence
$\boldsymbol\zeta=(\zeta_m)_{m\geq 1}$ with $\zeta_m>1$, we define the weighted
Legendre space
\begin{equation}
    \label{eq:uniform-legendre-rkhs}
    \mathcal A_{\boldsymbol\zeta}
    :=
    \left\{
    p=\sum_{\nu\in\mathcal F}p_\nu L_\nu
    \in L^2_\mu(\Gamma)
    :
    \norm{p}_{\mathcal A_{\boldsymbol\zeta}}^2
    :=
    \sum_{\nu\in\mathcal F}
    \frac{\boldsymbol\zeta^{2\nu}}{\omega_\nu^2}
    \abs{p_\nu}^2
    <
    \infty
    \right\}.
\end{equation}
The embedding of this weighted space into $L^2_\mu(\Gamma)$ is diagonal in the
Legendre basis.
In the notation of Section~\ref{sec:sample-complexity-smoothness}, the
eigenfunctions are $\phi_\nu=L_\nu$ and the corresponding eigenvalues are
$\gamma_\nu=\omega_\nu^2\boldsymbol\zeta^{-2\nu}$.
The corresponding $\theta$-kernel diagonal and trace are
\begin{equation}
    \label{eq:uniform-legendre-diagonal}
    \kappa_{\boldsymbol\zeta}^\theta(\yb)
    :=
    \sum_{\nu\in\mathcal F}
    \omega_\nu^{2\theta}
    \boldsymbol\zeta^{-2\theta\nu}
    L_\nu(\yb)^2
\end{equation}
and
\begin{equation}
    \label{eq:uniform-legendre-normalization}
    Z_{\theta,\boldsymbol\zeta}
    =
    \int_\Gamma
    \kappa_{\boldsymbol\zeta}^\theta(\yb)
    \,\dif\mu(\yb)
    =
    \sum_{\nu\in\mathcal F}
    \omega_\nu^{2\theta}
    \boldsymbol\zeta^{-2\theta\nu}.
\end{equation}

\begin{lemma}[Admissible exponents and bounded Legendre kernel diagonal]
    \label{lem:uniform-legendre-admissible-scales}
    {Let $\boldsymbol\zeta=(\zeta_m)_{m\geq 1}$ satisfy
        $\zeta_m>1$ and $(\zeta_m^{-1})_{m\geq1}\in\ell^2$.
        Then $\mathcal A_{\boldsymbol\zeta}$ is an RKHS.
        Moreover, every $\theta\in(0,1]$ satisfying}
    \begin{equation}
        \label{eq:uniform-legendre-theta-admissibility}
        {(\zeta_m^{-1})_{m\geq 1}\in\ell^{2\theta}}
    \end{equation}
    {is an admissible exponent for
    $\mathcal A_{\boldsymbol\zeta}$ and satisfies}
    \begin{equation}
        \label{eq:uniform-legendre-diagonal-bound}
        {\norm{
                \kappa_{\boldsymbol\zeta}^\theta
            }_{L^\infty(\Gamma)}
            <\infty.}
    \end{equation}
\end{lemma}

\begin{proof}
    {Using the Legendre bound and
        $(\zeta_m^{-1})_{m\geq1}\in\ell^2$, we have}
    \begin{equation*}
        {
            \sum_{\nu\in\mathcal F}
            \omega_\nu^2
            \boldsymbol\zeta^{-2\nu}
            L_\nu(\yb)^2
            \leq
            \prod_{m\geq1}
            \left(
            \sum_{n\geq0}
                (2n+1)^2\zeta_m^{-2n}
            \right)
            <\infty.}
    \end{equation*}
    {The product is finite because
    $(\zeta_m^{-2})_{m\geq1}\in\ell^1$ and
    $\sum_{n\geq0}(2n+1)^2t^n=1+O(t)$ as $t\to0$.
    Hence point evaluation is bounded on
    $\mathcal A_{\boldsymbol\zeta}$, which is therefore an RKHS.}

    Since the Legendre basis is orthonormal in $L^2_\mu(\Gamma)$,
    \eqref{eq:uniform-legendre-normalization} gives
    \begin{equation}
        Z_{\theta,\boldsymbol\zeta}
        =
        \sum_{\nu\in\mathcal F}
        \omega_\nu^{2\theta}
        \boldsymbol\zeta^{-2\theta\nu}
        =
        \prod_{m\geq 1}
        \left(
        \sum_{n\geq 0}
            (2n+1)^\theta \zeta_m^{-2\theta n}
        \right),
    \end{equation}
    where the infinite product is understood through monotone convergence over
    finite coordinate sets.
    Set $t_m:=\zeta_m^{-2\theta}$.
    Since $\sum_m t_m<\infty$, there exist $q\in(0,1)$ and
    $\bar m\in\N$ such that $t_m\leq q$ for all $m\geq\bar m$.
    For those $m$, the function
    \begin{equation}
        t
        \mapsto
        \sum_{n\geq 1}
        (2n+1)^\theta t^{n-1}
    \end{equation}
    is bounded on $[0,q]$.
    Hence there exists $C_{\theta,q}>0$ such that
    \begin{equation}
        \sum_{n\geq 0}
        (2n+1)^\theta t_m^n
        \leq
        1+C_{\theta,q}t_m,
        \qquad
        m\geq\bar m.
    \end{equation}
    The infinite product therefore converges, while the finitely many remaining
    factors are finite because $\zeta_m>1$.
    This proves $Z_{\theta,\boldsymbol\zeta}<\infty$.

    It remains to prove the uniform $\theta$-kernel diagonal bound.
    Recalling that $\norm{L_\nu}_{L^\infty(\Gamma)}\leq\omega_\nu$, we have
    \begin{equation}
        \kappa_{\boldsymbol\zeta}^\theta(\yb)
        \leq
        \sum_{\nu\in\mathcal F}
        \omega_\nu^{2(1+\theta)}
        \boldsymbol\zeta^{-2\theta\nu}
        =
        \prod_{m\geq 1}
        \left(
        \sum_{n\geq 0}
            (2n+1)^{1+\theta}\zeta_m^{-2\theta n}
        \right).
    \end{equation}
    The same finite-tail argument applies with the polynomial factor
    $(2n+1)^{1+\theta}$ and proves
    \eqref{eq:uniform-legendre-diagonal-bound}.
    {Finally, the constant Legendre mode shows that
    $\kappa_{\boldsymbol\zeta}^\theta$ is strictly positive.
    Together with $Z_{\theta,\boldsymbol\zeta}<\infty$, this proves that
    $\theta$ is admissible.}
\end{proof}

For Legendre expansions, differentiation mixes several lower polynomial
degrees.
Reducing the weights from $\boldsymbol\varrho$ to
$\boldsymbol\zeta=\alpha\boldsymbol\varrho$ {with $\alpha<1$}
absorbs this growth and also
controls the weighted sum of all coordinate derivatives.

\begin{lemma}[Legendre differentiation and weighted derivative summability]
    \label{lem:uniform-legendre-alpha-radius-loss}
    Let $\boldsymbol\varrho=(\varrho_m)_{m\geq 1}$ satisfy
        {$\varrho_m>1$ for every $m\geq1$ and}
    $(\varrho_m^{-1})_{m\geq1}\in\ell^2$, {so that}
    $\underline{\varrho}:=\inf_m\varrho_m>1$.
    Choose $\alpha\in(\underline{\varrho}^{-1},1)$ and set
    \begin{equation}
        \label{eq:uniform-zeta-alpha-choice}
        \zeta_m
        :=
        \alpha\varrho_m,
        \qquad
        m\geq 1.
    \end{equation}
    Then there exists a constant $c_\alpha>0$, depending only on $\alpha$, such
    that, for every $m\geq 1$ and every
    $p\in\mathcal A_{\boldsymbol\varrho}$,
    one has $\partial_m p\in\mathcal A_{\boldsymbol\zeta}$ and
    \begin{equation}
        \label{eq:uniform-legendre-derivative-component-bound}
        \norm{\partial_m p}_{\mathcal A_{\boldsymbol\zeta}}^2
        \leq
        c_\alpha\varrho_m^{-2}
        \norm{p}_{\mathcal A_{\boldsymbol\varrho}}^2.
    \end{equation}
    Moreover,
    \begin{equation}
        \label{eq:uniform-legendre-derivative-weighted-sum-bound}
        \sum_{m\geq1}
        \varrho_m^2
        \norm{\partial_m p}_{\mathcal A_{\boldsymbol\zeta}}^2
        \leq
        C_{\mathrm{unif}}
        \norm{p}_{\mathcal A_{\boldsymbol\varrho}}^2,
    \end{equation}
    where
    \begin{equation}
        \label{eq:uniform-legendre-derivative-weighted-sum-constant}
        C_{\mathrm{unif}}
        :=
        c_\alpha
        \sup_{s\geq1}
        s\alpha^{2(s-1)}.
    \end{equation}
\end{lemma}

\begin{proof}
    Since $\alpha>\underline{\varrho}^{-1}$, the choice
    $\zeta_m=\alpha\varrho_m$ gives
    \begin{equation}
        1<\zeta_m<\varrho_m,
        \qquad
        m\geq 1.
    \end{equation}
    We use the differentiation formula for the normalized Legendre basis,
    \begin{equation}
        \label{eq:normalized-legendre-derivative-formula}
        L_n'
        =
        \sum_{\substack{0\leq k<n\\ n-k\ \mathrm{odd}}}
        \sqrt{(2n+1)(2k+1)}\,L_k .
    \end{equation}
    Fix $m\geq 1$ and write a multi-index as $(\eta,n)$, where $n$ denotes the
    $m$-th component and $\eta$ collects all remaining components.
    We use the subscript $\neq m$ for sequences with their $m$-th component
    omitted.
    Thus
    \begin{equation}
        p(\yb)
        =
        \sum_{\eta}
        \sum_{n\geq 0}
        p_{\eta,n}
        L_\eta(\yb_{\neq m})L_n(y_m).
    \end{equation}
    By \eqref{eq:normalized-legendre-derivative-formula}, the coefficient of
    $L_\eta(\yb_{\neq m})L_k(y_m)$ in $\partial_m p$ is
    \begin{equation}
        q_{\eta,k}
        =
        \sum_{\substack{n>k\\ n-k\ \mathrm{odd}}}
        \sqrt{(2n+1)(2k+1)}\,p_{\eta,n}.
    \end{equation}
    {Since
    $\omega_{(\eta,k)}^2=\omega_\eta^2(2k+1)$ and
    $\boldsymbol\zeta^{2(\eta,k)}
        =\boldsymbol\zeta_{\neq m}^{2\eta}\zeta_m^{2k}$,}
    \begin{equation*}
        {
            \begin{aligned}
                \frac{\boldsymbol\zeta^{2(\eta,k)}}
                {\omega_{(\eta,k)}^2}
                \abs{q_{\eta,k}}^2
                 & =
                \frac{\boldsymbol\zeta_{\neq m}^{2\eta}}{\omega_\eta^2}
                \frac{\zeta_m^{2k}}{2k+1}
                \left|
                \sum_{\substack{n>k \\ n-k\ \mathrm{odd}}}
                \sqrt{(2n+1)(2k+1)}\,p_{\eta,n}
                \right|^2
                \\
                 & =
                \frac{\boldsymbol\zeta_{\neq m}^{2\eta}}{\omega_\eta^2}
                \zeta_m^{2k}
                \left|
                \sum_{\substack{n>k \\ n-k\ \mathrm{odd}}}
                \sqrt{2n+1}\,p_{\eta,n}
                \right|^2.
            \end{aligned}}
    \end{equation*}
    {Summing over $\eta$ and $k$ gives}
    \begin{equation}
        \begin{aligned}
            \norm{\partial_m p}_{\mathcal A_{\boldsymbol\zeta}}^2
             & =
            \sum_{\eta}
            \frac{\boldsymbol\zeta_{\neq m}^{2\eta}}{\omega_\eta^2}
            \sum_{k\geq 0}
            \zeta_m^{2k}
            \left|
            \sum_{\substack{n>k \\ n-k\ \mathrm{odd}}}
            \sqrt{2n+1}\,p_{\eta,n}
            \right|^2 .
        \end{aligned}
    \end{equation}
    {For each fixed $\eta$, applying weighted Cauchy--Schwarz for
    every $k$ and then interchanging the order of summation gives}
    \begin{equation}
        {
            \begin{aligned}
                 & \sum_{k\geq0}
                \zeta_m^{2k}
                \left|
                \sum_{\substack{n>k \\ n-k\ \mathrm{odd}}}
                \sqrt{2n+1}\,p_{\eta,n}
                \right|^2
                \\
                 & \quad\leq
                \sum_{k\geq0}
                \left(
                \sum_{\substack{n>k \\ n-k\ \mathrm{odd}}}
                \zeta_m^{2k}(2n+1)^2\varrho_m^{-2n}
                \right)
                \left(
                \sum_{\substack{n>k \\ n-k\ \mathrm{odd}}}
                \frac{\varrho_m^{2n}}{2n+1}
                \abs{p_{\eta,n}}^2
                \right)
                \\
                 & \quad\leq
                \left(
                \sum_{n\geq1}
                    (2n+1)^2\varrho_m^{-2n}
                \sum_{k=0}^{n-1}\zeta_m^{2k}
                \right)
                \left(
                \sum_{n\geq1}
                \frac{\varrho_m^{2n}}{2n+1}
                \abs{p_{\eta,n}}^2
                \right).
            \end{aligned}}
    \end{equation}
    Since $\zeta_m=\alpha\varrho_m$ and
    $\alpha\varrho_m\geq\alpha\underline{\varrho}>1$, we have
    \begin{equation}
        \sum_{k=0}^{n-1}\zeta_m^{2k}
        =
        \sum_{k=0}^{n-1}\alpha^{2k}\varrho_m^{2k}
        \leq
        n\alpha^{2(n-1)}\varrho_m^{2(n-1)}.
    \end{equation}
    Hence
    \begin{equation}
        \sum_{n\geq1}
        (2n+1)^2\varrho_m^{-2n}
        \sum_{k=0}^{n-1}\zeta_m^{2k}
        \leq
        \varrho_m^{-2}
        \sum_{n\geq 1}
        n(2n+1)^2\alpha^{2(n-1)}
        =:
        c_\alpha\varrho_m^{-2}.
    \end{equation}
    Since $\alpha<1$, the constant $c_\alpha$ is finite and depends only on
    $\alpha$.
    Combining this estimate with the norm identity above and using
    \begin{equation*}
        \boldsymbol\zeta_{\neq m}^{2\eta}
        =
        \alpha^{2\norm{\eta}_{\ell^1}}
        \boldsymbol\varrho_{\neq m}^{2\eta}
    \end{equation*}
    therefore gives
    \begin{equation}
        \label{eq:uniform-legendre-derivative-retained-bound}
        \norm{\partial_m p}_{\mathcal A_{\boldsymbol\zeta}}^2
        \leq
        c_\alpha\varrho_m^{-2}
        \sum_{\substack{\nu\in\mathcal F\\ \nu_m\geq1}}
        \alpha^{2\sum_{j\neq m}\nu_j}
        \frac{\boldsymbol\varrho^{2\nu}}{\omega_\nu^2}
        \abs{p_\nu}^2.
    \end{equation}
    Dropping the factor involving $\alpha$ proves
    \eqref{eq:uniform-legendre-derivative-component-bound}.

    To prove \eqref{eq:uniform-legendre-derivative-weighted-sum-bound}, let
    $s(\nu)$ denote the number of nonzero components of $\nu$.
    If $\nu_m\geq1$, then
    \begin{equation}
        \sum_{j\neq m}\nu_j
        \geq
        s(\nu)-1.
    \end{equation}
    Hence
    \begin{equation}
        \sum_{\substack{m\geq1\\ \nu_m\geq1}}
        \alpha^{2\sum_{j\neq m}\nu_j}
        \leq
        s(\nu)\alpha^{2(s(\nu)-1)}
        \leq
        \sup_{s\geq1}s\alpha^{2(s-1)}.
    \end{equation}
    Multiplying \eqref{eq:uniform-legendre-derivative-retained-bound} by
    $\varrho_m^2$, summing over $m$, and using the preceding bound gives
    \eqref{eq:uniform-legendre-derivative-weighted-sum-bound}.

    The calculation above applies first to finite Legendre expansions.
    Since $(\varrho_m^{-1})_{m\geq1}$ and
    $(\zeta_m^{-1})_{m\geq1}$ belong to $\ell^2$, both weighted spaces embed
    continuously into $L^2_\mu(\Gamma)$.
    Passing to the limit in finite Legendre partial sums proves the estimates
    for general $p\in\mathcal A_{\boldsymbol\varrho}$ and identifies the limits
    with the weak coordinate derivatives.
\end{proof}

We now combine the Legendre coefficient estimate with the abstract RKHS
consequences from Section~\ref{sec:sample-complexity-smoothness}.

\begin{theorem}[Smoothness consequences for the affine uniform elliptic model]
    \label{thm:uniform-pde-smoothness-consequences}
    Assume \eqref{eq:uniform-weighted-ellipticity} for a sequence
    $\boldsymbol\varrho$ with $\varrho_m>1$.
    Let $\theta\in(0,1]$ satisfy
    \begin{equation}
        \label{eq:uniform-smoothness-exponent-assumption}
        (\varrho_m^{-1})_{m\geq 1}\in\ell^{2\theta}.
    \end{equation}
    Then $\underline{\varrho}:=\inf_m\varrho_m>1$.
    Choose $\alpha\in(\underline{\varrho}^{-1},1)$ and set
    $\zeta_m=\alpha\varrho_m$.
    Let $u(\yb)$ solve \eqref{eq:parametric-elliptic-weak-form} with the
    affine coefficient \eqref{eq:uniform-coefficient}, let $Q\in\calV'$,
    and set $f=Q\circ u$ as in \eqref{eq:pde-quantity-of-interest}.
    Set
    \begin{equation}
        \label{eq:uniform-rkhs-choice}
        \mathcal H
        :=
        \mathcal A_{\boldsymbol\zeta}.
    \end{equation}
    Lemma~\ref{lem:uniform-legendre-admissible-scales} ensures that
    $\mathcal H$ is an RKHS and that $\theta$ is an admissible exponent.
    Set $\beta:=\theta/(1+\theta)$.
    Then the following statements hold.

    \begin{enumerate}[label=(\roman*)]
        \item The target and its coordinate derivatives satisfy
              \begin{equation}
                  \label{eq:uniform-target-and-gradient-regularity}
                  f\in\mathcal A_{\boldsymbol\varrho},
                  \qquad
                  \partial_m f
                  \in
                  \mathcal A_{\boldsymbol\zeta},
                  \quad m\geq 1.
              \end{equation}
              The coordinate derivatives are jointly summable in $\mathcal H$:
              \begin{equation*}
                  B_f
                  :=
                  \sum_{m\geq1}
                  \varrho_m^2
                  \norm{\partial_m f}_{\mathcal H}^2
                  <
                  \infty.
              \end{equation*}
              In particular, $f\in H^1_\mu(\Gamma)$ and
              $\nabla f=(\partial_m f)_{m\geq1}$ belongs to
              $L^2_\mu(\Gamma;\ell^2(d))$.

        \item Define the scalar-projection operator
              \begin{equation}
                  \label{eq:uniform-scalar-projection-operator}
                  G:\ell^2(d)\to\mathcal H,
                  \qquad
                  Gv
                  :=
                  \sum_{m\geq1}v_m\partial_m f.
              \end{equation}
              Then $G$ is well defined and bounded, and the smoothness
              covariance satisfies
              \begin{equation}
                  \label{eq:uniform-schatten-smoothness}
                  H
                  =
                  G^*G
                  \in
                  \calS_\beta(\ell^2(d)),
                  \quad
                  \norm{H}_{\calS_\beta(\ell^2(d))}
                  \leq
                  B_f
                  \left(
                  \sum_{m\geq1}
                  \varrho_m^{-2\theta}
                  \right)^{1/\theta}.
              \end{equation}

        \item The {regularized} kernel diagonal measure
              $\rho_{\alpha_\lambda}$ from
              \eqref{eq:smoothness-ridge-diagonal-proposal} is well defined for
              every $\lambda>0$ and satisfies
              \begin{equation}
                  \label{eq:uniform-ridge-diagonal-coherence-bound}
                  K_{\rho_{\alpha_\lambda},\lambda}(\nabla f)
                  \leq
                  d_{\mathrm{eff}}(\alpha_\lambda,\gamma)
                  \leq
                  Z_{\theta,\boldsymbol\zeta}
                  \norm{H}_{\mathcal L(\ell^2(d))}^{\theta}
                  \lambda^{-\theta}.
              \end{equation}
              The $\theta$-kernel diagonal measure $\rho^\theta$ from
              \eqref{eq:smoothness-diagonal-proposal} satisfies, for every
              $\lambda>0$,
              \begin{equation}
                  \label{eq:uniform-diagonal-ridge-coherence-bound}
                  K_{\rho^\theta,\lambda}(\nabla f)
                  \leq
                  Z_{\theta,\boldsymbol\zeta}
                  \norm{H}_{\mathcal L(\ell^2(d))}^{\theta}
                  \lambda^{-\theta}.
              \end{equation}
              Moreover, the base uniform measure is quasi-optimal for the {regularized}
              kernel diagonal {at scale $\alpha_\lambda$}. More precisely, for
              every $\lambda>0$,
              \begin{equation}
                  \label{eq:uniform-base-ridge-coherence-bound}
                  K_{\mu,\lambda}(\nabla f)
                  \leq
                  \norm{
                      \kappa_{\boldsymbol\zeta}^\theta
                  }_{L^\infty(\Gamma)}
                  \norm{H}_{\mathcal L(\ell^2(d))}^{\theta}
                  \lambda^{-\theta}.
              \end{equation}

        \item The active subspace tail energy satisfies, for every $r\geq 1$, the a priori bound
              \begin{equation}
                  \label{eq:uniform-active-subspace-tail-bound}
                  \err(\Pi_r)^2
                  \leq
                  Z_{\theta,\boldsymbol\zeta}^{1/\theta}
                  \norm{H}_{\calS_\beta(\ell^2(d))}
                  r^{-2/\theta}.
              \end{equation}
    \end{enumerate}

\end{theorem}

\begin{proof}
    Since $\zeta_m=\alpha\varrho_m$, condition
    \eqref{eq:uniform-smoothness-exponent-assumption} is equivalent to
    $(\zeta_m^{-1})_{m\geq 1}\in\ell^{2\theta}$.
    Lemma~\ref{lem:uniform-legendre-admissible-scales} therefore gives the
    asserted RKHS and admissibility properties.

    The Legendre coefficient estimate
    \eqref{eq:uniform-solution-legendre-regularity} gives
    \begin{equation}
        \sum_{\nu\in\mathcal F}
        \frac{\boldsymbol\varrho^{2\nu}}{\omega_\nu^2}
        \norm{u_\nu}_{\calV}^2
        \leq
        C_{\mathrm{UE}}.
    \end{equation}
    Since $Q\in\calV'$ is bounded and linear, the Legendre coefficients of
    $f=Q\circ u$ are $f_\nu=Q(u_\nu)$.
    Hence
    \begin{equation}
        \label{eq:uniform-target-rkhs-bound}
        \norm{f}_{\mathcal A_{\boldsymbol\varrho}}^2
        =
        \sum_{\nu\in\mathcal F}
        \frac{\boldsymbol\varrho^{2\nu}}{\omega_\nu^2}
        \abs{f_\nu}^2
        \leq
        \norm{Q}_{\calV'}^2 C_{\mathrm{UE}}
        <
        \infty.
    \end{equation}
    This proves $f\in\mathcal A_{\boldsymbol\varrho}$.

    Since $\zeta_m=\alpha\varrho_m$ with
    $\alpha\in(\underline{\varrho}^{-1},1)$, we have
    $1<\zeta_m<\varrho_m$.
    Applying Lemma~\ref{lem:uniform-legendre-alpha-radius-loss} with $p=f$ gives
    \begin{equation}
        \label{eq:uniform-gradient-weighted-summability}
        B_f
        =
        \sum_{m\geq1}
        \varrho_m^2
        \norm{\partial_m f}_{\mathcal H}^2
        \leq
        C_{\mathrm{unif}}
        \norm{f}_{\mathcal A_{\boldsymbol\varrho}}^2.
    \end{equation}
    Thus $B_f<\infty$.
    The embedding of $\mathcal H$ into $L^2_\mu(\Gamma)$ is bounded, and
    $\varrho_m\geq\underline{\varrho}$.
    Hence \eqref{eq:uniform-gradient-weighted-summability} implies
    \begin{equation*}
        \sum_{m\geq1}
        \norm{\partial_m f}_{L^2_\mu(\Gamma)}^2
        <\infty.
    \end{equation*}
    The coordinate characterization
    \eqref{eq:coordinate-weak-derivative-characterization} shows that
    $f\in H^1_\mu(\Gamma)$ and identifies its gradient with
    $(\partial_m f)_{m\geq1}$.

    The factorization used in the proof of
    Theorem~\ref{thm:lognormal-pde-smoothness-consequences} now applies with the
    present RKHS and weights.
    Namely, the operator with columns
    $\varrho_m\partial_m f$ is Hilbert--Schmidt by
    \eqref{eq:uniform-gradient-weighted-summability}, while the diagonal
    operator with entries $\varrho_m^{-1}$ belongs to $\calS_{2\theta}$.
    Their product is the operator $G$ in
    \eqref{eq:uniform-scalar-projection-operator}.
    Thus $G$ is well defined, $H=G^*G$, and the Schatten--H\"older inequality
    gives \eqref{eq:uniform-schatten-smoothness}.

    The constant Legendre mode and
    \eqref{eq:smoothness-ridge-theta-dimension-comparison} show that the {regularized}
    kernel diagonal measure is well defined and equivalent to $\mu$.
    Proposition~\ref{prop:smoothness-ridge-diagonal-sampling} and
    \eqref{eq:smoothness-ridge-to-theta-effective-dimension} give
    \eqref{eq:uniform-ridge-diagonal-coherence-bound}.
    Admissibility gives finiteness of $\kappa^\theta$ $\mu$-a.e., while the
    positive constant Legendre mode makes it strictly positive.
    Hence $\rho^\theta\sim\mu$, and the same corollary gives
    \eqref{eq:uniform-diagonal-ridge-coherence-bound}.

    The {pointwise comparison
            \eqref{eq:smoothness-ridge-theta-dimension-comparison} and}
    \eqref{eq:uniform-legendre-diagonal-bound} {give}
    \begin{equation*}
        {\mu\text{-}\operatorname*{ess\,sup}_{\yb\in\Gamma}
            \kappa_{\alpha_\lambda}(\yb)
            \leq
            \alpha_\lambda^{-\theta}
            \norm{\kappa_{\boldsymbol\zeta}^\theta}_{L^\infty(\Gamma)}.}
    \end{equation*}
    {Thus} $\mu$ is quasi-optimal for the {regularized} kernel diagonal {at scale
            $\alpha_\lambda$} with constant
    \begin{equation*}
        c_2
        =
        \frac{
            {\alpha_\lambda^{-\theta}}
            \norm{\kappa_{\boldsymbol\zeta}^\theta}_{L^\infty(\Gamma)}
        }{
            {d_{\mathrm{eff}}(\alpha_\lambda},{\gamma)}
        }.
    \end{equation*}
    {This value is at least one because the denominator is the $\mu$-integral
    of $\kappa_{\alpha_\lambda}$.}
    Applying {Proposition}~{\ref{prop:smoothness-ridge-diagonal-sampling}} with
    this value
    gives \eqref{eq:uniform-base-ridge-coherence-bound}.

    Finally, Proposition~\ref{prop:smoothness-active-subspace-tail} gives
    \eqref{eq:uniform-active-subspace-tail-bound}, since
    $q_{\theta,\beta}=2/\theta$.
\end{proof}

As in the lognormal model, the condition
$(\varrho_m^{-1})_{m\geq 1}\in\ell^{2\theta}$ gives kernel spectral
summability, which enters both the coherence and tail estimates.
Together with the weighted derivative estimate
\eqref{eq:uniform-gradient-weighted-summability}, it also gives the
coordinate-index summability $H\in\calS_\beta$, which enters only the tail
estimate.
Here $\beta=\theta/(1+\theta)$, so the resulting {regularized coherence} growth is
$\lambda^{-\theta}$ and the active subspace tail energy decay is
$r^{-2/\theta}$.
Consequently, the choices of $r_\tol$ and $\lambda_\tol$ in
\eqref{eq:lognormal-informal-tolerance-rates} also apply here, and the leading
algebraic power of $\tol^{-1}$ in the sufficient sample bound is
$2\theta+\theta^2$.

Recall, however, that we cannot convert this active subspace tail energy bound into a
ridge-reconstruction error bound.
Establishing the required subspace conditional Poincar\'e inequality for
uniform product measures is delicate and beyond the scope of this work; see
Remark~\ref{rem:uniform-product-poincare}.

The affine uniform model offers an additional certified sampling option.
The uniform bound on the normalized Legendre basis makes the base measure
quasi-optimal for the {regularized} kernel diagonal {at every regularization scale}, so the
$\lambda^{-\theta}$ coherence rate does not require a change of measure.
Sampling from $\rho^\theta$ replaces the base-measure constant
$\norm{\kappa_{\boldsymbol\zeta}^\theta}_{L^\infty(\Gamma)}$ by
$Z_{\theta,\boldsymbol\zeta}$, while the {regularized} kernel diagonal measure retains
the sharper kernel effective-dimension bound in
\eqref{eq:uniform-ridge-diagonal-coherence-bound}.

Under the hypotheses of Theorem~\ref{thm:uniform-pde-smoothness-consequences},
uniform ellipticity also gives a direct bounded-gradient estimate under the
base measure.
Differentiating the weak problem with respect to $y_m$ gives
\begin{equation*}
    \int_D
    a(x,\yb)\nabla\partial_m u(x,\yb)\cdot\nabla v(x)
    \,\dif x
    =
    -
    \int_D
    \psi_m(x)\nabla u(x,\yb)\cdot\nabla v(x)
    \,\dif x,
    \qquad
    v\in\calV.
\end{equation*}
The ellipticity bounds \eqref{eq:uniform-ellipticity-bounds}, the energy
estimate for $u$, and
$\varrho_m\abs{\psi_m}\leq\delta_{\boldsymbol\varrho}\bar a$ imply
\begin{equation*}
    \norm{\partial_m u(\yb)}_{\calV}
    \lesssim
    \varrho_m^{-1}
\end{equation*}
uniformly in $\yb$, with a constant independent of $m$.
Since $Q\in\calV'$, it follows that
\begin{equation*}
    \abs{\partial_m f(\yb)}
    \lesssim
    \varrho_m^{-1}
\end{equation*}
with the same uniformity.
Since $(\varrho_m^{-1})_{m\geq1}\in\ell^{2\theta}\subseteq\ell^2$, it follows
that
\begin{equation*}
    \mu\text{-}\operatorname*{ess\,sup}_{\yb\in\Gamma}
    \norm{\nabla f(\yb)}_{\ell^2(d)}^2
    \lesssim
    \sum_{m\geq1}\varrho_m^{-2}
    <
    \infty.
\end{equation*}
Thus, the bounded-gradient estimate \eqref{eq:Klambda-bounded-gradient} gives
the baseline dependence $K_{\mu,\lambda}(\nabla f)\lesssim\lambda^{-1}$, while the kernel diagonal estimate
\eqref{eq:uniform-base-ridge-coherence-bound} improves this dependence to
$\lambda^{-\theta}$ when $\theta<1$.

\section{Numerical illustrations}
\label{sec:numerical-experiments}

\newcommand{\hermitejointmarginalsfigure}{figures/hermite_joint_marginals_coherent_packet.pdf}

\subsection{A rare-activation Hermite model}
\label{subsec:numerics-hermite-rare-activation}
{
In this section, we construct a model with a simple covariance spectrum in which selected
gradient coordinates are difficult to resolve from Gaussian samples.
The example aims to show how the sampling measure can affect active subspace recovery.

Let $\mu=\mathcal N(0,1)^{\otimes\mathbb N}$, and let $H_n$ denote the
$L^2_{\mu_1}(\R)$-normalized Hermite polynomial of degree $n$.
We use the weighted Hermite RKHS
$\mathcal A_{\boldsymbol\varrho}^{(\mathfrak r)}$ from
Section~\ref{subsec:lognormal-model}.
Here we set $\mathfrak r=2$ and
$\varrho_j=j$.
Thus, for $p=\sum_{\nu\in\mathcal F}p_\nu H_\nu$,
\begin{equation*}
    \norm{p}_{\mathcal A_{\boldsymbol\varrho}^{(2)}}^2
    =
    \sum_{\nu\in\mathcal F}
    b_{\nu,2}(\boldsymbol\varrho)\abs{p_\nu}^2,
    \qquad
    b_{\nu,2}(\boldsymbol\varrho)
    =
    \prod_{j\geq1}
    \left(
    1+\nu_j j^2+\binom{\nu_j}{2}j^4
    \right).
\end{equation*}
The embedding eigenvalues are
$\gamma_\nu=b_{\nu,2}(\boldsymbol\varrho)^{-1}$.
For a univariate mode of degree $n$ in coordinate $j$, write
\begin{equation*}
    b_{n,2}(j)
    :=
    1+n j^2+\binom{n}{2}j^4.
\end{equation*}
For $j\geq1$ and $n\geq0$, set
\begin{equation*}
    a_{j,n}
    :=
    \exp(-j/2)\frac{j^{n/2}}{\sqrt{n!}}.
\end{equation*}
The Hermite generating function gives
\begin{equation*}
    \exp(\sqrt j\,y-j)
    =
    \sum_{n\geq0}a_{j,n}H_n(y),
    \qquad
    a_{j,n}^2
    =
    e^{-j}\frac{j^n}{n!}.
\end{equation*}
Thus, $a_{j,n}^2$ is the probability mass function of a random variable
$N\sim\operatorname{Poisson}(j)$.  In particular, for every function $\psi$
for which the expectation is finite,
\begin{equation}
    \label{eq:numerics-hermite-poisson-expectation}
    \sum_{n\geq0}a_{j,n}^2\psi(n)
    =
    \mathbb E[\psi(N)].
\end{equation}
Moreover, $\mathbb E[N]=j$ and
$\mathbb E[N(N-1)]=j^2$.
We now define the normalized function obtained by removing the
$H_0$, $H_1$, and $H_{2j}$ components from this expansion:
\begin{equation*}
    Q_j(y)
    :=
    \frac{
    \exp(\sqrt{j}\,y-j)
    -a_{j,0}H_0(y)
    -a_{j,1}H_1(y)
    -a_{j,2j}H_{2j}(y)
    }{
    \left(1-a_{j,0}^2-a_{j,1}^2-a_{j,2j}^2\right)^{1/2}
    }.
\end{equation*}
By construction,
\begin{equation*}
    \norm{Q_j}_{L^2_{\mu_1}}=1,
    \qquad
    \inp{Q_j}{H_n}_{L^2_{\mu_1}}=0,
    \quad
    n\in\{0,1,2j\}.
\end{equation*}
Let
\begin{equation*}
    \eta_j
    :=
    \begin{cases}
        1-\left(\dfrac{j}{j+1}\right)^8,
         & j\geq3\text{ odd}, \\
        0,
         & \text{otherwise},
    \end{cases}.
\end{equation*}
We define the centered univariate component
\begin{equation*}
    f_j(y)
    :=
    \sqrt{1-\eta_j}\,
    \frac{H_{2j+1}(y)}{\sqrt{2j+1}}
    +
    \frac{\sqrt{\eta_j}}{
        \left(1-a_{j,0}^2-a_{j,1}^2-a_{j,2j}^2\right)^{1/2}}
    \sum_{n\notin\{0,1,2j\}}
    a_{j,n}\frac{H_{n+1}(y)}{\sqrt{n+1}}.
\end{equation*}
The Hermite derivative identity and the definition of $Q_j$ give
\begin{equation*}
    f_j'(y)
    =
    \sqrt{1-\eta_j}\,H_{2j}(y)
    +
    \sqrt{\eta_j}\,Q_j(y),
    \qquad
    \norm{f_j'}_{L^2_{\mu_1}}=1.
\end{equation*}
We now define an additive model by
\begin{equation*}
    f(\yb)
    =
    \sum_{j\geq1}j^{-3}f_j(y_j),
\end{equation*}
to which we want to apply the active subspace method.
Since each $f_j'$ has zero mean and unit $L^2_{\mu_1}$-norm,
\begin{equation*}
    \norm{j^{-3}f_j'}_{L^2_{\mu_1}}^2
    =
    j^{-6}.
\end{equation*}
The gradient covariance is therefore diagonal:
\begin{equation}
    \label{eq:numerics-hermite-model}
    g_j(\yb)
    =
    j^{-3}f_j'(y_j),
    \qquad
    C
    =
    \operatorname{diag}(j^{-6})_{j\geq1}.
\end{equation}
Consequently,
\begin{equation}
    \label{eq:numerics-hermite-canonical-scale}
    \err(\Pi_r)^2
    =
    \sum_{j>r}j^{-6}
    \sim
    \frac{1}{5r^5},
    \qquad
    \lambda_r
    :=
    \frac{\err(\Pi_r)^2}{r}
    \sim
    \frac{1}{5r^6},
\end{equation}
where $\lambda_r$ is the rank-$r$ canonical regularization scale.

The choice of $\eta_j$ creates rare activation at odd target ranks.
At an odd target rank $r$, the $H_{2r}$ component of the last retained
gradient coordinate satisfies
\begin{equation*}
    \norm{
    r^{-3}\sqrt{1-\eta_r}\,H_{2r}
    }_{L^2_{\mu_1}}^2
    =
    (1-\eta_r)r^{-6}
    =
    \left(\frac{r}{r+1}\right)^2(r+1)^{-6}
    <
    \norm{
    (r+1)^{-3}H_{2(r+1)}
    }_{L^2_{\mu_1}}^2.
\end{equation*}
The function $Q_r$ is
concentrated in the Gaussian tails while retaining unit
$L^2_{\mu_1}$-norm.
Thus, if a finite sample misses $Q_r$, the observed energy in coordinate $r$
falls below the full energy of the adjacent excluded coordinate $r+1$, which
can reverse their empirical spectral ordering.

To specify the regularized kernel diagonal measure at the canonical scale $\lambda_r$,
we first compute the RKHS covariance norm, which determines
$\alpha_{\lambda_r}=\lambda_r/\norm{H}_{\mathcal L(\ell^2)}$.

\begin{lemma}[RKHS covariance of the Hermite model]
    The RKHS covariance $H$ is diagonal in the canonical basis of $\ell^2$,
    with diagonal entries
    \begin{equation*}
        h_j:=\norm{j^{-3}f_j'}_{\mathcal A_{\boldsymbol\varrho}^{(2)}}^2.
    \end{equation*}
    These satisfy $h_1=4$ and $h_j<2$ for every $j\geq2$.
    Consequently,
    \begin{equation}
        \label{eq:numerics-hermite-smoothness-scale}
        \norm{H}_{\mathcal L(\ell^2)}
        =
        4,
        \qquad
        \alpha_{\lambda_r}
        =
        \frac{\lambda_r}{4}.
    \end{equation}
\end{lemma}

\begin{proof}
    Different gradient coordinates have disjoint tensorized Hermite expansions,
    so the operator is diagonal.
    Hence, for $v\in\ell^2$,
    \begin{equation*}
        \inp{v}{Hv}_{\ell^2}
        =
        \sum_{j\geq1}
        \norm{j^{-3}f_j'}_{\mathcal A_{\boldsymbol\varrho}^{(2)}}^2
        \abs{v_j}^2.
    \end{equation*}
    Since $Q_j$ is $\mathcal A_{\boldsymbol\varrho}^{(2)}$-orthogonal to $H_{2j}$, the squared RKHS norm of the
    $j$-th gradient coordinate is
    \begin{equation*}
        \norm{j^{-3}f_j'}_{\mathcal A_{\boldsymbol\varrho}^{(2)}}^2
        =
        j^{-6}
        \left(
        (1-\eta_j)b_{2j,2}(j)
        +
        \frac{\eta_j
        \displaystyle\sum_{n\notin\{0,1,2j\}}
        a_{j,n}^2b_{n,2}(j)
        }{
        1-a_{j,0}^2-a_{j,1}^2-a_{j,2j}^2
        }
        \right).
    \end{equation*}
    Applying \eqref{eq:numerics-hermite-poisson-expectation} with
    $\psi(n)=b_{n,2}(j)$ gives
    \begin{equation*}
        \sum_{n\geq0}a_{j,n}^2b_{n,2}(j)
        =
        \mathbb E\left[
            1+j^2N+\frac{j^4}{2}N(N-1)
            \right]
        =
        1+j^3+\frac{j^6}{2}.
    \end{equation*}
    If $\eta_j=0$, then
    \begin{equation*}
        \norm{j^{-3}f_j'}_{\mathcal A_{\boldsymbol\varrho}^{(2)}}^2
        =
        \norm{j^{-3}H_{2j}}_{\mathcal A_{\boldsymbol\varrho}^{(2)}}^2
        =
        j^{-6}b_{2j,2}(j)
        =
        2-\frac1j+\frac{2}{j^3}+\frac{1}{j^6}.
    \end{equation*}
    This expression equals $4$ at $j=1$ and is smaller than $2$ for $j\geq2$.
    For odd $j\geq3$, the three removed Poisson masses satisfy
    \begin{equation*}
        a_{j,0}^2+a_{j,1}^2+a_{j,2j}^2
        =
        e^{-j}\left(1+j+\frac{j^{2j}}{(2j)!}\right)
        <
        \frac14.
    \end{equation*}
    The inequality follows by direct evaluation at $j=3$ while, for $j\geq5$, the
    bound $(2j)!\geq(2j/e)^{2j}$ gives
    \begin{equation*}
        e^{-j}\left(1+j+\frac{j^{2j}}{(2j)!}\right)
        \leq
        6e^{-5}+\left(\frac e4\right)^5
        <
        \frac14.
    \end{equation*}
    Consequently,
    \begin{equation*}
        \norm{j^{-3}Q_j}_{\mathcal A_{\boldsymbol\varrho}^{(2)}}^2
        =
        j^{-6}
        \frac{
        \displaystyle\sum_{n\notin\{0,1,2j\}}
        a_{j,n}^2b_{n,2}(j)
        }{
        1-a_{j,0}^2-a_{j,1}^2-a_{j,2j}^2
        }
        \leq
        \frac{\frac12+j^{-3}+j^{-6}}{3/4}
        <
        1.
    \end{equation*}
    For these odd indices, Hermite orthogonality also gives
    \begin{equation*}
        \norm{j^{-3}f_j'}_{\mathcal A_{\boldsymbol\varrho}^{(2)}}^2
        =
        (1-\eta_j)
        \norm{j^{-3}H_{2j}}_{\mathcal A_{\boldsymbol\varrho}^{(2)}}^2
        +
        \eta_j
        \norm{j^{-3}Q_j}_{\mathcal A_{\boldsymbol\varrho}^{(2)}}^2
        <
        2.
    \end{equation*}
    Thus every diagonal entry of $H$ except the first is smaller than $2$, whereas
    the first diagonal entry equals $4$.
    This proves \eqref{eq:numerics-hermite-smoothness-scale}.
\end{proof}

We want to compare measures that use different information about the gradient.
For each measure, we estimate the growth of its regularized coherence at $\lambda_r$;
Theorem~\ref{thm:slas-ridge-quasi-opt} converts this estimate into a sufficient
sample-size requirement for active subspace recovery.
The regularized Christoffel measure uses the exact covariance eigenfunctions.
In this model, its density, normalized by the covariance effective dimension,
is explicit and given by
\begin{equation*}
    \frac{\dif\rho_{\lambda_r}^\star}{\dif\mu}(\yb)
    =
    \frac{1}{d_{\mathrm{eff}}(\lambda_r)}
    \sum_{j\geq1}
    \frac{j^{-6}}{j^{-6}+\lambda_r}
    f_j'(y_j)^2,
    \quad
    d_{\mathrm{eff}}(\lambda_r)
    =
    \sum_{j\geq1}
    \frac{j^{-6}}{j^{-6}+\lambda_r}
    \leq
    r
    +
    \lambda_r^{-1}\sum_{j>r}j^{-6}
    =
    2r,
\end{equation*}
and Theorem~\ref{thm:slas-ridge-quasi-opt} gives a sufficient sample
size of order $r \log(r)$ for this measure.

The regularized kernel diagonal measure uses the full Hermite space, including modes
that are absent from the gradient.
Its normalizing constant is the kernel effective dimension
\begin{equation*}
    d_{\mathrm{eff}}(\alpha,\gamma)
    =
    \sum_{\nu\in\mathcal F}
    \frac{\gamma_\nu}{\gamma_\nu+\alpha},
\end{equation*}
which measures how many RKHS modes remain significant at scale $\alpha$.
The univariate linear modes belong to this space and have eigenvalues
$\gamma_{e_j}=(1+j^2)^{-1}$.
Consequently, for every admissible $\theta>1/2$, with
$Z_\theta:=\sum_{\nu\in\mathcal F}\gamma_\nu^\theta<\infty$,
\begin{equation*}
    \sum_{j\geq1}
    \frac{(1+j^2)^{-1}}{(1+j^2)^{-1}+\alpha}
    \leq
    d_{\mathrm{eff}}(\alpha,\gamma)
    \leq
    Z_\theta\alpha^{-\theta}.
\end{equation*}
The sum on the left equals
$\sum_{j\geq1}(1+\alpha(1+j^2))^{-1}$ and is of order
$\alpha^{-1/2}$ by an integral comparison.
Since $\alpha_{\lambda_r}$ is of order $r^{-6}$, choosing
$\theta=1/2+\epsilon/6$ gives the upper exponent $3+\epsilon$ for
$0<\epsilon\leq3$; larger values of $\epsilon$ follow from the case
$\epsilon=3$.  Thus, for every $\epsilon>0$, there are constants $c>0$ and
$C_\epsilon>0$ such that
\begin{equation*}
    c r^3
    \leq
    d_{\mathrm{eff}}(\alpha_{\lambda_r},\gamma)
    \leq
    C_\epsilon r^{3+\epsilon}.
\end{equation*}
The regularized kernel diagonal measure therefore pays for low-degree modes permitted
by the RKHS but absent from the gradient: its sufficient sample-size estimate
is hence of order $r^{3+\epsilon}$ for any $\epsilon>0$.
For a fixed $\theta$-kernel diagonal measure, the corresponding estimate is of
order $r^{6\theta}$.

Let us now describe how sampling can be performed in practice for these measures.
The non-Gaussian kernel diagonal measures share a common structure:
for nonnegative coefficients $(\omega_\nu)_{\nu\in\mathcal F}$ with
$Z_\omega=\sum_\nu\omega_\nu<\infty$, they have density
\begin{equation*}
    \frac{\dif\rho_\omega}{\dif\mu}(\yb)
    =
    \frac{1}{Z_\omega}
    \sum_{\nu\in\mathcal F}\omega_\nu H_\nu(\yb)^2.
\end{equation*}
Hence, we can first sample the degree index $\nu$ with probability
$\omega_\nu/Z_\omega$.  Conditional on $\nu$, the coordinates are independent and
coordinate $j$ has law $H_{\nu_j}^2\dif\mu_1$, which can be sampled by,
e.g., numerical inversion or rejection.
For the $\theta$-kernel diagonal measure,
$\omega_\nu=\gamma_\nu^\theta$ factorizes over the coordinates, with
\begin{equation*}
    \omega_\nu
    =
    \prod_{j\geq1}b_{\nu_j,2}(j)^{-\theta},
    \qquad
    \Prob{\nu_j=n}
    =
    \frac{b_{n,2}(j)^{-\theta}}
    {\sum_{k\geq0}b_{k,2}(j)^{-\theta}},
    \qquad
    n\geq0.
\end{equation*}
Thus, the degrees $\nu_j$ can be sampled independently.
In computations, each infinite degree distribution must be truncated at a
sufficiently large cutoff $N_j$ and renormalized.
The regularized kernel diagonal coefficients do not factorize, but the
unnormalized weights satisfy
\begin{equation*}
    \frac{\gamma_\nu}{\gamma_\nu+\alpha_{\lambda_r}}
    \leq
    \alpha_{\lambda_r}^{-\theta}\gamma_\nu^\theta.
\end{equation*}
Thus, the factorized $\theta$-kernel diagonal distribution provides a
proposal for rejection sampling.
The same coordinatewise cutoffs are used for the factorized proposal in the
regularized kernel diagonal sampler.
For the regularized Christoffel measure, we instead choose coordinate $j$ with
probability proportional to $j^{-6}/(j^{-6}+\lambda_r)$, sample that
coordinate from $f_j'^2\dif\mu_1$ using a rejection sampler, and draw
the remaining coordinates from $\mu_1$.

We compare the sample sizes required for active subspace recovery under the base Gaussian measure and the sampling measures described above, using
\begin{equation*}
    d=64,
    \qquad
    r\in\{3,5,7,9,11\},
    \qquad
    \theta\in\{0.8,1\}.
\end{equation*}
Both the model and the sampling measures are truncated to these $64$
coordinates.
The recovery errors below refer to this truncated model.
We take $N_j$ large enough to include every degree $n$ such that
$\gamma_{ne_j}\geq\alpha_{\lambda_{11}}$.
For each measure and rank, we estimate the recovery probability at each tested sample size using $200$ independent trials, counting a trial as successful when
\begin{equation*}
    \err(\hat\Pi_r)^2
    \leq
    1.1\err(\Pi_r)^2.
\end{equation*}
We search over powers of two to bracket the $90\%$ empirical success threshold, then refine this interval with three equally spaced interior sample sizes.
We report $\hat M$ as the smallest tested size at which the empirical success rate reaches $90\%$ and remains above this threshold at every larger tested size, requiring at least one such larger size.
If no qualifying size is found within $M_{\max}=2^{18}$, recovery is reported as unsuccessful.

Results are shown in Figure~\ref{fig:numerics-hermite-recovery}.
The figure shows that the regularized Christoffel measure achieves the best recovery,
followed by the regularized kernel diagonal measure, and then the $\theta$-kernel diagonal measures.
The base measure is the least effective, as expected, and it fails to reach the recovery criterion before $M_{\max}$ samples for $r\geq7$.
The dashed curves, which show the coherence-based rank dependence, closely match the empirical recovery sizes.

\begin{figure}[t]
    \centering
    \includegraphics[width=0.80\textwidth]{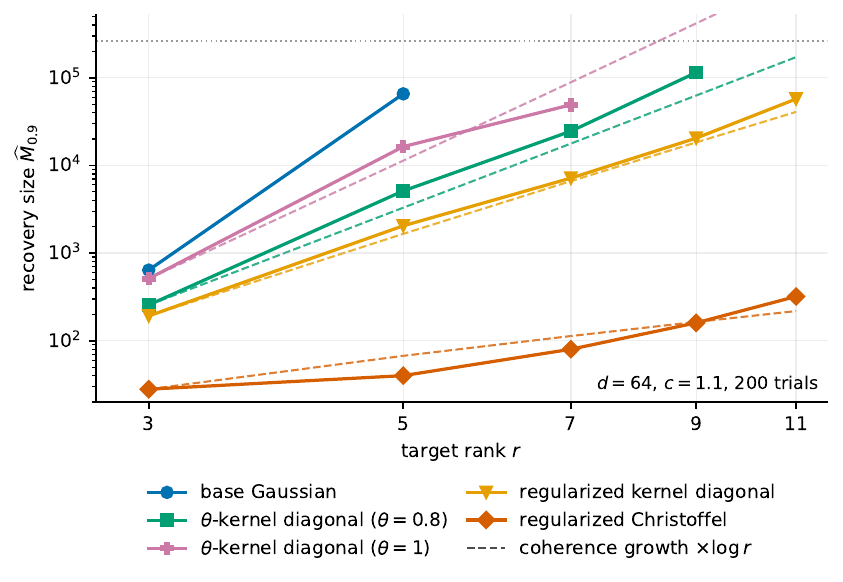}
    \caption{
    Empirical recovery size \(\hat M\) for the odd-rank Hermite experiment,
    based on \(200\) trials per tested sample size.
    The curves compare the base Gaussian, \(\theta\)-kernel diagonal measures
    with \(\theta=0.8\) and \(\theta=1\), the regularized kernel diagonal measure, and
    the regularized Christoffel measure.
    A curve ends when the \(90\%\) recovery criterion is not reached by
    \(M_{\max}=2^{18}\).
            The dashed curves show the coherence-based rank dependence
            \(d_{\mathrm{eff}}(\lambda_r)\log r\),
            \(d_{\mathrm{eff}}(\alpha_{\lambda_r},\gamma)\log r\), and
            \(\lambda_r^{-\theta}\log r\).
            Each guide is scaled to the first finite point of the corresponding curve.
        }
    \label{fig:numerics-hermite-recovery}
\end{figure}

Figure~\ref{fig:numerics-hermite-marginals} shows the densities of the absolute coordinate values,
illustrating how the sampling measures redistribute probability within individual coordinates.
They use $r=7$ and
$\theta\in\{0.8,1\}$.  The derivative $f_3'$ contains $Q_3$, whereas
$f_4'=H_8$.  The regularized Christoffel marginal uses the exact probability law
$f_j'^2\dif\mu_1$, whose density with respect to Lebesgue measure is
$f_j'(y)^2e^{-y^2/2}/\sqrt{2\pi}$.
The kernel diagonal measures also allocate probability
to modes permitted by the ambient Hermite class but absent from the gradient.
We refer to Appendix~\ref{app:numerics-hermite-joint-marginals} for additional
visualizations of the marginals.

\begin{figure}[t]
    \centering
    \includegraphics[width=\textwidth]{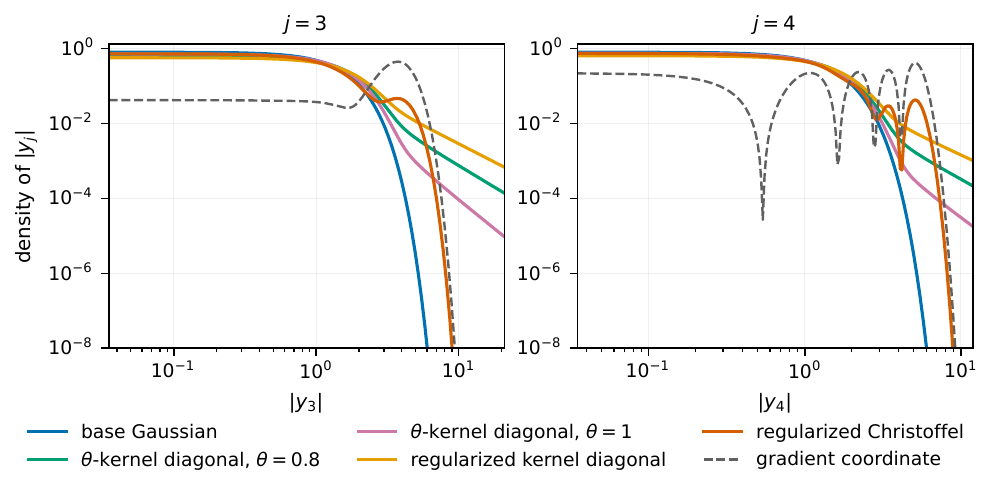}
    \caption{
    One-dimensional folded marginal densities with respect to Lebesgue
    measure on $[0,\infty)$, at $r=7$ for coordinates $j=3$ and $j=4$.
    The curves compare the base Gaussian, $\theta$-kernel diagonal
    measures with $\theta=0.8$ and $\theta=1$, the regularized kernel
    diagonal measure, and the regularized Christoffel measure.
    The gray curve is the folded density of the normalized-gradient
    probability law $f_j'(y)^2\dif\mu_1(y)$.
    }
    \label{fig:numerics-hermite-marginals}
\end{figure}
}

\subsection{A lognormal elliptic model with Haar spatial modes}
\label{subsec:numerics-lognormal-haar-modes}

We next consider a finite-dimensional instance of the lognormal model from
Section~\ref{subsec:lognormal-model} with a coefficient expansion in Haar wavelets.
More specifically, for $L=7$, let
\begin{equation*}
    \mathcal I_L
    :=
    \{(\ell,k):{1\leq\ell\leq L},\ 0\leq k<{2^{\ell-1}}\},
    \qquad
    d=\abs{\mathcal I_L}=2^L-1=127.
\end{equation*}
We order the pairs $(\ell,k)\in\mathcal I_L$ first by increasing level $\ell$
and then by increasing position $k$ within each level.
On the dyadic interval
$I_{\ell,k}=[k{2^{-(\ell-1)}},(k+1){2^{-(\ell-1)}})$, define the Haar function
\begin{equation*}
    \phi_{\ell,k}
    =
    \ind_{[k{2^{-(\ell-1)}},(k+1/2){2^{-(\ell-1)}})}
    -
    \ind_{[(k+1/2){2^{-(\ell-1)}},(k+1){2^{-(\ell-1)}})}.
\end{equation*}
Let $s>0$ denote the smoothness parameter and let $c_{\mathrm{amp}}>0$ denote
the amplitude.
For $Y\sim\mathcal N(0,I_d)$, define
\begin{equation}
    \label{eq:numerics-lognormal-haar-coefficient}
    a(x,\yb)
    =
    \exp\left(
    \sum_{(\ell,k)\in\mathcal I_L}
    y_{\ell,k}\psi_{\ell,k}(x)
    \right),
    \qquad
    \psi_{\ell,k}(x)
    =
    c_{\mathrm{amp}}{2^{-s(\ell-1)}}\phi_{\ell,k}(x).
\end{equation}
For a source location $x_{\mathrm s}\in(0,1)$, we solve
    {on $D=(0,1)$}
\begin{equation}
    \label{eq:numerics-lognormal-haar-pde}
    -\frac{\dif}{\dif x}
    \left(
    a(x,\yb)\frac{\dif u}{\dif x}(x,\yb)
    \right)
    =\delta_{x_{\mathrm s}},
    \qquad
    {x\in D},
    \qquad
    u(0,\yb)=u(1,\yb)=0.
\end{equation}
We fix
\begin{equation*}
    c_{\mathrm{amp}}=0.345077,
    \qquad
    x_{\mathrm s}=19/32,
    \qquad
    s\in\{1,1.5,2,2.5\}.
\end{equation*}
The quantity of interest is the point value $f(\yb)=u(3/8,\yb)$.
Both the point load and the point evaluation are bounded functionals on
$H_0^1(0,1)$.
We use a conservative second-order centered finite-difference scheme on a
uniform grid with $N=15359$ interior points and evaluate the coefficient at
cell midpoints.
Since the gradient covariance is not available in closed form, we use
$M_{\mathrm{ref}}=5\cdot10^5$ independent samples from the base Gaussian
measure to form the reference covariance
\begin{equation*}
    C_{\mathrm{ref}}
    :=
    \frac{1}{M_{\mathrm{ref}}}
    \sum_{i=1}^{M_{\mathrm{ref}}}
    \nabla f(Y_i)\otimes\nabla f(Y_i).
\end{equation*}
Let $\Pi_r^{\mathrm{ref}}$ be the leading rank-$r$ spectral projector of
$C_{\mathrm{ref}}$, and define the reference gradient projection error by
\begin{equation}
    \label{eq:numerics-lognormal-haar-reference-error}
    \err_{\mathrm{ref}}(\Pi)^2
    :=
    \operatorname{tr}\left((I-\Pi)C_{\mathrm{ref}}\right).
\end{equation}
We use $C_{\mathrm{ref}}$, $\Pi_r^{\mathrm{ref}}$, and
$\err_{\mathrm{ref}}$ for all reported reference active subspace errors, active directions, and recovery criteria, and note that these quantities remain subject to sampling error in $C_{\mathrm{ref}}$.
The experiments below compare the rank-$r$ reference active subspace error with the
theoretical decay estimates, estimate conditional ridge errors for coordinate
and reference active directions, and test empirical active subspace recovery
under several sampling measures.

We first record the theoretical decay scalings for the rank-$r$ reference
active subspace error.
Set
\begin{equation*}
    \epsilon_\varrho=0.1,
    \qquad
    a_s=s-\epsilon_\varrho,
\end{equation*}
and consider the corresponding infinite Haar hierarchy with weights
\begin{equation*}
    \varrho_{\ell,k}
    =
    \varrho_0 {2^{a_s(\ell-1)}}.
\end{equation*}
At each $x$, precisely one Haar function is active at each level.
For any Hermite order $\mathfrak r>2a_s$, the weight prefactor $\varrho_0$
can therefore be chosen sufficiently small that the weighted smallness
condition holds.
Moreover,
\begin{equation*}
    \sum_{{\ell\geq1}}\sum_{k=0}^{{2^{\ell-1}-1}}
    \varrho_{\ell,k}^{-2\theta}
    \asymp
    \sum_{{\ell\geq1}}{2^{(1-2\theta a_s)(\ell-1)}}
    <\infty
    \quad\Longleftrightarrow\quad
    \theta>\frac{1}{2a_s}.
\end{equation*}
Because $\mathfrak r>2a_s$, the condition $\theta>1/(2a_s)$ also implies
$\theta>1/\mathfrak r$.
Consequently, Theorem~\ref{thm:lognormal-pde-smoothness-consequences} gives
\begin{equation}
    \label{eq:numerics-lognormal-haar-tail-rate}
    \err(\Pi_r)
    \leq
    C_\theta r^{-1/\theta},
    \qquad
    \theta>\frac{1}{2a_s}.
\end{equation}
As $\theta$ decreases to $1/(2a_s)$, the exponent $-1/\theta$ approaches the
limiting slope $-2a_s$.
These limiting slopes are $-1.8,-2.8,-3.8$, and $-4.8$ for the four tested
values of $s$.
Figure~\ref{fig:numerics-lognormal-haar-population-tail} compares the normalized
reference active subspace errors
\begin{equation*}
    \frac{
        \err_{\mathrm{ref}}(\Pi_r^{\mathrm{ref}})
    }{
        \sqrt{\operatorname{tr}(C_{\mathrm{ref}})}
    }
\end{equation*}
with the limiting rates $r^{-2a_s}$.
The plateaus and drops reflect the dyadic level structure.
Away from these transitions, the envelopes of the numerical active subspace errors {qualitatively}
follow the theoretical slopes, providing numerical evidence that the bounds
capture both the observed decay and its dependence on $s$.
The corresponding reference spatial modes are shown in
Appendix~\ref{app:numerics-lognormal-haar-modes}.

\begin{figure}[t]
    \centering
    \includegraphics[width=0.80\textwidth]{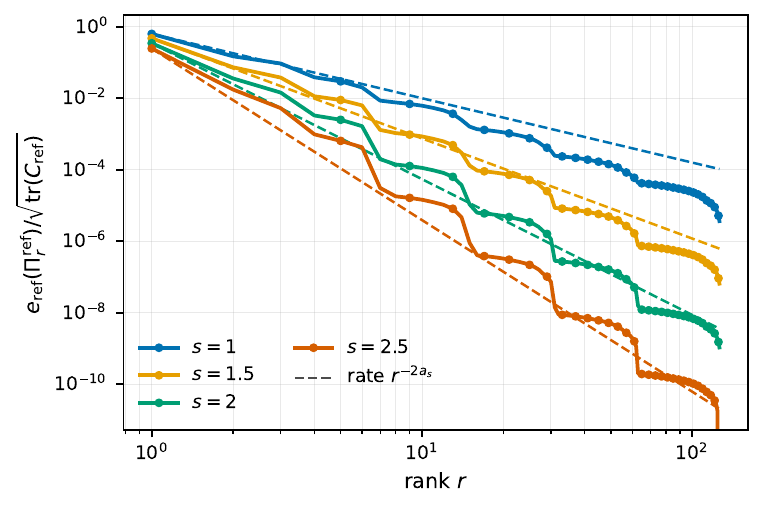}
    \caption{
        Normalized rank-$r$ reference active subspace error for the lognormal Haar
        model.
        The solid curves show
        $\err_{\mathrm{ref}}(\Pi_r^{\mathrm{ref}})/
            \sqrt{\operatorname{tr}(C_{\mathrm{ref}})}$.
        Dashed lines show the limiting scalings $r^{-2a_s}$ from
        \eqref{eq:numerics-lognormal-haar-tail-rate}, rescaled for comparison.
    }
    \label{fig:numerics-lognormal-haar-population-tail}
\end{figure}

We next estimate the error of the best ridge function associated with a fixed
subspace.
Let $V:\R^r\to\R^d$ be an isometry, and let
$V_\perp:\R^{d-r}\to\R^d$ be an isometry whose range is orthogonal to the
range of $V$.
    {Draw} mutually independent random variables
\begin{equation*}
    {X}\sim\mathcal N(0,I_r),
    \qquad
    \Xi^{(1)},\Xi^{(2)}\sim\mathcal N(0,I_{d-r}),
\end{equation*}
and set
\begin{equation*}
    {Y}^{(q)}
    =
    {VX}+V_\perp\Xi^{(q)},
    \qquad
    q\in\{1,2\}.
\end{equation*}
The two samples have the same retained coordinates
${X_V(Y^{(q)})}=V^\top {Y}^{(q)}={X}$ and conditionally independent discarded coordinates.
Therefore, the ridge error satisfies
\begin{equation}
    \label{eq:numerics-lognormal-haar-conditional-error}
    \mathcal E(V)^2
    :=
    \norm{f-\mathbb E[f\mid {\mathcal G_V}]}_{L^2_\mu}^2
    =
    \frac12\mathbb E
    \left[
        \left(f(Y^{(1)})-f(Y^{(2)})\right)^2
        \right].
\end{equation}
We approximate this expectation using $4096$ independent pairs for each
subspace.
We compare the coordinate subspace spanned by the first $r$ Haar coordinates
in the level-by-level ordering with the range of $\Pi_r^{\mathrm{ref}}$.
The coordinate estimate has the limiting scaling $r^{-a_s}$ from
\eqref{eq:lognormal-coordinate-ridge-error}, whereas
\eqref{eq:lognormal-as-ridge-error} gives the active scaling $r^{-2a_s}$ and
the exact ridge-error bound $\mathcal E(V_r)\leq\err(\Pi_r)$ in terms of the
rank-$r$ population active subspace error.
In the numerical comparison,
$\err_{\mathrm{ref}}(\Pi_r^{\mathrm{ref}})$ serves as a reference estimate
of this bound.
Figure~\ref{fig:numerics-lognormal-haar-ridge-errors} shows the estimates through
rank $20$.
    {
        Least-squares fits over the tested ranks $r\in\{8,9,10,15,20\}$ give
        coordinate slopes $-1.07$, $-1.47$, $-1.90$, and $-2.35$, and active slopes
        $-2.32$, $-3.15$, $-3.95$, and $-4.75$, for $s=1$, $1.5$, $2$, and $2.5$,
        respectively.
        The corresponding predicted coordinate slopes are $-0.9$, $-1.4$, $-1.9$,
        and $-2.4$, while the predicted active slopes are the limiting slopes
        listed above.
        The agreement is close for $s=2$ and $s=2.5$, while visible discrepancies remain for $s=1$ and $s=1.5$.
        The experiment shows that active
        ridge errors decay markedly faster than the coordinate ridge errors in all
        four cases, supporting the theoretical prediction that active subspaces can
        outperform coordinate projections.}

\begin{figure}[t]
    \centering
    \includegraphics[width=\textwidth]{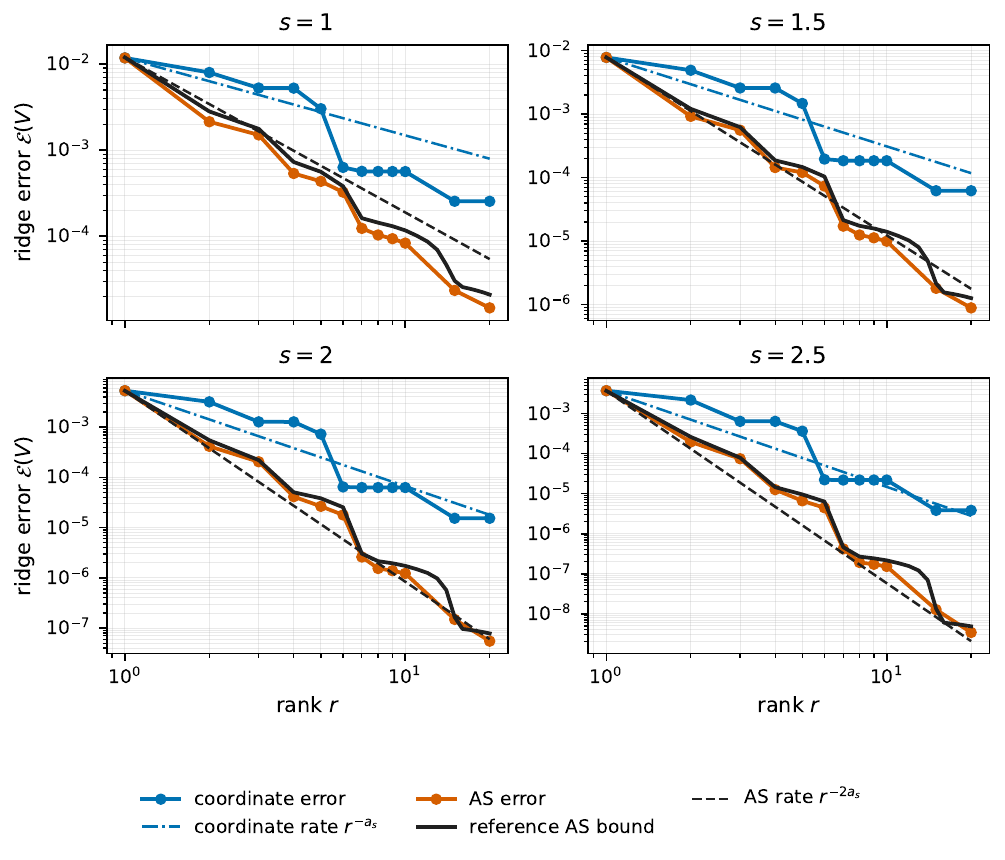}
    \caption{
        Conditional root-mean-square ridge error for coordinate and reference
        active directions in the lognormal Haar model.
        The black curve is the reference estimate of the active subspace bound,
        and the dash-dotted and dashed lines show the coordinate and active
        reference slopes, respectively.
    }
    \label{fig:numerics-lognormal-haar-ridge-errors}
\end{figure}

Finally, we test empirical active subspace recovery for $s=2.5$.
    {For this experiment, we use the $\mathfrak r=2$ weighted Hermite
        kernel introduced in Section~\ref{subsec:numerics-hermite-rare-activation}
        with $\varrho_{\ell,k}=\varrho_0 2^{a_s(\ell-1)}$, where
        $\varrho_0=0.2$ and $a_s=2.4$.}
Because precisely one Haar function is active at each level, these weights
satisfy
\begin{equation}
    \label{eq:numerics-lognormal-haar-weighted-size}
    \sup_{x\in D}
    \sum_{(\ell,k)\in\mathcal I_L}
    \varrho_{\ell,k}\abs{\psi_{\ell,k}(x)}
    =
    c_{\mathrm{amp}}\varrho_0
    \sum_{{\ell=1}}^{{L}}{2^{-\epsilon_\varrho(\ell-1)}}
    \approx
    0.3962
    <
    \frac{\log 2}{\sqrt{3}}
    \approx
    0.4002.
\end{equation}

In addition to the base measure $\mu$, we consider the
    {$\mathfrak r=2$} Hermite
$\theta$-kernel diagonal measures with
$\theta\in\{0.7,0.8,0.9,1\}$, which we sample
following the mixture construction from
Section~\ref{subsec:numerics-hermite-rare-activation}.
{Viewed as a finite truncation of the corresponding infinite
weighted Hermite space, the reciprocal-weight condition holds for
$\theta>1/(2a_s)$, whereas $\mathfrak r=2$ requires $\theta>1/2$.
Since $a_s=2.4$, every $\theta>1/2$ is admissible for the infinite-dimensional
weighted Hermite space, and hence for its finite-dimensional truncation.
Thus, $1/2$ is the infimum of the admissible exponents but is not itself
admissible.
The four tested values of $\theta$ lie in this range, and the weighted
smallness condition for the finite model has been verified above.}
As in the Hermite experiment, we estimate recovery probabilities by the fraction of successful independent trials.
Here, we use $10^3$ trials and a fixed grid of $78$ sample sizes between $1$ and $5000$, using the first $M$ samples of each trial to construct $\hat\Pi_r$.
We report $\hat M$ as the smallest grid value for which at least $90\%$ of the trials satisfy
\begin{equation*}
    \err_{\mathrm{ref}}(\hat\Pi_r)^2
    \leq
    2\err_{\mathrm{ref}}(\Pi_r^{\mathrm{ref}})^2.
\end{equation*}

As shown in Figure~\ref{fig:numerics-lognormal-haar-recovery}, sampling from the
base Gaussian measure requires substantially fewer samples than sampling from
the $\theta$-kernel diagonal measures over the tested range.
Among the latter measures, recovery improves as $\theta$ increases.
This behavior does not contradict the theory: the {regularized coherence} estimate in
\eqref{eq:lognormal-ridge-coherence-bound} is a sufficient worst-case bound
over the prescribed smoothness class and does not assert that the resulting
measures outperform the base measure for every problem instance.
For the present PDE, the base measure evidently resolves the relevant gradient
directions efficiently, whereas the modified measures also allocate samples to
behavior permitted by the smoothness class but not prominent in this instance.
Taken together, the two experiments illustrate that the effect of
kernel diagonal sampling is instance-dependent: it improves recovery in the
rare-activation example, whereas the base measure performs better for the
present PDE.

\begin{figure}[t]
    \centering
    \includegraphics[width=0.80\textwidth]{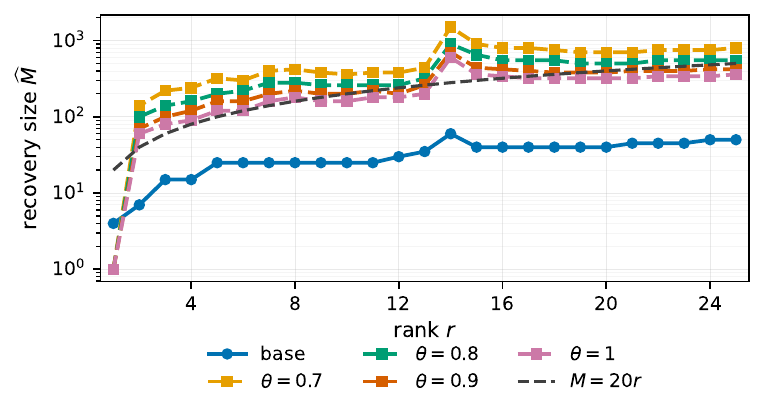}
    \caption{
        Empirical recovery size $\hat M$ for the lognormal Haar model,
        based on $10^3$ independent trials.
        The curves show sampling from the base Gaussian measure and from the
            {$\mathfrak r=2$} Hermite $\theta$-kernel diagonal measures.
        The dashed black line is the linear reference scaling $M=20r$.
    }
    \label{fig:numerics-lognormal-haar-recovery}
\end{figure}

\clearpage

\section{Conclusions}

This work has developed a finite-sample theory for empirical active subspaces
in the population projection-error metric relevant for ridge approximation.
The main point of our work is that recovering a quasi-optimal active subspace does not
require estimating the full gradient covariance accurately in operator norm.
By working directly with the projection objective and, in the positive-tail
regime, regularizing at the canonical {regularization} scale, namely the active subspace
tail energy normalized by the target rank, we obtained high-probability
quasi-optimality bounds governed by a {regularized} inverse Christoffel coherence.
In the exact finite-rank regime, the same analysis gives exact recovery at a
sufficiently small positive {regularization} scale.  Under a weighted
bounded-gradient assumption, this already improves the a priori tolerance dependence
of the standard covariance-concentration route.  Sampling from the {regularized} Christoffel
measure reduces the coherence to the covariance effective dimension and the sufficient
sample size is of order $r\log r$, up to constants and probability factors.

We then showed how a priori gradient smoothness makes this result usable without
knowing the active subspace in advance.
RKHS regularity of scalar gradient projections, together with summability of
the RKHS embedding spectrum, controls the {regularized} coherence.
The same parameter-space regularity enters the active subspace tail estimate,
where it is combined with Schatten summability of the {RKHS covariance}
across coordinate directions.
Importantly, this RKHS is only required to describe an
a priori regularity class and may contain modes that are not present in the
particular gradient field.  Its kernel diagonal nevertheless yields
constructive importance measures, and the resulting coherence and tail
estimates can be combined into prescriptions for the rank, {regularization} scale, and
number of gradient evaluations at a prescribed tolerance.
    { Weighted Hermite and Legendre regularity estimates verify these assumptions for quantities of interest of the solution to elliptic PDEs with lognormal Gaussian and affine uniform coefficients, respectively.
        In the lognormal Gaussian case, the resulting a priori ridge-error bound doubles the decay exponent obtained for ridge approximation based on the best-ranked coordinates according to the coefficients of the input field expansion. In the uniform case, the analysis also explains why sampling from the base measure can already be quasi-optimal for the associated smoothness class.}

The numerical experiments support the theoretical decay rates for the rank-$r$
active subspace error.  They also show that smoothness-adapted sampling
measures can improve empirical recovery, but may be conservative for problem
instances that are already well resolved by the base measure.

These observations suggest several directions for further work.  A natural
algorithmic goal is to refine an a priori kernel diagonal sampling measure
adaptively, using initial gradient samples to approach the unknown {regularized}
Christoffel measure.  On the analytical side, replacing worst-case coherence
by sharper instance-dependent or probabilistic quantities may better explain
typical recovery.  Finally, extending the framework to broader classes of
parametric PDEs would clarify how far the same link between smoothness,
sampling stability, and dimension reduction persists.

\clearpage
\bibliography{biblio.bib}

\begin{thebibliography}{35}

\bibitem[Adcock \emph{et~al.}(2022)Adcock, Brugiapaglia and
  Webster]{adcock2022sparse}
Adcock, B., Brugiapaglia, S. and Webster, C.~G. (2022) \emph{Sparse polynomial
  approximation of high-dimensional functions}.
\newblock Number~25 in Computational Science \& Engineering. SIAM.

\bibitem[Avron \emph{et~al.}(2019)Avron, Kapralov, Musco, Musco, Velingker and
  Zandieh]{avron2019universal}
Avron, H., Kapralov, M., Musco, C., Musco, C., Velingker, A. and Zandieh, A.
  (2019) A universal sampling method for reconstructing signals with simple
  {Fourier} transforms.
\newblock In \emph{Proceedings of the 51st Annual {ACM} {SIGACT} Symposium on
  Theory of Computing}, pp. 1051--1063.

\bibitem[Bach(2017)]{bach2017equivalence}
Bach, F. (2017) On the equivalence between kernel quadrature rules and random
  feature expansions.
\newblock \emph{Journal of Machine Learning Research} \textbf{18}(21), 1--38.

\bibitem[Bachmayr \emph{et~al.}(2017{a})Bachmayr, Cohen, {DeVore} and
  Migliorati]{bachmayr2017sparseii}
Bachmayr, M., Cohen, A., {DeVore}, R. and Migliorati, G. (2017{a}) Sparse
  polynomial approximation of parametric elliptic {PDEs}. {Part II}: Lognormal
  coefficients.
\newblock \emph{ESAIM: Mathematical Modelling and Numerical Analysis}
  \textbf{51}(1), 341--363.

\bibitem[Bachmayr \emph{et~al.}(2017{b})Bachmayr, Cohen and
  Migliorati]{bachmayr2017sparsei}
Bachmayr, M., Cohen, A. and Migliorati, G. (2017{b}) Sparse polynomial
  approximation of parametric elliptic {PDEs}. {Part I}: Affine coefficients.
\newblock \emph{ESAIM: Mathematical Modelling and Numerical Analysis}
  \textbf{51}(1), 321--339.

\bibitem[Bebendorf(2003)]{bebe2003}
Bebendorf, M. (2003) A note on the {Poincar{\'e}} inequality for convex
  domains.
\newblock \emph{Zeitschrift f{\"u}r Analysis und ihre Anwendungen} \textbf{22},
  751--756.

\bibitem[Bogachev(1998)]{bogachev1998gaussian}
Bogachev, V.~I. (1998) \emph{Gaussian measures}.
\newblock Number~62 in Mathematical Surveys and Monographs. American
  Mathematical Society.

\bibitem[Cohen and {DeVore}(2015)]{cohen2015approximation}
Cohen, A. and {DeVore}, R. (2015) Approximation of high-dimensional parametric
  {PDEs}.
\newblock \emph{Acta Numerica} \textbf{24}, 1--159.

\bibitem[Cohen and Migliorati(2023)]{cohen2023near}
Cohen, A. and Migliorati, G. (2023) Near-optimal approximation methods for
  elliptic {PDEs} with lognormal coefficients.
\newblock \emph{Mathematics of Computation} \textbf{92}(342), 1665--1691.

\bibitem[Cohen \emph{et~al.}(2017)Cohen, Musco and Musco]{cohen2017ridge}
Cohen, M.~B., Musco, C. and Musco, C. (2017) Input sparsity time low-rank
  approximation via ridge leverage score sampling.
\newblock In \emph{Proceedings of the Twenty-Eighth Annual {ACM--SIAM}
  Symposium on Discrete Algorithms}, pp. 1758--1777.

\bibitem[Constantine(2015)]{constantine2015active}
Constantine, P.~G. (2015) \emph{Active subspaces: Emerging ideas for dimension
  reduction in parameter studies}.
\newblock Number~2 in SIAM Spotlights. SIAM.

\bibitem[Constantine \emph{et~al.}(2017)Constantine, Eftekhari, Hokanson and
  Ward]{constantine2017near}
Constantine, P.~G., Eftekhari, A., Hokanson, J. and Ward, R.~A. (2017) A
  near-stationary subspace for ridge approximation.
\newblock \emph{Computer Methods in Applied Mechanics and Engineering}
  \textbf{326}, 402--421.

\bibitem[Constantine and Gleich(2014)]{constantine_gleich_2014_computing}
Constantine, P.~G. and Gleich, D.~F. (2014) Computing active subspaces with
  {Monte Carlo}.
\newblock arXiv:1408.0545.

\bibitem[Davies(1995)]{davies1995spectral}
Davies, E.~B. (1995) \emph{Spectral Theory and Differential Operators}.
  Volume~42 of \emph{Cambridge Studies in Advanced Mathematics}.
\newblock Cambridge: Cambridge University Press.

\bibitem[Diestel \emph{et~al.}(1995)Diestel, Jarchow and
  Tonge]{Diestel_Jarchow_Tonge_1995}
Diestel, J., Jarchow, H. and Tonge, A. (1995) \emph{Trace Duality}, pp.
  125--153.
\newblock Cambridge Studies in Advanced Mathematics. Cambridge: Cambridge
  University Press.

\bibitem[Hesthaven \emph{et~al.}(2016)Hesthaven, Rozza and
  Stamm]{hesthaven2016certified}
Hesthaven, J.~S., Rozza, G. and Stamm, B. (2016) \emph{Certified reduced basis
  methods for parametrized partial differential equations}.
\newblock SpringerBriefs in Mathematics. Springer.

\bibitem[Hoang and Schwab(2014)]{HoangSchwab2014}
Hoang, V.~H. and Schwab, C. (2014) {$N$-term} {Wiener} chaos approximation
  rates for elliptic {PDEs} with lognormal {Gaussian} random inputs.
\newblock \emph{Mathematical Models and Methods in Applied Sciences}
  \textbf{24}(04), 797--826.

\bibitem[Holodnak \emph{et~al.}(2018)Holodnak, Ipsen and
  Smith]{holodnak2018probabilistic}
Holodnak, J.~T., Ipsen, I. C.~F. and Smith, R.~C. (2018) A probabilistic
  subspace bound with application to active subspaces.
\newblock \emph{SIAM Journal on Matrix Analysis and Applications}
  \textbf{39}(3), 1208--1220.

\bibitem[Kundu and Wycoff(2026)]{kundu2026active}
Kundu, P. and Wycoff, N. (2026) Active subspaces in infinite dimension.
\newblock In \emph{Proceedings of the 29th International Conference on
  Artificial Intelligence and Statistics}, volume 300 of \emph{Proceedings of
  Machine Learning Research}, pp. 2089--2097. PMLR.

\bibitem[Lam \emph{et~al.}(2020)Lam, Zahm, Marzouk and
  Willcox]{lam2020multifidelity}
Lam, R.~R., Zahm, O., Marzouk, Y.~M. and Willcox, K.~E. (2020) Multifidelity
  dimension reduction via active subspaces.
\newblock \emph{SIAM Journal on Scientific Computing} \textbf{42}(2),
  A929--A956.

\bibitem[Luo \emph{et~al.}(2025)Luo, O'Leary-Roseberry, Chen and
  Ghattas]{luo2025dimension}
Luo, D., O'Leary-Roseberry, T., Chen, P. and Ghattas, O. (2025) Dimension
  reduction for derivative-informed operator learning: An analysis of
  approximation errors.
\newblock arXiv:2504.08730.

\bibitem[Milbradt and Wahl(2020)]{milbradt2020high}
Milbradt, C. and Wahl, M. (2020) High-probability bounds for the reconstruction
  error of {PCA}.
\newblock \emph{Statistics \& Probability Letters} \textbf{161}, 108741.

\bibitem[Minsker(2017)]{minsker2017bernstein}
Minsker, S. (2017) On some extensions of {Bernstein}'s inequality for
  self-adjoint operators.
\newblock \emph{Statistics \& Probability Letters} \textbf{127}, 111--119.

\bibitem[Musco and Musco(2020)]{musco2020pcp}
Musco, C. and Musco, C. (2020) Projection-cost-preserving sketches: Proof
  strategies and constructions.
\newblock arXiv:2004.08434.

\bibitem[Musco and Woodruff(2017)]{musco2017kernel}
Musco, C. and Woodruff, D.~P. (2017) Is input sparsity time possible for kernel
  low-rank approximation?
\newblock In \emph{Advances in Neural Information Processing Systems},
  volume~30, pp. 4435--4445.

\bibitem[Pauwels \emph{et~al.}(2018)Pauwels, Bach and
  Vert]{pauwels2018relating}
Pauwels, E., Bach, F. and Vert, J.-P. (2018) Relating leverage scores and
  density using regularized {Christoffel} functions.
\newblock In \emph{Advances in Neural Information Processing Systems},
  volume~31, pp. 1663--1672.

\bibitem[Rei{\ss} and Wahl(2020)]{reiss2020nonasymptotic}
Rei{\ss}, M. and Wahl, M. (2020) Nonasymptotic upper bounds for the
  reconstruction error of {PCA}.
\newblock \emph{The Annals of Statistics} \textbf{48}(2), 1098--1123.

\bibitem[Rudi \emph{et~al.}(2013)Rudi, Ca{\~n}as and Rosasco]{rudi2013sample}
Rudi, A., Ca{\~n}as, G.~D. and Rosasco, L. (2013) On the sample complexity of
  subspace learning.
\newblock In \emph{Advances in Neural Information Processing Systems},
  volume~26, pp. 2067--2075.

\bibitem[Simon(2005)]{simon2005trace}
Simon, B. (2005) \emph{Trace ideals and their applications}. Second edition.
\newblock Number 120 in Mathematical Surveys and Monographs. American
  Mathematical Society.

\bibitem[Stein(1970)]{stein1970singular}
Stein, E.~M. (1970) \emph{Singular Integrals and Differentiability Properties
  of Functions}.
\newblock Number~30 in Princeton Mathematical Series. Princeton, NJ: Princeton
  University Press.

\bibitem[T{\"o}lle(2012)]{tolle2012uniqueness}
T{\"o}lle, J.~M. (2012) Uniqueness of weighted {Sobolev} spaces with weakly
  differentiable weights.
\newblock \emph{Journal of Functional Analysis} \textbf{263}(10), 3195--3223.

\bibitem[Tsilifis and Ghanem(2017)]{tsilifis2017reduced}
Tsilifis, P. and Ghanem, R.~G. (2017) Reduced {Wiener} chaos representation of
  random fields via basis adaptation and projection.
\newblock \emph{Journal of Computational Physics} \textbf{341}, 102--120.

\bibitem[Tsilifis \emph{et~al.}(2019)Tsilifis, Huan, Safta, Sargsyan, Lacaze,
  Oefelein, Najm and Ghanem]{tsilifis2019compressive}
Tsilifis, P., Huan, X., Safta, C., Sargsyan, K., Lacaze, G., Oefelein, J.~C.,
  Najm, H.~N. and Ghanem, R.~G. (2019) Compressive sensing adaptation for
  polynomial chaos expansions.
\newblock \emph{Journal of Computational Physics} \textbf{380}, 29--47.

\bibitem[Williams(1991)]{williams_1991}
Williams, D. (1991) \emph{Probability with Martingales}.
\newblock Cambridge: Cambridge University Press.

\bibitem[Zahm \emph{et~al.}(2020)Zahm, Constantine, Prieur and
  Marzouk]{zahm2020gradient}
Zahm, O., Constantine, P.~G., Prieur, C. and Marzouk, Y.~M. (2020)
  Gradient-based dimension reduction of multivariate vector-valued functions.
\newblock \emph{SIAM Journal on Scientific Computing} \textbf{42}(1),
  A534--A558.

\end{thebibliography}
\bibliographystyle{CUP}    %

\appendix
\clearpage
\section{Auxiliary results for the active subspace setting}
\label{app:as-method}

This appendix proves the analytic facts used in Section~\ref{sec:slas}.
Throughout, $\mathcal C^\infty_{\mathrm{cyl}}$ denotes the space of smooth cylindrical functions introduced in Section~\ref{sec:slas}, and $\nabla_0$ denotes the cylindrical gradient defined in \eqref{eq:cylindrical-gradient}.
We write $\gamma$ for the standard Gaussian probability measure on $\R$ and
$\lambda=\frac12\mathcal L^1\vert_{[-1,1]}$ for the centered uniform probability measure on $[-1,1]$.

\subsection{Product Sobolev calculus}
\label{app:product-sobolev-calculus}

For both $\mu=\gamma^{\otimes d}$ and $\mu=\lambda^{\otimes d}$, the space
$\mathcal C^\infty_{\mathrm{cyl}}$ is dense in $L^2_\mu(\Gamma)$.
Indeed, let
\begin{equation}
    d_n:=\min\{n,d\},
    \qquad
    \mathcal F_n:=\sigma({\pi_1},\ldots,{\pi_{d_n}}),
\end{equation}
with the convention $\min\{n,+\infty\}=n$ {and where we recall that $\pi_m$ denotes the $m$th coordinate map}.
Conditional expectations onto $\mathcal F_n$ converge in $L^2_\mu$ to the original function, while {functions in $C_b^\infty(\R^{d_n})$} are dense in the corresponding finite-dimensional $L^2$ spaces.
When $d<+\infty$, the filtration is constant after $n=d$.

\begin{proposition}[Closability for Gaussian and uniform product measures]
    \label{prop:closability-product-measures}
    Let $d\in\N\cup\{+\infty\}$ and let either
    $\mu=\gamma^{\otimes d}$ or $\mu=\lambda^{\otimes d}$.
    Then
    \begin{equation}
        \nabla_0:
        \mathcal C^\infty_{\mathrm{cyl}}
        \subset L^2_\mu(\Gamma)
        \longrightarrow
        L^2_\mu(\Gamma;\ell^2(d))
    \end{equation}
    is closable.
\end{proposition}

\begin{proof}
    We first verify that the cylindrical gradient is well defined on $L^2_\mu$ equivalence classes.
    Suppose that $F\in\mathcal C^\infty_{\mathrm{cyl}}$ vanishes $\mu$-almost everywhere.
    In the Gaussian case, integration by parts gives
    \begin{equation}
        \int_\Gamma
        \partial_kF\,\Phi
        \dif\mu
        =
        \int_\Gamma
        F\left({\pi_k}\Phi-\partial_k\Phi\right)
        \dif\mu
        =0
    \end{equation}
    for every $\Phi\in\mathcal C^\infty_{\mathrm{cyl}}$.
    Density of the cylindrical functions implies $\partial_kF=0$ in $L^2_\mu$.
    In the uniform case, {take $\Phi\in\mathcal C^\infty_{\mathrm{cyl}}$ whose dependence on} the {$k$th coordinate is compactly supported in $(-1,1)$.
    Integration by parts in the $k$th coordinate gives}
    \begin{equation}
        \int_\Gamma
        \partial_kF\,\Phi
        \dif\mu
        =
        -
        \int_\Gamma
        F\partial_k\Phi
        \dif\mu
        =0
    \end{equation}
    {Since} these {test functions} are dense in $L^2_\mu{(\Gamma)}${, it follows that} ${\partial_kF=0}$.
    Hence two cylindrical representatives that agree $\mu$-almost everywhere have the same cylindrical gradient $\mu$-almost everywhere{.

    By the graph criterion for closability, it suffices to show that the following convergences imply $G=0$; see, for example, \cite[Section~1.1]{davies1995spectral}}.
    Let $F_n\in\mathcal C^\infty_{\mathrm{cyl}}$ now satisfy
    \begin{equation}
        F_n\to0
        \quad\text{in }L^2_\mu(\Gamma),
        \qquad
        \nabla_0F_n\to G
        \quad\text{in }L^2_\mu(\Gamma;\ell^2(d)).
    \end{equation}
    We prove that every coordinate $G_k$ vanishes.

    Suppose first that $\mu=\gamma^{\otimes d}$.
    For $\Phi\in\mathcal C^\infty_{\mathrm{cyl}}$, one-dimensional Gaussian integration by parts in the $k$th coordinate gives
    \begin{equation}
        \int_\Gamma
        \partial_kF_n\,\Phi
        \dif\mu
        =
        \int_\Gamma
        F_n\left({\pi_k}\Phi-\partial_k\Phi\right)
        \dif\mu.
    \end{equation}
    The factor ${\pi_k}\Phi-\partial_k\Phi$ belongs to $L^2_\mu(\Gamma)$.
    Passing to the limit yields
    \begin{equation}
        \int_\Gamma G_k\Phi\dif\mu=0.
    \end{equation}
    Density of $\mathcal C^\infty_{\mathrm{cyl}}$ in $L^2_\mu(\Gamma)$ implies $G_k=0$.

    Suppose next that $\mu=\lambda^{\otimes d}$.
    {For} ${\Phi}\in{\mathcal C}^\infty{_{\mathrm{cyl}}}$ {whose dependence on} the $k$th coordinate {is compactly supported in $(-1,1)$, integration by parts in that coordinate} gives
    \begin{equation}
        \int_\Gamma
        \partial_kF_n\,\Phi
        \dif\mu
        =
        -
        \int_\Gamma
        F_n\partial_k\Phi
        \dif\mu.
    \end{equation}
    Passing to the limit shows that
    \begin{equation}
        \int_\Gamma G_k\Phi\dif\mu=0.
    \end{equation}
    {Since these test functions} are dense in $L^2_\mu(\Gamma)$, {it follows that} $G_k=0$.

    Since $G_k=0$ for every coordinate,
    \begin{equation}
        \norm{G}_{L^2_\mu(\Gamma;\ell^2(d))}^2
        =
        \sum_{k=1}^d
        \norm{G_k}_{L^2_\mu(\Gamma)}^2
        =0.
    \end{equation}
    {Hence $G=0$, proving that $\nabla_0$} is {closable}.
\end{proof}

{In Section~\ref{sec:slas}, we stated} the {coordinate characterization \eqref{eq:coordinate-weak-derivative-characterization}.
We now prove this equivalence.
The weak derivatives used hereafter are those defined by \eqref{eq:gaussian-weak-derivative} and \eqref{eq:uniform-weak-derivative}.
The forward implication follows by passing smooth cylindrical approximations to} the {limit in the defining integration-by-parts identities}.
For the {converse}, {we first approximate a function by functions depending on finitely many coordinates and must then approximate each finite-dimensional function, together with its} weak {gradient, by a smooth} function.
{The following lemma provides this} finite-dimensional approximation.

\begin{lemma}[Finite-dimensional graph density]
    \label{lem:finite-dimensional-graph-density}
    Let $m\in\N$ and let either $\nu=\gamma^{\otimes m}$ on $\R^m$ or
    $\nu=\lambda^{\otimes m}$ on $[-1,1]^m$.
    Suppose that $u\in L^2_\nu$ has weak derivatives
    $\partial_ku\in L^2_\nu$, $k=1,\ldots,m$, in the sense of
    \eqref{eq:gaussian-weak-derivative} or \eqref{eq:uniform-weak-derivative}.
    Then there are $u_j\in C_b^\infty(\R^m)$ such that
    \begin{equation}
        u_j\to u
        \quad\text{in }L^2_\nu,
        \qquad
        \nabla u_j\to
        (\partial_1u,\ldots,\partial_mu)
        \quad\text{in }L^2_\nu(\R^m;\R^m).
    \end{equation}
\end{lemma}

\begin{proof}
    Consider first $\nu=\gamma^{\otimes m}$.
    {We first identify the Gaussian weak derivatives with the usual distributional derivatives locally.
    We then cut off at infinity and mollify.
    The} density {of $\nu$} with respect to Lebesgue measure is
    \begin{equation}
        q_m(x)
        :=
        (2\pi)^{-m/2}\exp\left(-\frac12\abs{x}^2\right).
    \end{equation}
    Let $g_k=\partial_ku$ in the Gaussian weak sense.
    Given $\zeta\in C_c^\infty(\R^m)$, use
    $\Phi=\zeta/q_m$ in \eqref{eq:gaussian-weak-derivative}.
    Since $\partial_kq_m=-x_kq_m$, this gives
    \begin{equation}
        \int_{\R^m}g_k\zeta\dif x
        =
        -
        \int_{\R^m}u\partial_k\zeta\dif x.
    \end{equation}
    Thus $u\in W^{1,2}_{\mathrm{loc}}(\R^m)$ with distributional gradient $g=(g_1,\ldots,g_m)$.

    Choose $\chi_R\in C_c^\infty(\R^m)$ such that
    $0\leq\chi_R\leq1$, $\chi_R=1$ on the ball of radius $R$,
    $\chi_R=0$ outside the ball of radius $R+1$, and
    $\norm{\nabla\chi_R}_{L^\infty}\leq C$ independently of $R$.
    Then $\chi_Ru\in W^{1,2}(\R^m)$ has compact support and
    \begin{equation}
        \nabla(\chi_Ru)
        =
        \chi_Rg+u\nabla\chi_R.
    \end{equation}
    The Gaussian $L^2$ norm of $(1-\chi_R)u$ tends to zero, and
    \begin{equation}
        \norm{(1-\chi_R)g}_{L^2_\nu}
        +
        \norm{u\nabla\chi_R}_{L^2_\nu}
        \longrightarrow0
        \qquad
        \text{as }R\to+\infty.
    \end{equation}
    Hence $\chi_Ru\to u$ in {$L^2_\nu$ and
    $\nabla(\chi_Ru)\to g$ in $L^2_\nu(\R^m;\R^m)$}.
    For fixed $R$, standard convolution approximates $\chi_Ru$ in the unweighted $W^{1,2}(\R^m)$ norm by functions in $C_c^\infty(\R^m)$.
    On the common bounded support, the Gaussian density is bounded above and below by positive constants, so {unweighted $W^{1,2}$} convergence {implies $L^2_\nu$ convergence of both the functions and their gradients.
    Choosing $R\to+\infty$ and, for each $R$, a sufficiently accurate smooth approximation gives the sequence asserted} in the {lemma}.

    Consider now $\nu=\lambda^{\otimes m}$ and set $Q=(-1,1)^m$.
    {In this case, the weak derivatives are} the usual distributional {derivatives} on $Q$, so $u\in W^{1,2}(Q)${.
    We therefore extend $u$ from $Q$ to $\R^m$ and mollify the extension}.
    Since $Q$ is a bounded Lipschitz domain, there is a bounded linear extension operator
    \begin{equation}
        E:W^{1,2}(Q)\longrightarrow W^{1,2}(\R^m),
        \qquad
        Eu=u
        \quad\text{almost everywhere on }Q,
    \end{equation}
    by \cite[Chapter~VI, Theorem~5, p.~181]{stein1970singular}.
    Set $\widetilde u=Eu$, and let $\rho_\varepsilon$ be a standard mollifier.
    Then
    \begin{equation}
        u_\varepsilon
        :=
        \rho_\varepsilon*\widetilde u
        \in C_b^\infty(\R^m),
    \end{equation}
    and $u_\varepsilon\to\widetilde u$ in $W^{1,2}(\R^m)$ as $\varepsilon\to0$.
    Restriction to $Q$, together with $\lambda^{\otimes m}=2^{-m}\mathcal L^m\vert_Q$, gives
    {$u_\varepsilon\to u$} in {$L^2_\nu$ and
    $\nabla u_\varepsilon\to(\partial_1u,\ldots,\partial_mu)$ in $L^2_\nu(\R^m;\R^m)$}.
\end{proof}

\begin{proposition}[Coordinate characterization of $H^1_\mu$]
    \label{prop:coordinate-characterization-H1}
    Let $d\in\N\cup\{+\infty\}$ and let either
    $\mu=\gamma^{\otimes d}$ or $\mu=\lambda^{\otimes d}$.
    Then $h\in H^1_\mu(\Gamma)$ if and only if
    $h\in L^2_\mu(\Gamma)$ has weak coordinate derivatives
    $\partial_kh\in L^2_\mu(\Gamma)$ satisfying
    \begin{equation}
        \sum_{k=1}^d
        \norm{\partial_kh}_{L^2_\mu(\Gamma)}^2
        <+\infty.
    \end{equation}
    In this case,
    \begin{equation}
        \nabla h
        =
        (\partial_kh)_{k=1}^d
        \in L^2_\mu(\Gamma;\ell^2(d)).
    \end{equation}
\end{proposition}

\begin{proof}
    Suppose first that $h\in H^1_\mu(\Gamma)$.
    There are $F_n\in\mathcal C^\infty_{\mathrm{cyl}}$ such that
    \begin{equation}
        F_n\to h
        \quad\text{in }L^2_\mu(\Gamma),
        \qquad
        \nabla_0F_n\to\nabla h
        \quad\text{in }L^2_\mu(\Gamma;\ell^2(d)).
    \end{equation}
    For each fixed $k$, passing to the limit in the corresponding integration-by-parts identity shows that $(\nabla h)_k$ is the $k$th weak derivative of $h$.
    Square summability follows from $\nabla h\in L^2_\mu(\Gamma;\ell^2(d))$.

    Conversely, suppose that $h\in L^2_\mu(\Gamma)$ has weak derivatives
    $g_k:=\partial_kh$ and that
    \begin{equation}
        g:=(g_k)_{k=1}^d
        \in L^2_\mu(\Gamma;\ell^2(d)).
    \end{equation}
    Let $d_n=\min\{n,d\}$,
    let $\mathcal F_n=\sigma({\pi_1},\ldots,{\pi_{d_n}})$, and let
    $P_n$ denote conditional expectation onto $\mathcal F_n$.
    Then $P_nh\to h$ in $L^2_\mu(\Gamma)$.

    For $k\leq d_n$, the weak derivative of $P_nh$ in the $k$th coordinate is $P_ng_k$.
    To verify this in the Gaussian case, take an $\mathcal F_n$-measurable smooth test function $\Phi$ and compute
    \begin{align}
        \int_\Gamma P_ng_k\,\Phi\dif\mu
         & =
        \int_\Gamma g_k\Phi\dif\mu
        \nonumber \\
         & =
        \int_\Gamma h\left({\pi_k}\Phi-\partial_k\Phi\right)\dif\mu
        \nonumber \\
         & =
        \int_\Gamma P_nh\left({\pi_k}\Phi-\partial_k\Phi\right)\dif\mu.
    \end{align}
    The {calculation in the} uniform {case} is identical with the right-hand side of \eqref{eq:uniform-weak-derivative}.
    Therefore, by Lemma~\ref{lem:finite-dimensional-graph-density}, there is
    $F_n\in\mathcal C^\infty_{\mathrm{cyl}}$, depending only on the first $d_n$ coordinates, such that
    \begin{equation}
        \norm{F_n-P_nh}_{L^2_\mu(\Gamma)}
        +
        \norm{\nabla_0F_n-G^{(n)}}_{L^2_\mu(\Gamma;\ell^2(d))}
        \leq
        \frac1n,
    \end{equation}
    where
    \begin{equation}
        G^{(n)}
        :=
        (P_ng_1,\ldots,P_ng_{d_n},0,0,\ldots).
    \end{equation}

    Let $Q_n$ be the orthogonal projector in $\ell^2(d)$ onto
    $\operatorname{span}\{e_1,\ldots,e_{d_n}\}$.
    Since conditional expectation of the separable Hilbert-valued random variable $g$ is taken coordinatewise,
    \begin{equation}
        G^{(n)}
        =
        Q_n\mathbb E\left[g\,\middle\vert\,\mathcal F_n\right].
    \end{equation}
    The two {terms} in
    \begin{equation}
        G^{(n)}-g
        =
        Q_n\left(
        \mathbb E\left[g\,\middle\vert\,\mathcal F_n\right]-g
        \right)
        -
        (I-Q_n)g
    \end{equation}
    {lie pointwise in $\operatorname{Ran}(Q_n)$ and $\operatorname{Ran}(I-Q_n)$, respectively, and} are {therefore} orthogonal in $\ell^2(d)$.
    Hence
    \begin{align}
        \norm{G^{(n)}-g}_{L^2_\mu(\Gamma;\ell^2(d))}^2
         & \leq
        \norm{
            \mathbb E\left[g\,\middle\vert\,\mathcal F_n\right]-g
        }_{L^2_\mu(\Gamma;\ell^2(d))}^2
        \nonumber \\
         & \quad+
        \norm{(I-Q_n)g}_{L^2_\mu(\Gamma;\ell^2(d))}^2.
    \end{align}
    The first term tends to zero by the Hilbert-valued martingale convergence theorem, and the second tends to zero by square summability of the coordinates of $g$.
    In finite dimension, both statements are immediate once $n\geq d$.
    It follows that $F_n\to h$ in $L^2_\mu(\Gamma)$ and
    $\nabla_0F_n\to g$ in $L^2_\mu(\Gamma;\ell^2(d))$.
    Therefore $h\in H^1_\mu(\Gamma)$ and $\nabla h=g$.
\end{proof}

\subsection{Linear observations}
\label{app:linear-observables-H1}

{The following lemma applies under the standing closability assumption of Section~\ref{sec:slas} and the observation bound \eqref{eq:linear-observable-bound}.}
\begin{lemma}[Linear observations are Sobolev functions]
    \label{lem:linear-observables-H1}
    {Let} ${\mu}$ {be any Borel probability measure on $\Gamma$ satisfying these two assumptions}.
    Then, for every $v\in\ell^2(d)$, the observation ${X_v}$ belongs to $H^1_\mu(\Gamma)$ and
    \begin{equation}
        \nabla{X_v}=v.
    \end{equation}
\end{lemma}

\begin{proof}
    Let $a\in c_{00}(d)$.
    By \eqref{eq:linear-observable-bound}, ${X_a}\in L^2_\mu(\Gamma)$.
    Choose $\eta_R\in C_c^\infty(\R)$ such that
    $0\leq\eta_R\leq1$ {and} $\eta_R=1$ on $[-R,R]$.
    Define
    \begin{equation}
        \theta_R(t)
        :=
        \int_0^t\eta_R(s)\dif s,
        \qquad
        F_R
        :=
        \theta_R{(X_a)}.
    \end{equation}
    Then $F_R\in\mathcal C^\infty_{\mathrm{cyl}}$,
    $\abs{F_R}\leq\abs{{X_a}}$, and
    \begin{equation}
        \nabla_0F_R
        =
        \eta_R({X_a})a.
    \end{equation}
    {Because $\eta_R=1$ on $[-R,R]$, both
    $F_R\to X_a$ and $\nabla_0F_R\to a$ pointwise.
    Moreover,}
    \begin{equation}
        {\abs{F_R-X_a}^2
        \leq
        4\abs{X_a}^2,}
    \end{equation}
    {so dominated} convergence gives $F_R\to {X_a}$ in $L^2_\mu(\Gamma)${.
    For the gradients},
    \begin{align}
        {\norm{\nabla_0F_R-a}}_{{L^2_\mu(\Gamma;\ell^2(d))}^2}
        & {=
        \norm{a}_{\ell^2(d)}^2
        \norm{\eta_R(X_a)-1}_{L^2_\mu(\Gamma)}^2
        \longrightarrow 0,}
    \end{align}
    {because $\eta_R(X_a)-1$ converges pointwise to zero and is bounded in absolute value by $1$}.
    Thus ${X_a}\in H^1_\mu(\Gamma)$ and
    $\nabla{X_a}=a$.

    For $v\in\ell^2(d)$, let
    \begin{equation}
        v^{(n)}
        :=
        \sum_{m=1}^{\min\{n,d\}}v_me_m.
    \end{equation}
    Then ${X_{v^{(n)}}}\to {X_v}$ in $L^2_\mu(\Gamma)$ and
    \begin{equation}
        \nabla{X_{v^{(n)}}}
        =
        v^{(n)}
        \longrightarrow
        v
        \quad\text{in }\ell^2(d).
    \end{equation}
    Closedness of $\nabla$ proves the claim.
\end{proof}

\subsection{Gaussian subspace conditional Poincar\'e inequality}
\label{app:gaussian-conditional-poincare}

We prove Proposition~\ref{prop:gaussian-conditional-poincare} directly on the Gaussian Sobolev space $H^1_\mu(\Gamma)$.
The argument applies to both finite and infinite Gaussian {product measures}.

\begin{proof}[Proof of Proposition~\ref{prop:gaussian-conditional-poincare}]
    {By \eqref{eq:linear-observation-functional}, for every} $a\in\ell^2(d)$, {the observation} ${X_a}$ is the $L^2_\mu$ limit of finite linear combinations of the coordinate variables,
    which are independent and standard Gaussian.
    For $a_1,\ldots,a_m\in\ell^2(d)$, simultaneous coordinate truncation therefore gives centered Gaussian vectors converging in $L^2_\mu(\Gamma;\R^m)$ to
    \begin{equation}
        ({X_{a_1}},\ldots,{X_{a_m}}).
    \end{equation}
    Their covariance matrices converge to the Gram matrix
    $(\inp{a_i}{a_j}_{\ell^2(d)})_{i,j=1}^m$.
    Convergence of the corresponding characteristic functions shows that the limiting vector is centered Gaussian{.
    In particular},
    \begin{equation}
        \mathbb E\left[{X_aX_b}\right]
        =
        \inp{a}{b}_{\ell^2(d)}.
    \end{equation}
    Thus ${\{X_a:a\in\ell^2(d)\}}$ is a centered Gaussian process whose covariance is the $\ell^2(d)$ inner product.

    Let $\Pi=VV^\adj$, where $V:\R^r\to\ell^2(d)$ is an isometry with columns $v_1,\ldots,v_r$, and recall that
    \begin{equation}
        {X_V}
        =
        ({X_{v_1}},\ldots,{X_{v_r}}).
    \end{equation}

    We first prove the inequality for $h\in\mathcal C^\infty_{\mathrm{cyl}}$.
    {Choose distinct coordinate indices $i_1,\ldots,i_m$ and $\phi\in C_b^\infty(\R^m)$ representing $h$, and let $E:\R^m\to\ell^2(d)$ be the isometry with columns $e_{i_1},\ldots,e_{i_m}$.}
    Then
    \begin{equation}
        h
        =
        \phi({X_E}),
        \qquad
        {X_E}
        =
        ({X_{e_{i_1}}},\ldots,{X_{e_{i_m}}}).
    \end{equation}
    Define $B\in\R^{m\times r}$ by ${B:=E^\adj V}$, that is,
    \begin{equation}
        B_{kj}
        :=
        \inp{e_{i_k}}{v_j}_{\ell^2(d)}.
    \end{equation}
    The pair $({X_E,X_V})$ is jointly centered Gaussian, with
    \begin{equation}
        \mathbb E[{X_E X_E^\adj}]=I_m,
        \qquad
        \mathbb E[{X_V X_V^\adj}]=I_r,
        \qquad
        \mathbb E[{X_E X_V^\adj}]=B.
    \end{equation}
    Set
    \begin{equation}
        R:={X_E-BX_V}.
    \end{equation}
    Then
    \begin{equation}
        \mathbb E[{R X_V^\adj}]=0,
        \qquad
        \mathbb E[RR^\adj]=I_m-BB^\adj=:\Sigma.
    \end{equation}
    The vectors $R$ and ${X_V}$ are jointly Gaussian and uncorrelated, hence independent.
    Therefore, conditionally on ${X_V=z}$, the vector ${X_E}$ has the Gaussian law
    \begin{equation}
        \mathcal N(Bz,\Sigma).
    \end{equation}

    Let $A=\Sigma^{1/2}$ and let $\xi\sim\mathcal N(0,I_m)$.
    {Applying the} finite-dimensional Gaussian Poincar\'e inequality
    {\cite[Theorem~5.5.11]{bogachev1998gaussian}} to
    $\xi\mapsto\phi(Bz+A\xi)$ gives
    \begin{align}
        \operatorname{Var}\left(
        \phi({X_E})\,\middle\vert\,{X_V}=z
        \right)
         & \leq
        \mathbb E\left[
            \nabla\phi({X_E})^\adj
            \Sigma
            \nabla\phi({X_E})
            \,\middle\vert\,
            {X_V}=z
            \right].
    \end{align}
    This argument also covers singular $\Sigma$.
    Since $\mathcal G_V$ is the $\mu$-completion of $\sigma({X_V})$, the conditional-variance identity gives
    \begin{equation}
        \norm{
            h-\mathbb E\left[h\,\middle\vert\,\mathcal G_V\right]
        }_{L^2_\mu(\Gamma)}^2
        =
        \mathbb E\left[
            \operatorname{Var}\left(h\,\middle\vert\,{X_V}\right)
            \right].
    \end{equation}
    Integrating the preceding pointwise inequality with respect to the law of ${X_V}$ therefore yields
    \begin{align}
        \norm{
            h-\mathbb E\left[h\,\middle\vert\,\mathcal G_V\right]
        }_{L^2_\mu(\Gamma)}^2
         & \leq
        \int_\Gamma
        \nabla\phi({X_E(\yb)})^\adj
        (I_m-BB^\adj)
        \nabla\phi({X_E(\yb)})
        \dif\mu(\yb).
        \label{eq:gaussian-conditional-core-bound}
    \end{align}
    On the other hand,
    \begin{equation}
        \nabla_0h
        =
        \sum_{k=1}^m
        \partial_k\phi({X_E})e_{i_k},
    \end{equation}
    and
    \begin{equation}
        \left(V^\adj\nabla_0h\right)_j
        =
        \sum_{k=1}^m
        B_{kj}\partial_k\phi({X_E}).
    \end{equation}
    Since $\Pi=VV^\adj$ is an orthogonal projector,
    \begin{align}
        \norm{(I-\Pi)\nabla_0h}_{\ell^2(d)}^2
         & =
        \norm{\nabla_0h}_{\ell^2(d)}^2
        -
        \norm{V^\adj\nabla_0h}_{\R^r}^2
        \nonumber \\
         & =
        \nabla\phi({X_E})^\adj
        (I_m-BB^\adj)
        \nabla\phi({X_E}).
    \end{align}
    Combining this identity with \eqref{eq:gaussian-conditional-core-bound} proves
    \begin{equation}
        \norm{
            h-\mathbb E\left[h\,\middle\vert\,\mathcal G_V\right]
        }_{L^2_\mu(\Gamma)}^2
        \leq
        \norm{(I-\Pi)\nabla_0h}_{L^2_\mu(\Gamma;\ell^2(d))}^2
    \end{equation}
    for every $h\in\mathcal C^\infty_{\mathrm{cyl}}$.

    Let now $h\in H^1_\mu(\Gamma)$.
    By definition of the closed gradient, there are
    $h_n\in\mathcal C^\infty_{\mathrm{cyl}}$ such that
    \begin{equation}
        h_n\to h
        \quad\text{in }L^2_\mu(\Gamma),
        \qquad
        \nabla_0h_n\to\nabla h
        \quad\text{in }L^2_\mu(\Gamma;\ell^2(d)).
    \end{equation}
    Conditional expectation is an $L^2$ contraction, so
    \begin{equation}
        h_n-\mathbb E\left[h_n\,\middle\vert\,\mathcal G_V\right]
        \longrightarrow
        h-\mathbb E\left[h\,\middle\vert\,\mathcal G_V\right]
        \quad\text{in }L^2_\mu(\Gamma).
    \end{equation}
    Moreover,
    \begin{equation}
        (I-\Pi)\nabla_0h_n
        \longrightarrow
        (I-\Pi)\nabla h
        \quad\text{in }L^2_\mu(\Gamma;\ell^2(d)).
    \end{equation}
    Passing to the limit proves \eqref{eq:conditional-poincare} with $C_P(\Pi)=1$.
\end{proof}

\section{Additional numerical results}
\label{app:numerical-results}

\subsection{Joint marginals for the rare-activation Hermite model}
\label{app:numerics-hermite-joint-marginals}

Figure~\ref{fig:numerics-hermite-joint-marginals} shows the $(y_5,y_7)$
marginals at $r=7$ for the model introduced in
Section~\ref{subsec:numerics-hermite-rare-activation}.
The base and $\theta$-kernel diagonal measures factorize across coordinates,
whereas the {regularized} kernel diagonal measure generally does not.
For the {regularized} kernel diagonal measure, the dependence is induced by its
nonfactorized Hermite coefficients.
The {regularized} Christoffel measure is a mixture that selects one normalized
gradient coordinate at a time, while the remaining coordinates retain their
Gaussian law; this produces its horizontal and vertical features.
The joint marginals therefore exhibit the different ways in which the
sampling measures redistribute probability across coordinates.

\begin{figure}[H]
    \centering
    \includegraphics[width=\textwidth]{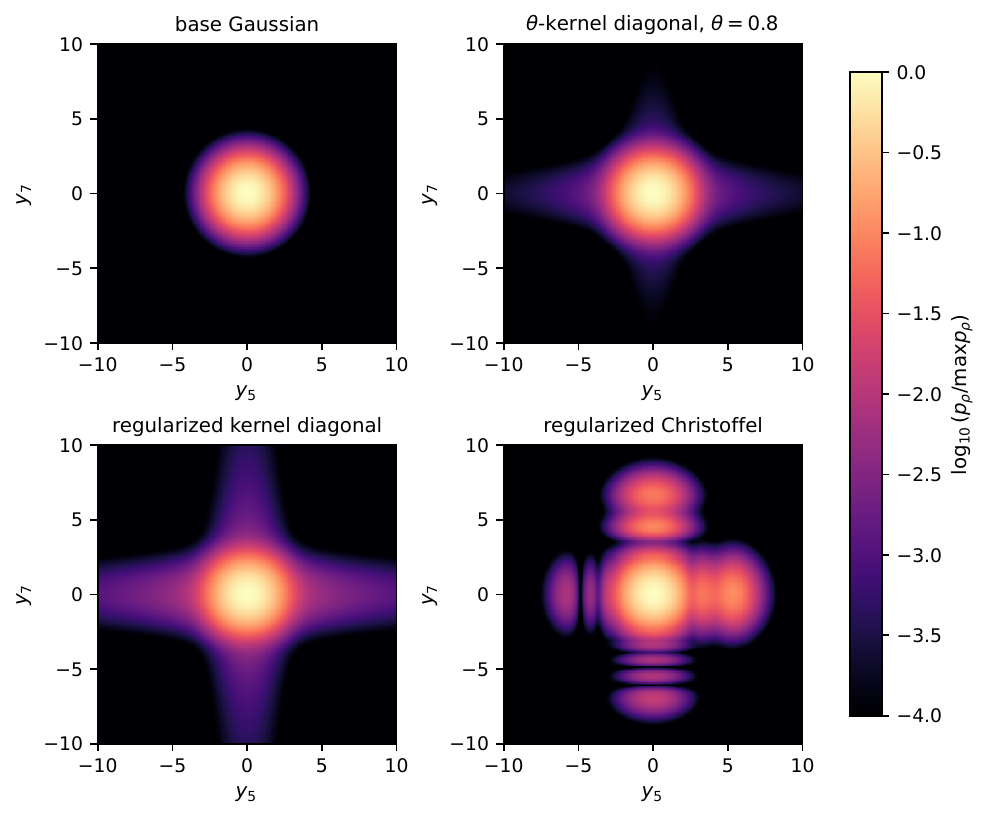}
    \caption{
        Joint $(y_5,y_7)$-marginals at $r=7$ for the base Gaussian, the
        $\theta$-kernel diagonal measure with $\theta=0.8$, the {regularized}
        kernel diagonal measure, and the {regularized} Christoffel measure.
        For each sampling measure $\rho$, let $p_\rho$ denote the Lebesgue
        density of its $(y_5,y_7)$-marginal.
        Colors show
        $\log_{10}\left(p_\rho/\norm{p_\rho}_{L^\infty}\right)$, truncated to
        $[-4,0]$.
        All four panels use the fixed box $[-10,10]^2$.
    }
    \label{fig:numerics-hermite-joint-marginals}
\end{figure}

\subsection{Spatial modes for the lognormal elliptic model}
\label{app:numerics-lognormal-haar-modes}

Let $V^{(s)}$ contain the leading eigenvectors of the reference gradient
covariance $C_{\mathrm{ref}}$ for the model in
Section~\ref{subsec:numerics-lognormal-haar-modes}.
The corresponding reference active spatial modes are
\begin{equation*}
    w_m^{(s)}(x)
    :=
    \sum_{(\ell,k)\in\mathcal I_L}
    V^{(s)}_{(\ell,k),m}\psi_{\ell,k}(x).
\end{equation*}
Figure~\ref{fig:numerics-lognormal-haar-spatial-modes} compares the first four
reference active spatial modes with the first four coordinate Haar functions.
The leading reference active mode remains close to the coarsest Haar pattern
for both values of $s$, whereas the following modes combine several dyadic
levels.
Thus, the reference active directions are not merely a reordering of the
coordinate directions.

\begin{figure}[H]
    \centering
    \includegraphics[width=\textwidth]{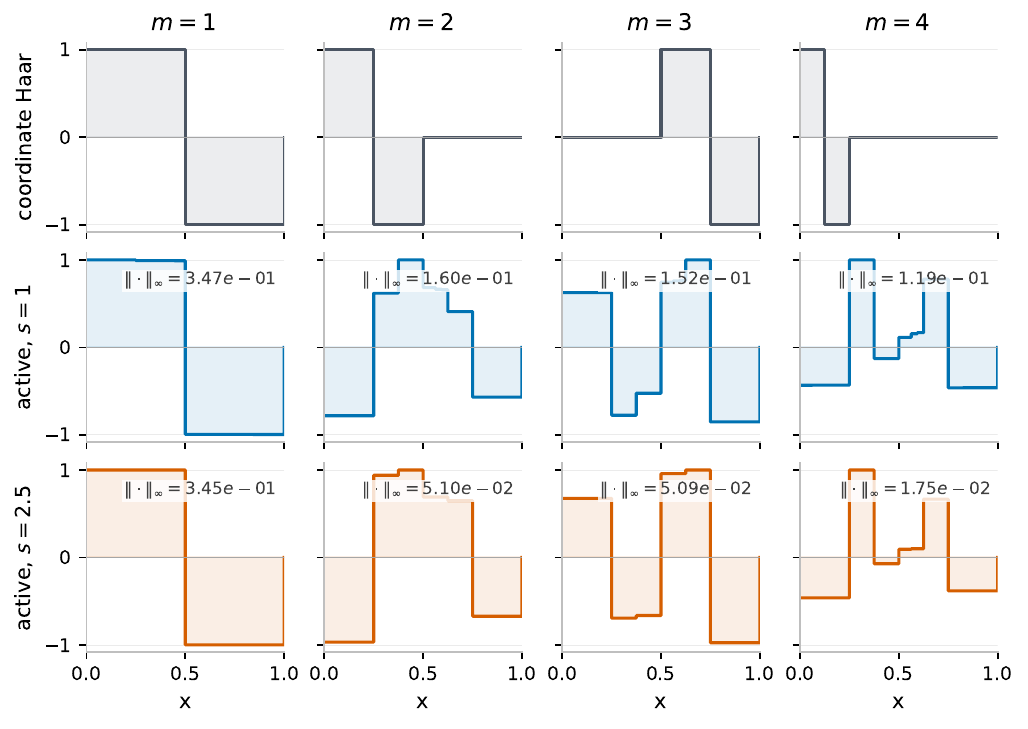}
    \caption{
        Spatial modes for the lognormal Haar benchmark.
        The top row shows the first four coordinate Haar functions, and the
        lower rows show the first four reference active modes for $s=1$ and
        $s=2.5$.
        Each curve is normalized by its spatial supremum norm, and the active-mode
        annotations report the norm before normalization.
    }
    \label{fig:numerics-lognormal-haar-spatial-modes}
\end{figure}

\end{document}